\documentclass[
    10pt,
    a4paper
]{article}
\usepackage{amsmath}
\usepackage{color} 
\usepackage{amssymb}
\usepackage{amsthm}
\usepackage{accents}
\usepackage[utf8]{inputenc}
\usepackage[english]{babel}
\usepackage{bbm}
\usepackage{chngcntr}
\usepackage{enumitem}
\usepackage{listings}
\usepackage{lmodern} 
\usepackage{mathtools}
\usepackage[unicode=true,bookmarks=false,hypertexnames=false]{hyperref}
\usepackage{
graphicx,tikz,tabularx} 
\usepackage{csquotes}
\usepackage[textwidth=16cm,textheight=23.5cm]{geometry}
\usepackage[capitalise,noabbrev,nameinlink,sort]{cleveref}
\let\etoolboxforlistloop\forlistloop 
\usepackage{autonum}
\let\forlistloop\etoolboxforlistloop 
\cslet{blx@noerroretextools}\empty 

\counterwithin{equation}{section}

\DeclarePairedDelimiter\norm{\lVert}{\rVert}%

\DeclareMathOperator*{\esssup}{ess\,sup}

\let\oldc\c 
\renewcommand{\c}[1]{\mathcal{#1}}

\renewcommand{\b}[1]{\mathbb{#1}}

\newcommand{\N}{\mathbb{N}}
\newcommand{\R}{\mathbb{R}}
\renewcommand{\P}{\mathbb{P}}
\DeclareMathOperator{\E}{\mathbb{E}}

\newcommand{\Q}{\mathbb{Q}}

\newcommand{\1}{\mathbbm{1}}

\newtheorem{theorem}{Theorem}[section]
\newtheorem{lemma}[theorem]{Lemma}
\newtheorem{remark}[theorem]{Remark}
\newtheorem{example}[theorem]{Example}
\newtheorem{corollary}[theorem]{Corollary}
\newtheorem{proposition}[theorem]{Proposition}
\newtheorem{definition}[theorem]{Definition}

\crefalias{enumi}{part}
\newtheorem{assumption}{Assumption}
\crefname{assumption}{Assumption}{Assumptions}

\renewcommand{\theassumption}{\Alph{assumption}}
\newlist{assumptionenum}{enumerate}{1} 
\setlist[assumptionenum]{label=(\theassumption\arabic*)}
\crefalias{assumptionenumi}{assumption}

\let\eqref\labelcref
\makeatletter
\autonum@generatePatchedReferenceCSL{eqref}
\makeatother

\begin{document}

\title{A randomisation method for mean-field control problems with common noise}
\author{Robert Denkert%
\thanks{Humboldt University of Berlin, Department of Mathematics, rob at denkert.eu; This author gratefully acknowledges financial support by the Deutsche Forschungsgemeinschaft (DFG, German Research Foundation) -- Project-ID 410208580 -- IRTG2544 (``Stochastic Analysis in Interaction'').}
\quad Idris Kharroubi%
\thanks{
LPSM, UMR CNRS 8001, Sorbonne Universit\'e
  and Universit\'e Paris Cit\'e, idris.kharroubi at sorbonne-universite.fr. Research of the author partially supported by Agence Nationale de la Recherche (ReLISCoP grant ANR-21-CE40-0001).
}
\quad Huy\^en Pham%
\thanks{Ecole Polytechnique, CMAP, huyen.pham at polytechnique.edu; This author  is supported by  the BNP-PAR Chair ``Futures of Quantitative Finance", and by FiME, Laboratoire de Finance des March\'es de l'Energie, and the ``Finance and Sustainable Development'' EDF - CACIB Chair}
}
\maketitle

\begin{abstract}
    We study mean-field control (MFC) problems with common noise using the control randomisation framework, where we substitute the control process with an independent Poisson point process, controlling its intensity instead. To address the challenges posed by the mean-field interactions in this randomisation approach, we reformulate the admissible control as $L^0$-valued processes adapted only to the common noise. We then construct the randomised control problem from this reformulated control process, and show its equivalence to the original MFC problem. Thanks to this equivalence, we can represent the value function as the minimal solution to a backward stochastic differential equation (BSDE) with constrained jumps. Finally, using this probabilistic representation, we derive a randomised dynamic programming principle (DPP) for the value function, expressed as a supremum over equivalent probability measures.
\end{abstract}

\vspace{5mm}


\vspace{3mm}

\noindent     \textbf{ Keywords:}{  Mean-field control with common noise; control randomisation;  decomposition of processes;
    randomised dynamic programming principle; backward stochastic differential equations.



\section{Introduction}

We will consider mean-field control problems with the following state dynamics
\begin{align}
    dX_s^{t,\xi,\alpha}
    &= b(s, \P_{X^{t,\xi,\alpha}_s}^{\c F^{B,\P}_s},X^{t,\xi,\alpha}_s,\alpha_s)ds + \sigma(s,\P_{X^{t,\xi,\alpha}_s}^{\c F^{B,\P}_s},X^{t,\xi,\alpha}_s,\alpha_s) dW_s\\
    &\quad + \sigma^0(s,\P_{X^{t,\xi,\alpha}_s}^{\c F^{B,\P}_s},X^{t,\xi,\alpha}_s,\alpha_s) dB_s,\qquad X^{t,\xi,\alpha}_t = \xi,
\end{align}
with two independent Brownian motions $W,B$ on some probability space $(\Omega,\c F,\P)$ and an independent given starting value $\xi$ and a control process $\alpha$ in some action space $A$.
Here $(\P^{\c F^{B,\P}_s}_{X^{t,\xi,\alpha}_s})_{s\in [t,T]}$ denotes the $\b F^{B,\P}$-predictable and $\P$-a.s.\@ continuous version of $(\c L(X^{t,\xi,\alpha}_s|\c F^{B,\P}_s))_{s\in [t,T]}$.
Our goal is to maximise the reward functional
\begin{align}
    J(t,\xi,\alpha) \coloneqq \E\Big[ g(X^{t,\xi,\alpha}_T, \P_{X^{t,\xi,\alpha}_T}^{{\c F}^{B,\P}_T}) + \int_t^T f(s,X^{t,\xi,\alpha}_s,\P_{X^{t,\xi,\alpha}_s}^{{\c F}^{B,\P}_s},\alpha_s) ds \Big].
\end{align}
There are generally two main approaches to solving mean-field (or McKean-Vlasov) control problems in the literature. The first is the stochastic maximum principle, which characterises the optimal control by an adjoint backward stochastic differential equation (BSDE). In the context of mean-field control problems, the stochastic maximum principle was first studied by \cite{buckdahn_general_2011,andersson_maximum_2011} for cases where the coefficients only depend on the law through its moments. Building on the notion of differentiability introduced by \cite{lions2007theorie} in his course at Coll\`ege de France, \cite{carmona_forwardbackward_2015} extended this approach to general MFC problems; see also the standard work \cite{carmona_probabilistic_1_2018} for a broader overview of the related literature.
The second approach revolves around the dynamic programming principle (DPP), which breaks the original optimisation problem down into several local optimisation problems. Combined with an Itô formula for McKean-Vlasov diffusions, this method leads to a Hamilton-Jacobi-Bellman (HJB) equation on the Wasserstein space that characterises the value function. \cite{pham_bellman_2018} were the first to develop such a DPP in the general setting without common noise, while \cite{pham_dynamic_2017} derived a DPP for MFC problems with common noise, where the admissible controls are solely adapted to the common noise. To our knowledge, the most comprehensive results on mean-field control problems involving common noise can be found in works of \cite{djete_mckeanvlasov_2022,djete_mckeanvlasov_2022-1}.

In this paper, we establish a randomised dynamic programming principle based on a randomised formulation of the control problem. The idea is to replace the control process by an independent Poisson point process, and then optimising over a set of equivalent probability measures, effectively making the intensity of the Poisson point process the new control.
The randomisation method was introduced by \cite{bouchard_stochastic_2009} to deal with optimal switching problems, with further works for optimal switching problems found in \cite{elie_bsde_2014,elie_adding_2014,elie_probabilistic_2010} and with \cite{fuhrman_optimal_2020} addressing the infinite dimensional case. Additionally, it has been applied to impulse control problems in \cite{kharroubi_backward_2010} and to optimal stopping in non-Markovian settings in \cite{fuhrman_representation_2016}. \cite{fuhrman_randomized_2015} applied the randomisation method to non-Markovian regular control problems. 
In addition, \cite{kharroubi_feynmankac_2015} studied the constrained BSDE for the randomised value function across a wide range of control problems, establishing a general connection between a class of constrained BSDEs and HJB equations.
The randomisation framework has also found applications beyond establishing a randomised DPP. Specifically, 
\cite{kharroubi_discrete_2015,kharroubi_numerical_2014} proposed a numerical scheme for fully non-linear HJB equations based on the constrained BSDEs arising from the control randomisation framework, leading to various applications including portfolio optimisation \cite{zhang_dynamic_2019} and algorithmic trading \cite{abergel_algorithmic_2020}.
Additionally, \cite{denkert_control_2024} introduced a policy gradient method for continuous-time reinforcement learning based on the equivalence between the original problem and its randomised formulation.

In a mean-field setting, the randomisation method as a tool has previously only been considered by \cite{bayraktar_randomized_2018} based on a decoupled MKV representation introduced by \cite{buckdahn_mean-field_2017} in the uncontrolled case.
\cite{bayraktar_randomized_2018} provided a randomisation method for mean-field control problems without common noise and used this to obtain a Feynman-Kac formulation and also a randomised dynamic programming principle. 
Their results apply to MFC problems with \emph{deterministic} initial laws of the form $\pi = \delta_x$, $x\in\R^d$. This result does not extend to general initial laws as the following example illustrates.

\begin{example}
We consider the following controlled, decoupled dynamics
\begin{align}
    dX^{\xi,\alpha}_s &= (\alpha_s + \E[X^{\xi,\alpha}_s]) ds,\qquad X^{\xi,\alpha}_0 = \xi \sim \frac 1 2 \delta_{-1} + \frac 1 2 \delta_1,\\
    dX^{\xi,x,\alpha}_s &= (\alpha_s + \E[X^{\xi,\alpha}_s]) ds,\qquad X^{\xi,x,\alpha}_0 = x,
\end{align}
where the control process $\alpha$ should be $A=[-1,1]$-valued, $x\in\R$ and $\xi$ a square-integrable random variable with distribution $\xi\sim \frac 1 2 \delta_{-1} + \frac 1 2 \delta_1$. Further we assume that $T=1$ and that the reward functional is given by
\[
J(\xi,\alpha) = \E\big[(X^{\xi,\alpha}_1 - 1)_+ + (X^{\xi,\alpha}_1 - \tfrac 5 2)_-\big].
\]
Then the value function is given by
\[
V(\xi) = \sup_\alpha J(\xi,\alpha),
\]
and solving for the optimal control, we find $\alpha^*_s \equiv 1$ for all $s\in [0,1]$, which results in $V(\xi) = J(\xi,\alpha^*) 
= \frac 1 2 (e-1)$.

At the same time, the decoupled value function in \cite{bayraktar_randomized_2018} is defined by
\[
V(\xi,x) = \sup_\alpha \E\big[(X^{\xi,x,\alpha}_1 - 1)_+ + (X^{\xi,x,\alpha}_1 - \tfrac 5 2)_-\big].
\]
For the starting value $x=-1$, the optimal control is $\alpha^*_s \equiv -1$ for $s\in [0,1]$, resulting in $V(\xi,-1) = e - \frac 5 2$. For $x=1$, the optimal control is $\alpha^*_s \equiv 1$ for all $s\in [0,1]$, leading to $V(\xi,1) = e-1$.

According to \cite[Proposition 2.2]{bayraktar_randomized_2018}, we would have $V(\xi) = \E_{\xi'\sim\xi}[V(\xi,\xi')] = \frac 1 2 V(\xi,-1) + \frac 1 2 V(\xi,1)$, which is not the case.
\end{example}

In this paper, we improve on the randomisation approach in \cite{bayraktar_randomized_2018} to handle general mean-field control problems and to allow for common noise.
Unlike the previous approach that relies on the decoupled SDEs, we directly work with the mean-field SDE. This approach presents a challenge because we must ensure that the mean-field interaction is preserved during the control randomisation.

Our approach is based on a novel reformulation of the admissible control set as a set of $L^0$-valued processes, adapted solely to the common noise, where we require that at each time point $s\in [t,T]$, the reformulated admissible control 
takes values in $L^0(\Omega,\c G\lor\c F^W_s,\P;A)$. This decomposition of the original 
controls 
is motivated by the observation that for the McKean-Vlasov dynamics, by viewing not each particle but instead the whole current conditional distribution as the state, the idiosyncratic noise will become part of the "deterministic" state dynamics, and only the common noise remains as "stochastic noise" for this system. This viewpoint has the advantage that it leads to a very natural way to apply the control randomisation method, as randomising over such $L^0$-valued processes adapted solely to the common noise keeps the mean-field interaction, in form of the conditional distribution of state process given the common noise, intact.

We establish an isomorphism between the original and the reformulated admissible control set, up to modification of processes. A key challenge lies in the way back from a reformulated control to an original control, as it presents a subtle measurability issue. Specifically, given two filtrations $\b F = (\c F_t)_t$, $\b G = (\c G_t)_t$ and an $\b F$-progressive process $\hat\alpha:[0,T]\times\Omega\to L^0(\Omega',\c G_{T},\P;A)$ satisfying $\hat\alpha_t(\omega) \in L^0(\Omega',\c G_t,\P;A)$ for all $(t,\omega)\in [0,T]\times\Omega$, such as the reformulated control, we need to find an $\b F\otimes\b G$-progressive process $\alpha:[0,T]\times\Omega\times\Omega'\to A$ that is a modification of $\hat\alpha$, meaning that
\[
\alpha_t(\omega,\cdot) = \hat\alpha_t(\omega)\text{ in }L^0(\Omega',\c G_{T},\P;A)\qquad\text{for all }(t,\omega)\in [0,T]\times\Omega,
\]
as detailed in \cref{proposition selecting a measurable process}.

We introduce a randomised control problem based on the reformulated admissible control set, by replacing the reformulated control by an independent Poisson process taking values in $L^0$. By showing that we can again find a corresponding progressive $A$-valued control process, we are able to define the randomised state dynamics. The optimisation takes now place over a set of equivalent probability measures that vary the intensity of the Poisson point process. Crucially, the prior decomposition of the control now ensures that the conditional law of state given the common noise remains unchanged under different randomised controls. This allows us to apply the standard tools from the randomisation framework, although adapted to account for the additional layer of abstraction between the original and reformulated control. Our first main result shows that the randomised value function coincides with the value function of the original problem.

Subsequently we represent the value function as the minimal solution to a backward stochastic differential equation (BSDE) with constrained jumps. This result is obtained by studying a class of penalised BSDEs corresponding to value functions for a restricted set of randomised controls.

Finally, with this probabilistic representation of the value function, we directly derive a randomised dynamic programming principle (DPP), expressed as a supremum over equivalent probability measures. In the benchmark case where the value function is sufficiently regular, we recover the standard (non-randomised) DPP for mean-field control problems, which has been established in the comprehensive works of \cite{djete_mckeanvlasov_2022,djete_mckeanvlasov_2022-1}.

The remainder of this paper is organised as follows. In \cref{section original problem formulation} we formulate the mean-field control problem. In \cref{section problem l2 formulation}, we introduce the reformulated admissible control set and shows its equivalence to the original problem. \cref{section randomised problem} constructs the randomised control problem and gives our first main result \cref{theorem equivalence original and randomised value function} on the equivalence of the non-randomised and the randomised problem. Finally, \cref{section bsde characterisation and randomised dpp} characterises the value function as the minimal solution to a constrained BSDE, enabling us to subsequently derive the randomised DPP.

\paragraph{Notation} Given a random variable $X$, we will denote the (non-augmented) filtration generated by $X$ by $\b F^X$. Further, for a filtration $\b F$, we will denote the $\P$-augmented filtration by $\b F^{\P}$ and the predictable $\sigma$-algebra w.r.t. $\b F$ by $Pred(\b F)$. For a path $w:[0,T]\to A$, we denote its stopped path by $[w]_t \coloneqq (w_{s\land t})_{s\in [0,T]}$.

For some integer $\ell$, we denote by $|\cdot|$, and $\cdot$ the Euclidean norm and the inner (or scalar)  product on $\R^\ell$. 
We denote by $\c{P}(\R^d)$ the space of Borel probability measures on $\R^d$ and by $\c{P}_2(\R^d)$ the set of elements of $\c{P}(\R^d)$ having a finite second order moment. We equip $\c{P}_2(\R^d)$ with the 2-Wasserstein distance $\c{W}_2$ so that it is a metric space itself and even a Polish space. We recall that the 2-Wasserstein distance on $\c{P}_2(\R^d)$ is defined by 
\begin{equation}\label{eq definition w_2 distance}
    \c{W}_2(\mu, \nu) = \Big( \inf_{X\sim\mu, Y\sim\nu} \E[|X-Y|^2] \Big)^{\frac 1 2} 
\end{equation}
for $\mu, \nu \in \c{P}_2(\R^d)$. We also denote by $C^\ell$ the set of continuous functions from $[0,T]$ to $\R^\ell$ and by $\mathbf{C}^\ell=(\c C ^\ell_t)_{t \in [0,T]}$ its canonical filtration. Further we denote the Wiener measure on $C^\ell$ by $\mu^\ell_W$. Finally, given a Polish space $E$, we write $D([0,T];E)$ for the space of $E$-valued càdlàg functions on $[0,T]$ equipped with the usual Skorokhod topology.

\section{Setting}\label{section original problem formulation}

Let $(\Omega,\c F,\P)$ be a complete probability space supporting an $m$-dimensional Brownian motion $W$ and an independent $n$-dimensional Brownian motion $B$. Further we suppose that there exists a separable 
$\sigma$-algebra $\c G\subseteq \c F$ independent of $\b F^{W,B}$ sufficiently large such that $\{\P_\xi \,|\, \xi \in L^2(\Omega,\c G,\P;\R^d) \} = \c P_2(\R^d)$.
We define the set of admissible controls $\c A$ as all $\b F^{W,B}\lor \c G$-predictable%
\footnote{We could also allow all $(\b F^{W,B,\P}\lor\c G)$-predictable processes, since by \cite[Lemma 2.17(b)]{jacod_limit_2003},
every $(\b F^{W,B,\P}\lor\c G)$-predictable process is indistinguishable from an $(\b F^{W,B}\lor \c G)$-predictable process. We choose $(\b F^{W,B}\lor\c G)$-predictable processes only to slightly simplify the presentation in \cref{section problem l2 formulation}.
}
processes taking values in some Polish space $A$ equipped with some metric $\rho < 1$ (otherwise we replace it by $\frac{\rho}{1+\rho}$). Let $b,\sigma,\sigma^0:~[0,T]\times \R^d\times \c{P}_2(\R^d)\times {A}\rightarrow \R^d,\R^{d\times m},\R^{d\times n}$ be three measurable functions that we suppose to satisfy the following assumption.
\begin{assumption}\label{assumptions sde coefficients}
\begin{assumptionenum}
\item The functions $b,\sigma$ and $\sigma^0$ are continuous.
\item There exists a constant $L$ such that
\begin{align}
&|b(t,x,\mu,a)-b(t,y,\nu,a)|+|\sigma(t,x,\mu,a)-\sigma(t,y,\nu,a)|+|\sigma^0(t,x,\mu,a)-\sigma^0(t,y,\nu,a)|\\
&\leq L\Big(|x-y|+\c{W}_2(\mu,\nu)\Big),
\end{align}
for all $x,y\in\R^d$, $\mu,\nu\in \c{P}_2(\R^d)$, $t\in[0,T]$ and $a\in A$.
\item There exists a constant $M$ such that
\begin{align}
|b(t,0,\delta_0,a)|+|\sigma(t,0,\delta_0,a)|+|\sigma^0(t,0,\delta_0,a)| & \leq M,
\end{align}
for all $t\in[0,T]$ and $a\in A$.
\end{assumptionenum}
\end{assumption}
Given an initial time $t\in[0,T]$, an initial condition $\xi\in L^2(\Omega,\c G\lor \c F^W_t,\P;\R^d)$ and an admissible control $\alpha\in \c A$, we consider the $\b F^{W,B,\P}\lor\c G$-progressive state process $(X^{t,\xi,\alpha}_s)_{s\in [t,T]}$   defined by
\begin{align}
    X_r^{t,\xi,\alpha}
    &= \xi+\int_t^rb(s, X^{t,\xi,\alpha}_s,\P_{X^{t,\xi,\alpha}_s}^{\c F^{B,\P}_s},\alpha_s)ds + \int_t^r \sigma(s,X^{t,\xi,\alpha}_s,\P_{X^{t,\xi,\alpha}_s}^{\c F^{B,\P}_s},\alpha_s) dW_s\\
    &\quad +\int_t^r \sigma^0(s,X^{t,\xi,\alpha}_s,\P_{X^{t,\xi,\alpha}_s}^{\c F^{B,\P}_s},\alpha_s) dB_s,\qquad r\in[t,T]\;.
    \label{eq mfc state dynamics}
\end{align}
Here we denote by $(\P^{\c F^{B,\P}_s}_{X^{t,\xi,\alpha}_s})_{s\in [t,T]}$ the $\b F^{B,\P}$-predictable and $\P$-a.s.\@ continuous version of $(\c L(X^{t,\xi,\alpha}_s|\c F^{B,\P}_s))_{s\in [t,T]}$, that is $\P^{\c F^{B,\P}_s}_{X^{t,\xi,\alpha}_s} = \c L(X^{t,\xi,\alpha}_s|\c F^{B,\P}_s)$, $\P$-a.s., for all $s\in [t,T]$. The existence of such a version is ensured by \cite[Lemma A.1]{djete_mckeanvlasov_2022-1}, since the sets of $\b F^{B,\P}$-optional and $\b F^{B,\P}$-predictable processes coincide by \cite[Chapter V, Corollary 3.3]{revuz_continuous_1999}.
We note that under \cref{assumptions sde coefficients}, we have existence and uniqueness of $(X^{t,\xi,\alpha}_s)_{s\in [t,T]}$ and
\begin{align}
\sup_{\alpha\in \c A}\E\Big[\sup_{s\in[t,T]}|X^{t,\xi,\alpha}_s|^2\Big] & < \infty\;.
\end{align} 

Let $f:~[0,T]\times \R^d\times \c{P}_2(\R^d)\times {A}\rightarrow \R$ and $g:~ \R^d\times \c{P}_2(\R^d)\rightarrow \R$ be two Borel-measurable functions on which we make the following assumption.
\begin{assumption}\label{assumptions reward functional} There exists a constant $M>0$ such that
\begin{align}
|f(t,x,\mu,a)|+|g(x,\mu)| &\leq M\Big( 1+|x|+ \Big(\int_{\R^d}|y|^2\mu(dy)\Big)^{\frac 1 2} \Big)
\end{align}
for all $(t,x,\mu,a)\in[0,T]\times \R^d\times \c{P}_2(\R^d)\times A$.
\end{assumption}
\begin{remark}
In certain results, such as the equivalence between the original and the randomised problem in \cref{theorem equivalence original and randomised value function}, the linear growth assumption on $f$ and $g$ can be relax to quadratic growth. Nevertheless, for the law invariance result of the value function by \cite{djete_mckeanvlasov_2022-1}, and to establish the (randomised) DPP in \cref{section randomised dpp} and the BSDE characterisation of the value function in \cref{section bsde characterisation}, we need the more restrictive linear growth assumption.
\end{remark}
We next define the function $J$ by
\begin{align}
    J(t,\xi,\alpha) \coloneqq \E\Big[ g(X^{t,\xi,\alpha}_T, \P_{X^{t,\xi,\alpha}_T}^{{\c F}^{B,\P}_T}) + \int_t^T f(s,X^{t,\xi,\alpha}_s,\P_{X^{t,\xi,\alpha}_s}^{{\c F}^{B,\P}_s},\alpha_s) ds \Big]
\end{align}
for $t\in[0,T]$, $\xi\in L^2(\Omega,\c G\lor \c F^W_t,\P;\R^d)$ and $\alpha\in \c A$. Under \cref{assumptions sde coefficients,assumptions reward functional}, the function $J$ is well defined.  We consider the optimal value of $J$ 
over the set of admissible controls $\c A$. Consequently, we define the value function as
\[
V(t,\xi) \coloneqq \sup_{\alpha\in \c A} J(t,\xi,\alpha) < \infty.
\]

We note that \cite{djete_mckeanvlasov_2022,djete_mckeanvlasov_2022-1} show that the value function does not depend on the choice of the probability space, and that it depends on $\xi$ only through its law $\P_\xi$.%
\footnote{They furthermore show that the strong and the weak formulation of the problem both lead to the same value function.}

\begin{proposition}[{\cite[Proposition 2.4]{djete_mckeanvlasov_2022-1}}]
\label{proposition value function law invariant and independent of probabilistic setting}
The value function is \emph{law invariant}, that is, it depends on $\xi$ only through its law $\P_\xi$. 
Furthermore, the value function does not depend on the specific choice of $(\Omega,\c F,\P)$ and $\c G$, as long as $\c G$ is rich enough.
\end{proposition}

Therefore, with slight abuse of notation, we will also write
\[
V(t,\P_\xi) = V(t,\xi).
\]

\section{Reformulating the set of admissible controls}\label{section problem l2 formulation}

In this section, we derive an equivalent reformulation of the original mean-field control problem, that will form the basis for developing the randomised formulation in \cref{section randomised problem}. The key idea behind this reformulation is that in McKean-Vlasov dynamics, by treating the entire current distribution as the state space rather than focusing on individual particles, the idiosyncratic noise becomes part of the deterministic state dynamics, leaving the common noise as the only "true noise" affecting the system. This leads us to reframe the control set $\c A$ of $\b F^{W,B}\lor\c G$-predictable processes by connecting it to a suitable set $\hat{\c A}$ of only $\b F^B$-predictable processes. This approach aligns the control problem thus aligning the control set more closely with the traditional (non-mean-field) framework, where the control is adapted to the external noise.

In this spirit, let us consider the sets
\[
\c A_s \coloneq L^0(\Omega,\c G\lor\c F^W_s,\P;A),\qquad s\in [0,T],
\]
of $\P$-equivalence classes of $\c G\lor\c F^W_s$-$\c B(A)$-measurable functions $\phi:\Omega\to A$ equipped with the metric
\[
\hat\rho(\phi^1,\phi^2) \coloneqq \E^\P[\rho(\phi^1,\phi^2)] 
< 1
,\qquad\text{for all } \phi^1,\phi^2\in \c A_s.
\]
They will now take the role of the action space, each at its corresponding time $s\in [0,T]$. We remark that $\c A_u\subseteq\c A_v$, for $u\leq v$. Further, since $A$ is a Polish space and $\c G$ and $\c F^W_s$ are separable, the spaces $\c A_s$ are also Polish, $s\in [0,T]$.
Finally, we define the set $\hat{\c A}$ of all $\b F^B$-predictable processes $\hat\alpha:[0,T]\times\Omega\to \c A_T$, such that additionally $\hat\alpha_s\in\c A_s$ for $s\in [0,T]$. 

We will spend the rest of this section on the relation between the sets $\c A$ and $\hat{\c A}$. More precisely, in \cref{subsection connecting A and L0 A processes} we will introduce in which sense we can identify such process $\hat\alpha:[0,T]\times\Omega\to \c A_T$ with the original control processes of the form $\alpha:[0,T]\times\Omega\to A$ in a slightly more general setting, which we will need in \cref{section randomised problem}. Afterwards, in \cref{subsection equivalent action set}, we will use this concept to establish an equivalence between the action sets $\c A$ and $\hat{\c A}$, which will aid in the upcoming construction of the randomised control problem in \cref{section randomised problem}.

\subsection{Identifying decomposed controls with original controls}\label{subsection connecting A and L0 A processes}

Since in \cref{section randomised problem} we are going to extend our probability space for the randomised formulation to allow for an additional Poisson point process $\mu$, we will in this section consider more generally processes defined on an extension $(\hat\Omega,\hat{\c F},\hat\P)$ of our original probability space $(\Omega,\c F,\P)$. This means that $\hat \Omega = \Omega\times \tilde \Omega$, $\tilde \Omega \times \c F\subseteq \hat{\c F} $ and $\hat \P(\tilde\Omega\times \cdot)=\P$. We denote by $\pi:\hat\Omega\to\Omega$ the canonical projection, which by definition is probability preserving. We next identify $\sigma$-algebras on $\Omega$ as $\sigma$-algebras on $\hat\Omega$ by replacing them with their reciprocal images by $\pi$. We also extend the definition of random variables on $\Omega$ to $\hat \Omega$  by composing by $\pi$.

For a suitably general class of processes $\hat\alpha:[0,T]\times\hat\Omega\to \c A_T$ and $\alpha:[0,T]\times\hat\Omega\to A$, we will define the identification relation as follows.

\begin{definition}
Let $(\hat\Omega,\hat{\c F},\hat\P)$ be an extension of $(\Omega,\c F,\P)$ and $\b F$ a filtration on the extension independent of $\c G,W$. Then we say that an $\b F^W\lor\c G\lor \b F$-progressive process $\alpha:[0,T]\times\hat\Omega\to A$ $\b F$-identifies to a $\b F$-progressive process $\hat\alpha:[0,T]\times\hat\Omega\to \c A_T = L^0(\Omega,\c G\lor\c F^W_T,\P;A)$ if there exists an $(\c G\lor\b F^W)\otimes\b F$-progressive process $\Phi:[0,T]\times\Omega\times\hat\Omega\to A$ such that
\[
\alpha_s(\hat\omega) = \Phi_s(\pi(\hat\omega),\hat\omega),\qquad\text{for all }(s,\hat\omega)\in [0,T]\times\hat\Omega,
\]
and, viewed as equality in $L^0(\Omega,\c G\lor\c F^W_T,\P;A)$,
\[
\hat\alpha_s(\hat\omega) = \Phi_s(\cdot,\hat\omega),\qquad\text{for all }(s,\hat\omega)\in [0,T]\times\hat\Omega.
\]
\end{definition}

We remark that just by the definition it is not clear whether given an $\alpha$ also an $\hat\alpha$ identifying with $\alpha$ exists (and vice versa). However, we can show that, if it then exists, $\hat\alpha$ given $\alpha$ (and similarly $\alpha$ given $\hat\alpha$) is unique up to modification, as the following \cref{lemma alpha bar alpha uniqueness} shows.

\begin{lemma}\label{lemma alpha bar alpha uniqueness}
    Let $(\hat\Omega,\hat{\c F},\hat\P)$ be an extension of $(\Omega,\c F,\P)$ and for $i\in\{1,2\}$ let $\alpha^i$ be a $\b F^W\lor\c G\lor\b F^i$-progressive process $\b F^i$-identifying to a $\b F^i$-progressive process $\hat\alpha^i$. Then for all $s\in [0,T]$,
    \[
    \E^{\hat\P}[\rho(\alpha^1_s,\alpha^2_s)] = \E^{\hat\P}[\hat\rho(\hat\alpha^1_s,\hat\alpha^2_s)].
    \]
    In particular, $\alpha^1$ is a modification of $\alpha^2$ iff $\hat\alpha^1$ is a modification of $\hat\alpha^2$.
\end{lemma}
\begin{proof}
    For $i\in\{1,2\}$, by definition there exists $(\c G\lor\b F^W)\otimes\b F^i$-progressive processes $\Phi^i:[0,T]\times\Omega\times\hat\Omega\to A$ such that
    \[
    \alpha^i_s(\hat\omega) = \Phi^i_s(\pi(\hat\omega),\hat\omega),\qquad\text{for all }(s,\hat\omega)\in [0,T]\times\hat\Omega,
    \]
    and, viewed as equality in $L^0(\Omega,\c G\lor\c F^W_T,\P;A)$,
    \[
    \hat\alpha^i_s(\hat\omega) = \Phi^i_s(\cdot,\hat\omega),\qquad\text{for all }(s,\hat\omega)\in [0,T]\times\hat\Omega.
    \]
    Defining the filtration $\b F\coloneqq \b F^1\lor\b F^2$, we note that $\b F$ is also independent of $\c G,W$ and $\Phi^1,\Phi^2$ are also $(\c G\lor\b F^W)\otimes\b F$-progressive. Thus, for all $s\in [0,T]$,
    \begin{align}
    \E^{\hat\P}[\rho(\alpha^1_s,\alpha^2_s)] &= \int_{\hat\Omega}  \rho(\Phi^1_s(\pi(\hat\omega),\hat\omega),\Phi^2_s(\pi(\hat\omega),\hat\omega)) \hat\P(d\hat\omega)\\
    &= \int_{\hat\Omega} \int_{\hat\Omega} \rho(\Phi^1_s(\pi(\hat\omega^1),\hat\omega^2),\Phi^2_s(\pi(\hat\omega^1),\hat\omega^2)) \hat\P(d\hat\omega^1)\hat\P(d\hat\omega^2)
    = \E^{\hat\P}[\hat\rho(\hat\alpha^1_s,\hat\alpha^2_s)],
    \end{align}
    where the second equality comes from the independence of $\c G$, $W$ and $\b F$.
\end{proof}

Moreover, in some cases, e.g. if the filtration $\b F$ is generated by a càdlàg process, we can even show that uniqueness in joint law holds for this the identification, see the following \cref{lemma joint law N bar N map}, which will be useful in the case of Poisson processes in \cref{section randomised problem}. 

\begin{lemma}\label{lemma joint law N bar N map}
    Let $(\hat\Omega^i,\hat{\c F}^i,\hat\P^i)$, $i=\{1,2\}$ be two extensions of $(\Omega,\c F,\P)$. Denotes by $\c G^i,W^i$ be the canonical extensions of $\c G,W$, and by $\pi^i:\hat\Omega^i\to\Omega$ be the canonical projection, $i=\{1,2\}$. Further let on each $(\hat\Omega^i,\hat{\c F}^i,\hat\P^i)$  be a càdlàg,
    $\b F^{N^i}$-progressive process $N^i$ independent of $\c G^i$ and $W^i$ taking values in a Polish space $E$, an $\b F^{W^i,N^i}\lor\c G^i$-progressive map $\alpha^i:[0,T]\times\hat\Omega^i\to A$ and an $\b F^{N^i}$-progressive $\hat\alpha^i:[0,T]\times\hat\Omega^i\to\c A_T$ which $\b F^{N^i}$-identifies to $\alpha^i$,  $i=\{1,2\}$. Then the following statements (of equality in finite dimensional distributions) are equivalent
    \begin{enumerate}[label=(\roman*)]
        \item\label{lemma 4.2 assumption i} $\c L^{\hat\P^1}(\pi^1,N^1,\alpha^1) = \c L^{\hat\P^2}(\pi^2,N^2,\alpha^2)$,
        \item\label{lemma 4.2 assumption ii} $\c L^{\hat\P^1}(\pi^1,N^1,\hat\alpha^1) = \c L^{\hat\P^2}(\pi^2,N^2,\hat\alpha^2)$,
        \item\label{lemma 4.2 assumption iii} $\c L^{\hat\P^1}(\pi^1,N^1,\alpha^1,\hat\alpha^1) = \c L^{\hat\P^2}(\pi^2,N^2,\alpha^2,\hat\alpha^2)$.
    \end{enumerate}
\end{lemma}
\begin{proof}
By definition there exists $(\c G\lor\b F^W)\otimes\b F^{N^i}$-progressive processes $\Phi^i:[0,T]\times\Omega\times\hat\Omega^i\to A$ such that
\[
\alpha^i_s(\hat\omega^i) = \Phi^i_s(\pi^i(\hat\omega^i),\hat\omega^i),\qquad\text{for all }(s,\hat\omega^i)\in [0,T]\times\hat\Omega^i,
\]
and, viewed as equality in $L^0(\Omega,\c G\lor\c F^W_T,\P;A)$,
\[
\hat\alpha^i_s(\hat\omega^i) = \Phi^i_s(\cdot,\hat\omega^i),\qquad\text{for all }(s,\hat\omega^i)\in [0,T]\times\hat\Omega^i.
\]
Then the Doob-Dynkin Lemma implies that there exists a $\c B([0,T])\otimes(\c G\lor\c F^W_T)\otimes\c B(D([0,T];E))$-measurable $\Psi^i:[0,T]\times\Omega\times D([0,T];E)\to A$ such that
\[
\Phi^i_s(\omega,\hat\omega^i) = \Psi^i_s(\omega,N^i(\hat\omega^i)),\qquad\text{for all }(s,\omega,\hat\omega^i)\in[0,T]\times\Omega\times\hat\Omega^i.
\]
Now we define ${\alpha^2}':[0,T]\times\hat\Omega^2\to A$ and ${\hat\alpha^2}':[0,T]\times\hat\Omega^2\to\c A_T$ via
\[
{\alpha^2_s}'(\hat\omega^2) \coloneqq \Psi^1_s(\pi^2(\hat\omega^2),N^2(\hat\omega^2)),\quad\text{for all }(s,\hat\omega^2)\in [0,T]\times\hat\Omega^2,
\]
and
\[
{\hat\alpha^2_s}'(\hat\omega^2) = \Psi^1_s(\cdot,N^2(\hat\omega^2)),\qquad\text{for all }(s,\hat\omega^2)\in [0,T]\times\hat\Omega^2.
\]
By construction, ${\hat\alpha^2}'$ is $\c B([0,T])\otimes\c F^{N^2}_T$-$\c B(A)$-measurable and ${\alpha^2}'$ is $\c B([0,T]\otimes(\c F^{W^2,N^2}_T\lor\c G)$-measurable.

Now let us assume that \labelcref{lemma 4.2 assumption i} holds, which means that $\c L^{\hat\P^1}(\pi^1,N^1,\alpha^1) = \c L^{\hat\P^2}(\pi^2,N^2,\alpha^2)$. Then by construction
\begin{align}
\c L^{\hat\P^1}(\pi^1,N^1,\alpha^1,\alpha^1,\hat\alpha^1) = \c L^{\hat\P^2}(\pi^2,N^2,\alpha^2,{\alpha^2}',{\hat\alpha^2}'),
\label{eq lemma 4.2 joint law equality}
\end{align}
which shows that $\alpha^2_s = {\alpha^2_s}'$, $\hat\P^2$-a.s., for all $s\in [0,T]$, which means that $\alpha^2$ is a modification of ${\alpha^2_s}'$. Therefore, using that $\Psi^2(\pi^2(\cdot),N^2(\cdot))$ is $\c B([0,T])\otimes (\c G^2\lor\c F^{W^2}_T)\otimes \c F^{N^2}_T$-measurable and that $\c G^2\lor\c F^{W^2}_T$ and $N^2$ are independent, we obtain
\begin{align}
\E^{\hat\P^2}[\rho(\alpha^2_s,{\alpha^2_s}')] = \E^{\hat\P^2}[\rho(\Psi^2_s(\pi^2,N^2),\Psi^1_s(\pi^2,N^2))] = \E^{\hat\P^2}[\hat\rho(\hat\alpha^2_s,{\hat\alpha^2_s}')] = 0,
\end{align}
and hence $\hat\alpha^2_s = {\hat\alpha^2_s}'$, $\hat\P^2$-a.s., for all $s\in [0,T]$, which together with \eqref{eq lemma 4.2 joint law equality} implies \labelcref{lemma 4.2 assumption iii}. The direction \labelcref{lemma 4.2 assumption ii} $\Rightarrow$ \labelcref{lemma 4.2 assumption iii} can be proven similarly.
\end{proof}

\begin{remark}
    It is not strictly necessary for there to be a càdlàg process $N$ generating the filtration $\b F=\b F^N$. The result \cref{lemma joint law N bar N map} holds more generally as long as the filtration $\b F$ is countably generated. However, in this paper, we focus on filtrations $\b F=\b F^N$ generated by càdlàg, piece-wise constant processes $N$, such as Poisson processes, as these will be relevant to the randomisation framework outlined in \cref{section randomised problem}.
\end{remark}

\subsection{Equivalence of the action sets $\c A$ and $\hat{\c A}$}\label{subsection equivalent action set}

In this section we will show that for each $\alpha\in\c A$ there exists an $\hat\alpha\in\hat{\c A}$ such that $\alpha$ identifies with $\hat\alpha$ and vice versa. We equip $\c A$ and $\hat{\c A}$ with the following pseudometrics
\[
d_{\c A}^{\hat\P}(\alpha^1,\alpha^2) = \E^{\hat\P}\Big[\int_0^T \rho(\alpha^1_s,\alpha^2_s)ds \Big],\qquad d_{\hat{\c A}}^{\hat\P}(\hat\alpha^1,\hat\alpha^2) \coloneqq \E^{\hat\P}\Big[\int_0^T \hat\rho(\hat\alpha^1_s,\hat\alpha^2_s) ds\Big].
\]
By \cref{lemma alpha bar alpha uniqueness}, this will then define an isomorphism between the induced quotient spaces $\c A/_\sim$ and $\hat{\c A}/_\sim$, where $\alpha^1\sim\alpha^1$ (resp. $\hat\alpha^1\sim\hat\alpha^2$) iff $d_{\c A}^{\hat\P}(\alpha^1,\alpha^2) = 0$ (resp. $d_{\hat{\c A}}^{\hat\P}(\hat\alpha^1,\hat\alpha^2) = 0$).

In particular, this justifies considering the optimal control problem over the action set $\hat{\c A}$ rather than $\c A$, which will be crucial for \cref{section randomised problem}, as it provides a more natural starting point for the construction of the randomised control problem.

\subsubsection{From $\c A$ to $\hat{\c A}$}
For the construction of an $\hat\alpha\in\hat{\c A}$, which $\b F^B$-identifies to a given $\alpha\in\c A$, we first prove the following general results about decomposing predictable processes.

\begin{proposition}\label{proposition decompose filtrations of predictable processes}
    Let $(\Omega,\c F)$ be a measurable space with two filtrations $\b F$ and $\b G$ and let $A$ be a Polish space.
    Then a process $X:[0,T]\times\Omega\to A$ is $\b F\lor\b G$-predictable if and only if there exists a $\b F\otimes\b G$-predictable process $Y:[0,T]\times\Omega\times\Omega\to A$ defined on the space $(\Omega\times \Omega,\c F\otimes \c F)$ such that
    \[
    X_t(\omega) = Y(t,\omega,\omega),\qquad\text{for all }(t,\omega)\in [0,T]\times\Omega.
    \]
\end{proposition}
\begin{proof}
    Given $Y$, the $\b F\lor\b G$-predictability of $X$ is obvious. So let us focus on the other direction.
    We assume that the Polish space $A$ is isomorphic $\R$, and thus w.l.o.g. assume that $A = \R$.
    
    We will prove this result using the functional monotone class theorem. To this end, we denote by $\c H$ the space of $\b F\lor\b G$-predictable functions $X:[0,T]\times\Omega\to\R$ such that a corresponding $Y^X$ exists. We note that if $X^1,X^2\in\c H$ and $c\in\R$, then also $X^1 X^2,X^1+c X^2\in\c H$ since we can choose $Y^{X^1 X^2} \coloneqq Y^{X^1}Y^{X^2}$ and $Y^{X^1+c X^2} \coloneqq Y^{X^1}+c Y^{X^2}$. Furthermore, if $(X^n)_n\subseteq \c H$ and $0\leq X^n\uparrow X$ to a bounded process $X$, then we can choose $Y^X\coloneqq \limsup_n Y^{X^n}$ to conclude that also $X\in\c H$.

    Thus, recalling the generator of $Pred(\b F\lor\b G)$, we only need to show that (i) $\1_{(t_1,t_2]\times B\cap C}\in\c H$ for all $0\leq t_1<t_2\leq T$ and $B\in \c F_{t_1},C\in\c G_{t_1}$ (ii) $\1_{\{0\}\times B\cap C}\in\c H$ for all $B\in \c F_{0-},C\in\c G_{0-}$. By defining $Y^{\1_{(t_1,t_2]\times B\cap C}}\coloneqq \1_{(t_1,t_2]\times B\times C}$ (respectively $Y^{\1_{\{0\}\times B\cap C}}\coloneqq \1_{\{0\}\times B\times C}$) this is clear and the result follows.
\end{proof}

\begin{proposition}\label{proposition l0 control is predictable}
Let $(\Omega,\c F,\b F)$ a filtered space, $(\Omega',\c G,\P)$ a probability space and $X:[0,T]\times\Omega\times\Omega'\to A$ an $\b F\otimes\c G$-predictable process taking values in some Polish space $A$. Then the process $Y:[0,T]\times\Omega\to L^0(\Omega',\c G,\P;A)$, where $L^0$ is meant w.r.t. $\P$-equivalence classes, defined by
\[
Y_t(\omega) \coloneqq X_t(\omega,\cdot) \in L^0(\Omega',\c G,\P;A),\qquad (t,\omega)\in [0,T]\times\Omega,
\]
is $\b F$-predictable.
\end{proposition}

\begin{proof}
Note that it suffices to show the claim for $A=\R$ since every Polish space is isomorphic to a Borel subset of $\R$. Therefore we will only consider the case $A=\R$, and carry out the proof using the functional monotone class theorem. To this end, let $\c H$ be the set of such processes $X$ for which $Y^X$ is indeed $\b F$-predictable. Note that if $X\equiv 1$, then $Y^X\equiv 1$ is trivially $\b F$-predictable. Similarly, if $X^1,X^2\in \c H$ and $c\in\R$, then also $X^1 X^2, X^1+c X^2\in\c H$ since $Y^{X^1 X^2} = Y^{X^1} Y^{X^2}$ and $Y^{X^1 + c X^2} = Y^{X^1} + c Y^{X^2}$. Furthermore, if $(X^n)_n\subseteq \c H$ and $0\leq X^n\uparrow X$ to a bounded process $X$, then $Y^X = \lim_n Y^{X^n}$ is also $\b F$-predictable, and thus $X\in\c H$.

Finally, we need to verify that $\c H$ contains a generating set for $Pred(\b F)$. For this we note that both (i) $\1_{(t_1,t_2]\times B\times C}$ for $0 \leq t_1 < t_2\leq T$ and $B\in \c F_{t_1},C\in\c G$ and (ii) $\1_{\{0\}\times B\times C}$ for $B\in\c F_{0-},C\in\c G$ are in $\c H$ since the corresponding processes $Y$ are clearly $\b F$-predictable and thus $\c H$ indeed contains all $\b F\otimes\c H$-predictable processes $X$.
\end{proof}

Putting these results together, we can now prove the following \cref{proposition bar alpha given alpha well defined}.

\begin{proposition}\label{proposition bar alpha given alpha well defined}
    For every $\alpha\in\c A$, there exists an $\hat\alpha\in\hat{\c A}$ that $\b F^B$-identifies to $\hat\alpha$.
\end{proposition}
\begin{proof}
We recall that $\alpha\in\c A$ is $\b F^{W,B}\lor\c G$-predictable, and thus by \cref{proposition decompose filtrations of predictable processes} there exists a $(\c G\lor\b F^W)\otimes\b F^B$-predictable process $\Psi:[0,T]\times\Omega\times\Omega \to A$ such that
\[
\Psi_s(\omega,\omega) = \alpha_s(\omega),\qquad\text{for all }(s,\omega)\in [0,T]\times\Omega.
\]
Next we note that we can canonically extend $\Psi$ to $\Phi:[0,T]\times\Omega\times\hat\Omega\to A$, that means it satisfies $\Phi = \Psi(\cdot,\pi(\cdot))$. 
In particular $\Phi:[0,T]\times\Omega\times\hat\Omega\to A$ is $(\c G\lor\b F^W)\otimes\b F^B$-predictable, and thus by \cref{proposition l0 control is predictable}, the process $\hat\alpha:[0,T]\times\hat\Omega\to L^0(\Omega,\c G\lor\c F^W_T,\P;A) = \c A_T$ defined by
\[
\hat\alpha_s(\hat\omega) \coloneqq \Phi_s(\cdot,\hat\omega),
\]
is $\b F^B$-predictable and satisfies $\hat\alpha_s(\hat\omega)\in L^0(\Omega,\c G\lor\c F^W_s,\P;A) = \c A_s$ for all $s\in [0,T]$, $\hat\omega\in\hat\Omega$. Therefore $\hat\alpha\in\hat{\c A}$ and $\alpha$ is $\b F^B$-identifying to $\hat\alpha$.
\end{proof}

\subsubsection{From $\hat{\c A}$ to $\c A$}

Our goal in this section is constructing an $\alpha\in\c A$ which $\b F^{\hat\alpha}$-identifies to a given $\hat\alpha\in\hat{\c A}$. For this, we will need several auxiliary result. The first key tool is the following measurable selection result.

\begin{proposition}\label{proposition selecting a measurable process}
    Let $(\Omega,\c F,\b F)$ be a filtered space and $(\Omega',\c G,\b G,\P)$ be a filtered probability space, such that $\c G$ is separable. Let $A$ be a Polish space and $X:[0,T]\times\Omega\to L^0(\Omega',\c G,\P;A)$ be a $(\c B([0,T])\otimes\c F)$-$\c B(L^0(\Omega',\c G,\P;A))$-measurable and $\b F$-adapted process. 
    Further assume that $X_t(\omega) \in L^0(\Omega',\c G_t,\P;A)\subseteq L^0(\Omega',\c G,\P;A)$ for all $(t,\omega)\in [0,T]\times\Omega$.

    Then there exists a $(\c B([0,T])\otimes\c F\otimes\c G)$-$\c B(A)$-measurable and $\b F\otimes\b G$-adapted process $Y:[0,T]\times\Omega\times\Omega'\to A$ such that
    \[
    Y_t(\omega,\cdot) = X_t(\omega)\text{ in }L^0(\Omega',\c G,\P;A),\qquad\text{for all }(t,\omega)\in [0,T]\times\Omega.
    \]
    Further, if $X$ is $\b F$-progressive, then $Y$ can also be chosen $\b F\otimes\b G$-progressive.
\end{proposition}

\begin{proof}
\begin{enumerate}[wide,label=(\roman*)]
    \item \label{proof proposition selecting a measurable process case A unit interval}
Let us first consider the case that $A = [0,1]$. 
The following construction and proof is based on \cite[Theorem IV.30]{dellacherie_probabilities_1978} and its detailed version in \cite{kaden_progressive_2004}. Since $[0,1]$ is bounded and thus the notions of convergence in probability and convergence in $L^1$ coincide (and therefore induce the same topology), we will in the following consider $L^1$ instead of $L^0$.
Let $X:[0,T]\times\Omega\to L^1(\Omega',\c G,\P;[0,1])$ be a $(\c B([0,T])\otimes\c F)$-$\c B(L^1(\Omega',\c G,\P;[0,1]))$-measurable and $\b F$-adapted process. Then, since $\c G$ and thus also $L^1(\Omega',\c G,\P;[0,1])$ is separable, we can approximate $X$ uniformly by countable sums of simple predictable processes as follows:

Let $n\in\N$, and let $(A^n_k)_k\subset\c B(L^1(\Omega',\c G,\P;[0,1]))$ be a Borel sets of diameter $d^n_k \coloneqq \sup_{x,y\in A^n_k} \norm{x-y}_{L^1(\Omega',\c G,\P;[0,1])} \leq 2^{-n}$ covering $L^1(\Omega',\c G,\P;[0,1])$, that is $\bigcup_k A^n_k = L^1(\Omega',\c G,\P;[0,1])$. Further let $(B^n_k)_k\subseteq \c B([0,T])\otimes \c F$ be the corresponding preimages under $X$, that is $B^n_k \coloneqq X^{-1}(A^n_k)$ for all $k\in\N$. We note that if $X$ is $\b F$-progressive, then $B^n_k\in Prog(\b F)$.
Now for each $k\in\N$, we choose a $\c G$-measurable $\alpha^n_k:\Omega'\to [0,1]$ such that $\alpha^n_k\in A^n_k$ (more precise: the $\P$-equivalence class of $\alpha^n_k$ is in $A^n_k$). Then we can define
\[
\Phi^n_s(\omega,\omega') \coloneqq \sum_{k=1}^\infty \alpha^n_k(\omega') \1_{B^n_k}(s,\omega),\qquad (s,\omega,\omega')\in [0,T]\times\Omega\times\Omega'.
\]
Thus we observe that $\norm{\Phi^n_s(\omega,\cdot) - X_s(\omega)}_{L^1(\Omega',\c G,\P;[0,1])}\leq 2^{-n}\to 0$ uniformly in $(s,\omega) \in [0,T]\times\Omega$. Moreover, the $\Phi^n$ are $\c B([0,T])\otimes \c F\otimes \c G$-measurable, but not yet adapted to $\b F\otimes\b G$, since we only have $\alpha^n_k\in L^1(\Omega',\c G,\P;[0,1])$.

To fix this, we define $t^n_k\coloneqq \inf \{t \geq 0 \,|\, (t,\omega)\in B^n_k\text{ for some }\omega\in\Omega\}$ and distinguish two cases: If the infimum is attained, i.e. $t^n_k \in B^n_k(\omega)$ for some $\omega\in\Omega$, we note that $Z^n_k \coloneqq X_{t^n_k} \in L^1(\Omega',\c G_{t^n_k},\P;[0,1])$. In particular there exists a $\c G_{t^n_k}$-measurable $\hat Z^n_k:\Omega'\to [0,1]$ such that $\hat Z^n_k(\cdot) = Z^n_k$ in $L^1(\Omega',\c G,\P;[0,1])$. Furthermore, by construction $\norm{\hat Z^n_k(\cdot) - \alpha^n_k(\cdot)}_{L^1(\Omega',\c G,\P;[0,1])}\leq 2^{-n}$.

In the case that $t^n_k \not\in B^n_k(\omega)$ for all $\omega\in\Omega$, then we can instead find $(s^n_{k,l},\omega^n_{k,l})_l\subseteq B^n_k$ such that $s^n_{k,l}\downarrow t^n_k$.
We now consider the sequence $(Z^n_{k,l})_l \coloneqq (X_{s^n_{k,l}}(\omega^n_{k,l}))_l\subseteq L^1(\Omega',\c G,\P;[0,1])$. Note that by construction $Z^n_{k,l} \in L^1(\Omega',\c G_{s^n_{k,l}},\P;[0,1])$ for all $l\in\N$.
Now since $(Z^n_{k,l})_l$ are uniformly bounded, by \cite[Theorem II.25]{dellacherie_probabilities_1978}, the sequence is relatively compact in the weak $\sigma(L^1,L^\infty)$ topology. Thus, \cite[Theorem II.24]{dellacherie_probabilities_1978} implies that the sequence $(Z^n_{k,l})_l$ converges (along a subsequence) in the weak $\sigma(L^1,L^\infty)$ topology to some $Z^n_k\in L^1(\Omega',\c G,\P;[0,1])$. Further, since $Z^n_{k,l}\in L^1(\Omega',\c G_{s^n_{k,l}},\P;[0,1])$ and $s^n_{k,l}\downarrow t^n_k$, we conclude that $Z^n_k \in L^1(\Omega',\c G_{t^n_k+},\P;[0,1])$,
which means that we can find a $\c G_{t^n_k+}$-measurable $\hat Z^n_k:\Omega'\to [0,1]$ such that $\hat Z^n_k(\cdot) = Z^n_k$ in $L^1(\Omega',\c G,\P;[0,1])$.
At the same time, by construction $\norm{\hat Z^n_k(\cdot) -\alpha^n_k(\cdot)}_{L^1(\Omega',\c G,\P;[0,1])} \leq \liminf_{l\to\infty} \norm{Z^n_{k,l} - \alpha^n_k(\cdot)}_{L^1(\Omega',\c G,\P;[0,1])} \leq 2^{-n}$ since the $L^1$-norm is l.s.c.\@ w.r.t.\@ weak $\sigma(L^1,L^\infty)$ convergence.

Therefore, by defining
\[
Y^n_s(\omega,\omega') \coloneqq \sum_{k=1}^\infty \hat Z^n_k(\omega') \1_{B^n_k}(s,\omega),\qquad (s,\omega,\omega')\in [0,T]\times\Omega\times\Omega',
\]
we obtain $(\c B([0,T])\otimes\c F\otimes\c G)$-$\c B([0,1])$-measurable and $\b F\otimes\b G$-adapted processes $Y^n$ for $n\in\N$, which by construction satisfy for all $(s,\omega)\in [0,T]\times \Omega$,
\begin{align}
\norm{Y^n_s(\omega,\cdot) - \Phi^n_s(\omega,\cdot)}_{L^1(\Omega',\c G,\P;[0,1])} &\leq \sup_k \norm{\hat Z^n_k - \alpha^n_k}_{L^1(\Omega',\c G,\P;[0,1])}
= \sup_k \norm{Z^n_k - \alpha^n_k}_{L^1(\Omega',\c G,\P;[0,1])} \leq 2^{-n},
\end{align}
and thus
\begin{align}
&\norm{Y^n_s(\omega,\cdot) - X_s(\omega)}_{L^1(\Omega',\c G,\P;[0,1])}\\
&\leq \norm{Y^n_s(\omega,\cdot) - \Phi^n_s(\omega,\cdot)}_{L^1(\Omega',\c G,\P;[0,1])}
+ \norm{\Phi^n_s(\omega,\cdot) - X_s(\omega)}_{L^1(\Omega',\c G,\P;[0,1])}
\leq 2^{-n+1} \to 0,
\end{align}
uniformly in $(s,\omega)\in [0,T]\times\Omega$.
Finally, by defining the point-wise limit
\[
Y_s(\omega,\omega') \coloneqq \liminf_{n\to\infty} Y^n_s(\omega,\omega'),\qquad (s,\omega,\omega') \in [0,T]\times\Omega\times\Omega',
\]
we obtain the desired $(\c B([0,T])\otimes\c F\otimes\c G)$-$\c B([0,1])$-measurable and $\b F\otimes\b G$-adapted process, which satisfies by construction $X_s(\omega) = Y_s(\omega,\cdot)$ in $L^1(\Omega',\c G,\P;[0,1])$ for all $(s,\omega)\in [0,T]\times\Omega$.


Furthermore, if $X$ is $\b F$-progressive, then $B^n_k\in Prog(\b F)$, for $n,k\in\N$ and thus, using a similar argument as \cite[Lemma 2.7]{kaden_progressive_2004}, we see that the constructed processes $\hat Z^n_k \1_{B^n_k}$ will be $\b F\otimes\b G$-progressive. This in turn implies that then also $Y^n = \sum_k \hat Z^n_k \1_{B^n_k}$ and $Y = \liminf Y^n$ are $\b F\otimes\b G$-progressive.


    \item 
Let us now consider the case where $A$ is a general Polish space. By a standard result, there exists a Borel-subset $I \subseteq [0,1]$ and a Borel-isomorphism $\phi:A\to I\subseteq [0,1]$ between $A$ and $I$.%
\footnote{If $A$ is uncountable, one may choose $I=[0,1]$.}
We define a left-inverse $\psi:[0,1]\to A$ to $\phi$ by $\psi = \1_I \phi^{-1} + \1_{I^c} a_0$, where $a_0\in A$ is a fixed but arbitrary element of the Polish space $A$. Clearly, $\psi$ is Borel-measurable and $\psi\circ\phi = \operatorname{id}_A$.

Given this identification of $A$ with $I\subseteq [0,1]$, every process $X:[0,T]\times\Omega\to L^0(\Omega',\c G,\P;A)$ can be viewed as a process taking values in $L^0(\Omega',\c G,\P;[0,1])$. Specifically, we define $\bar X \coloneqq \phi(X):[0,T]\times\Omega\to L^0(\Omega',\c G,\P;I) \subseteq L^0(\Omega',\c G,\P;[0,1])$. Then $\bar X$ is  
a $(\c B([0,T])\otimes \c F)$-$(\c B(L^0(\Omega',\c G,\P;[0,1])))$-measurable and $\b F$-adapted process. Furthermore, it holds that 
$\bar X_t(\omega) \in L^0(\Omega',\c G_t,\P;[0,1])\subseteq L^0(\Omega',\c G,\P;[0,1])$ for all $(t,\omega)\in [0,T]\times\Omega$.

By \cref{proof proposition selecting a measurable process case A unit interval}, there exists a corresponding $(\c B([0,T])\otimes \c F\otimes\c G)$-$\c B([0,1])$-measurable and $\b F\otimes\b G$-adapted process $\bar Y:[0,T]\times\Omega\times\Omega'\to [0,1]$, such that
\[
    \bar Y_t(\omega,\cdot) = \bar X_t(\omega)\text{ in }L^0(\Omega',\c G,\P;[0,1]),\qquad\text{for all }(t,\omega)\in [0,T]\times\Omega.
\]
If $\bar X$ is $\b F$-progressive, then $\bar Y$ can also be chosen to be $\b F\otimes\b G$-progressive.

Next, we map this process back to $A$ by defining $Y\coloneqq \psi(\bar Y):[0,T]\times\Omega\times\Omega'\to A$. By construction, $Y$ is $(\c B([0,T])\otimes \c F\otimes\c G)$-$\c B(A)$-measurable and $\b F\otimes\b G$-adapted. Furthermore, if $X$, and thus $\bar X$, is $\b F$-progressive, then $\bar Y$, and thus $Y$, can also be chosen to be $\b F\otimes\b G$-progressive. Finally we have
\[
Y_t = \psi(\bar Y_t(\omega,\cdot)) = \psi(\bar X_t(\omega)) = \psi(\phi(X_t(\omega))) = X_t(\omega)\text{ in }L^0(\Omega',\c G,\P;A),\qquad\text{for all }(t,\omega)\in [0,T]\times\Omega,
\]
which shows that $Y$ is the desired process.

\end{enumerate}
\end{proof}

While the preceding measurable selection result \cref{proposition selecting a measurable process} allows us to select a progressive process for each $\hat\alpha\in\hat{\c A}$, the definition of the action set $\c A$ requires predictability. To obtain such a predictable process, we will leverage the fact that a large class of predictable processes can be mapped to the canonical space $(C^d, \c C^d_T, \mathbf C^d, \mu^d_W)$, see \cref{proposition canonical representation of predictable processes}. W.r.t.\@ this canonical filtration, the concepts of progressiveness and predictability then coincide, see also the upcoming \cref{lemma canonical space predictable adapted measurable equivalence}. Consequently, we will begin by proving these two auxiliary results, which are extensions of \cite[Propositions 10 and 9]{claisse_pseudo-markov_2016}.

\begin{proposition}
\label{proposition canonical representation of predictable processes}
Let $(\Omega,\c F,\b F)$ be a filtered space, let $(\Omega',\c G)$ be a measurable space with a continuous process $X:[0,T]\times\Omega'\to\R^d$ and its generated filtration $\b G^X$ and let $A$ be a Polish space.
Then a process $Y:[0,T]\times\Omega\times\Omega'\to A$ is $\b F\otimes \b G^X$-predictable if and only if there exists a $\b F\otimes \mathbf C^d$-predictable process $\phi:[0,T]\times\Omega\times C^d\to A$ defined on the space $(\Omega\times C^d,\c F\otimes \c C^d_T)$, where $\mathbf C^d$ is the canonical filtration on $C^d$, such that
\[
Y_t(\omega,\omega') = \phi(t,\omega,[X(\omega')]_t) = \phi(t,\omega,X(\omega')),\qquad\text{for all }(t,\omega,\omega')\in [0,T]\times\Omega\times\Omega',
\]
where $[w]_t \coloneqq (w_{s\land t})_{s\in [0,T]}$.
\end{proposition}
\begin{proof}
We first prove that $\psi:(t,\omega,\omega')\mapsto (t,\omega,[X(\omega')]_t)$ is $Pred(\b F\otimes\b G^X)$-$Pred(\b F\otimes\mathbf C^d)$-measurable. To this end, we recall that $Pred(\b F\otimes\mathbf C^d)$ is generated by all sets of the form (i) $(t_1,t_2]\times C\times D$, where $0\leq t_1<t_2\leq T$ and $C\in\c F_{t_1}$, $D\in\c C^d_{t_1}$, and (ii) $\{0\}\times C\times D$ where $C\in \c F_{0-}$, $D\in\{\emptyset,C^d\}$. Thus, the measurability directly follows from noticing that $\psi^{-1}((t_1,t_2]\times C\times D) = (t_1,t_2]\times C\times X^{-1}(D)\in Pred(\b F\otimes\b G^X)$, since for $s\geq t_1$, $X^{-1}_{t_1}(D) = ([X]_s)^{-1}_{t_1}(D)$, respectively from $\psi(\{0\}\times C\times D) = \{0\}\times C\times D'\in Pred(\b F\otimes\b G^X)$, where $D' = \emptyset$ if $D=\emptyset$ and $D' = \Omega'$ if $D=C^d$.

At the same time, we observe that $\psi
$ also generates the $\sigma$-algebra $Pred(\b F\otimes\b G^X)$, that is $Pred(\b F\otimes\b G^X) = \sigma(\psi)$. Recall that the measurability of $\psi$ shows that $\sigma(\psi)\subseteq Pred(\b F\otimes\b G^X)$. For the other direction, we note that since $\c G^X_t = \sigma(X_s^{-1}(C)|s\in [0,t],C\in\c B(\R^d))$, the $\sigma$-algebra algebra $Pred(\b F\otimes\b G^X)$ is generated by all sets of the form (i) $(t_2,t_3]\times B\times X_{t_1}^{-1}(C)$, where $0\leq t_1\leq t_2 < t_3\leq T$ and $B\in\c F_{t_2}$, $C\in\c B(\R^d)$, and (ii) $\{0\}\times B\times D$, where $B\in\c F_{0-}$ and $D\in \{\emptyset,\Omega'\}$.
For (i), noting that for $s\geq t_1$, $X_{t_1}^{-1}(C) = ([X]_s)_{t_1}^{-1}(C) = \{w^d_s\in C\}\in\c C^d_s$, we see that $(t_2,t_3]\times B\times X_{t_1}^{-1}(C) = \psi^{-1}((t_2,t_3]\times B\times \{w_{t_1}\in C\})\in\psi^{-1}(Pred(\b F\otimes\mathbf C^d))$.
For (ii), we note that $\{0\}\times B\times D = \psi^{-1}(\{0\}\times B\times D')$, where $D' = \emptyset$ if $D=\emptyset$ and $D' = C^d$ if $D=\Omega'$, which implies that 
$\{0\}\times B\times D\in \psi^{-1}(Pred(\b F\otimes\b C^d))$.
Therefore also $Pred(\b F\otimes\b G^X)\subseteq \sigma(\psi)$, and we conclude that $Pred(\b F\otimes\b G^X) = \sigma(\psi)$.

Using this auxiliary result, we can now prove both directions separately as follows.
\begin{enumerate}[wide,label=(\roman*)]
    \item Let $Y:[0,T]\times\Omega\times\Omega'\to A$ be $\b F\otimes\b G^X$-predictable. Using that $\psi
    $ is $Pred(\b F\otimes\mathbf C^d)$-$Pred(\b F\otimes\b G^X)$-measurable, and that $Pred(\b F\otimes\mathbf C^d) = \sigma(\psi)$, an application of the Doob-Dynkin Lemma gives us the desired $\b F\otimes\mathbf C^d$-predictable process $\phi:[0,T]\times\Omega\times C^d\to A$ with $Y(t,\omega,\omega') =(\phi\circ\psi)(t,\omega,\omega') = \phi(t,\omega,[X(\omega')]_t)$ for all $(t,\omega,\omega')\in [0,T]\times\Omega\times\Omega'$. Finally, since $[\cdot]_t\circ[\cdot]_t = [\cdot]_t$, we conclude that $\phi(t,\omega,X(\omega')) = Y_t(\omega,\omega') = \phi(t,\omega,[X(\omega')]_t)$ for all $(t,\omega,\omega')\in [0,T]\times\Omega\times\Omega'$.
    \item For the other direction, let $\phi:[0,T]\times\Omega\times C^d\to A$ be such an $\b F\otimes\mathbf C^d$-predictable process. Then since $\psi
    $ is $Pred(\b F\otimes\b G^X)$-$Pred(\b F\otimes\mathbf C^d)$-measurable, we directly obtain that $Y:[0,T]\times\Omega\times\Omega'\to A$, which satisfies $Y_t(\omega,\omega') = \phi(t,\omega,[X(\omega')]_t) = (\phi \circ \psi)(t,\omega,\omega')$ for all $(t,\omega,\omega')\in [0,T]\times\Omega\times\Omega'$, is $\b F\otimes\b G^X$-predictable.
\end{enumerate}
\end{proof}

The following generalisation of \cite[Proposition 9]{claisse_pseudo-markov_2016} is proven the same way as in \cite{claisse_pseudo-markov_2016}, but since the proof is quite short, we will include it for the reader's convenience.

\begin{lemma}
\label{lemma canonical space predictable adapted measurable equivalence}
Let $(\Omega,\c F)$ be a measurable space and $(C^d,\c C^d_T,\mathbf C^d,\mu^d_W)$ the canonical filtered probability space of continuous $\R^d$-valued paths. Let $E$ be a Polish space equipped with its Borel $\sigma$-algebra $\c B(E)$. Then for a process $X:[0,T]\times\Omega\times C^d\to E$ the following statements are equivalent:
\begin{enumerate}[label=(\roman*)]
    \item $X$ is $\c F\otimes \mathbf C^d$-predictable,
    \item $X$ is $\c F\otimes \mathbf C^d$-optional,
    \item $X$ is $\c F\otimes \mathbf C^d$-progressive,
    \item $X$ is $\c B([0,T])\otimes\c F\otimes\c C^d_T$-$\c B(E)$-measurable and $\c F\otimes \mathbf C^d$-adapted,
    \item $X$ is $\c B([0,T])\otimes\c F\otimes\c C^d_T$-$\c B(E)$-measurable and satisfies
    \[
    X_s(\omega,w^d) = X_s(\omega,[w^d]_s),\qquad\text{for all }(s,\omega,w^d) \in [0,T]\times\Omega\times C^d.
    \]
\end{enumerate}
\end{lemma}
\begin{proof}
\begin{enumerate}[wide,label=(\roman*)]
    \item It is known that (i) $\Rightarrow$ (ii) $\Rightarrow$ (iii) $\Rightarrow$ (iv).
    \item Regarding (iv) $\Rightarrow$ (v), we first note that $\c F\otimes\c C^d_s = \sigma((\omega,w^d)\mapsto (\omega,[w^d]_s))$ for all $s\in [0,T]$. Thus using that $E$ is Polish together with $X_s$ being $\c F\otimes\c C^d_s$-measurable, we can apply the Doob-Dynkin Lemma to obtain an $\c F\otimes\c C^d_s$-measurable $Y_s:\Omega\times C^d\to E$ such that $X_s(\omega,w^d) = Y_s(\omega,[w^d]_s)$ for all $(\omega,w^d)\in \Omega\times C^d$. Now using that $[\cdot]_s\circ [\cdot]_s = [\cdot]_s$ for all $s\in [0,T]$, we conclude that $X_s(\omega,[w^d]_s) = Y_s(\omega,[w^d]_s) = X_s(\omega,w^d)$ for all $(s,\omega,w^d)\in [0,T]\times\Omega\times C^d$.
    \item Regarding (v) $\Rightarrow$ (i), we note that $\psi:(s,\omega,w^d)\mapsto (s,\omega,[w^d]_s)$ is $\c F\otimes\b C^d$-predictable. To this end, we recall since $\c C^d_T = \sigma(\{w^d_t\in B\} | t\in [0,T], B\in\c B(\R^d))$ that $\c B([0,T])\otimes\c F\otimes\c C^d_T$ is generated by all sets of the form $D = A\times B\times \{w^d_t\in C\}$, where $A\in\c B([0,T])$, $B\in\c F$ and $t\in [0,T]$, $C\in\c B(\R^d)$. For such sets $D$, we have $\psi^{-1}(D) = \{(s,\omega,w^d)\in A\times B\times C^d \,|\, ([w^d]_s)_t \in C\} = \{(s,\omega,w^d)\in A\times B\times C^d \,|\, w^d_{s\land t}\in C\} \in Pred(\c F\otimes\b C^d)$ since $(s,\omega,w^d)\mapsto w^d_{s\land t}$ is continuous and $\b C^d$-adapted and hence $\b C^d$-predictable. Therefore $\psi$ is $\c F\otimes\b C^d$-predictable, and we can conclude that $X = X\circ \psi$ is also $\c F\otimes\b C^d$-predictable.
\end{enumerate}
\end{proof}

With these tools in hand, we are now ready to go back to the action sets $\c A$ and $\hat{\c A}$ and prove the following \cref{corollary alpha given alpha bar well defined}. The key idea is noting that $\c A_s = L^0(\Omega,\c G\lor\c F^W_s,\P;A)$ and $L^0(\Omega\times C^m,\c G\otimes\c C^m_s,\P\otimes\mu^m_W;A)$ are isomorphic: by decomposing $\alpha\in\c A_s$ similar to the preceding \cref{proposition decompose filtrations of predictable processes} and together with the Doob-Dynkin Lemma, we can find for each $\alpha\in\c A_s$ an $\tilde\alpha\in L^0(\Omega\times C^m,\c G\otimes\c C^m_s,\P\otimes\mu^m_W;A)$ such that $\alpha(\omega) = \tilde\alpha(\omega,W(\omega))$ for all $\omega\in\Omega$, and vice versa. Furthermore, since $\c G$ and $W$ are independent, we also see that $\norm{\alpha}_{\c A_s} = \norm{\tilde\alpha}_{L^0(
\P\otimes\mu^m_W;A
)}$, that means that their respective norms also coincide. Using this canonical probability space has now the advantage that we can use \cref{lemma canonical space predictable adapted measurable equivalence} to show the predictability of the process, which allows us to prove the following \cref{corollary alpha given alpha bar well defined}.

\begin{proposition}\label{corollary alpha given alpha bar well defined}
    For every $\hat\alpha\in\hat{\c A}$, there exists an $\alpha\in\c A$ that $\b F^B$-identifies to $\hat\alpha$.
\end{proposition}
\begin{proof}
Starting from $\hat\alpha\in\hat{\c A}$, recalling that $\hat\alpha$ is $\b F^B$-predictable, we can by \cite[Proposition 9]{claisse_pseudo-markov_2016} (or \cref{proposition canonical representation of predictable processes} with $\Omega$ being the trivial probability space) find a $\mathbf C^n$-predictable $\hat\Psi:[0,T]\times C^n \to \c A_T
$ with $\hat\Psi(s,w)\in \c A_s$ for all $(s,w)\in [0,T]\times C^n$ such that
\[
\hat\alpha_s(\hat\omega) = \hat\Psi_s([B(\hat\omega)]_s) = \hat\Psi_s(B(\hat\omega)),\qquad\text{for all }(s,\hat\omega)\in [0,T]\times\hat\Omega.
\]
Furthermore as indicated before, we use that $\c A_s$ is isomorphic to $L^0(\Omega\times C^m,\c G\otimes\c C^m_s,\P\otimes\mu^m_W;A)$ for all $s\in [0,T]$, and we therefore can find a $\mathbf C^n$-predictable map $\tilde\Psi:[0,T]\times C^n\to L^0(\Omega\times C^m,\c G\otimes\c C^m_T,\P\otimes\mu^m_W;A)$ such that (viewed as equality in $\c A_T$)
\[
\tilde\Psi_s(B(\hat\omega))(\cdot,W(\cdot)) = \hat\Psi_s(B(\hat\omega))(\cdot),\qquad\text{for all }(s,\hat\omega)\in [0,T]\times\hat\Omega\times\Omega.
\]
Note that since $\hat\Psi_s(\hat\omega)\in\c A_s$, also $\tilde\Psi_s(\hat\omega)\in L^0(\Omega\times C^m,\c G\otimes\c C^m_s,\P\otimes\mu^m_W;A)$.
Next we apply \cref{proposition selecting a measurable process}, which enables us to select a $\c G\otimes\mathbf C^m\otimes\mathbf C^n$-progressive process $\Psi:[0,T]\times\Omega\times C^m\times C^n\to A$ from the equivalence classes from $\tilde\Psi$ in the sense that (as equality in $L^0(\Omega\times C^m,\c G\otimes \c C^m_T,\P\otimes\mu^m_W;A)$)
\[
\tilde\Psi_s(w) = \Psi_s(\cdot,\cdot,w),\qquad\text{for all }(s,w)\in [0,T]\times C^n.
\]
Now $\Psi$ is defined on the canonical space and hence \cref{lemma canonical space predictable adapted measurable equivalence} shows that $\Psi$ is also $\c G\otimes\mathbf C^m\otimes\mathbf C^n$-predictable. 
Finally, we define the process $\Phi:[0,T]\times\Omega\times\hat\Omega\to A$ by
\[
\Phi_s(\omega,\hat\omega) \coloneqq \Psi_s(\omega,W(\omega),B(\hat\omega)) 
,\qquad\text{for all }(s,\omega,\hat\omega)\in [0,T]\times\Omega\times\hat\Omega,
\]
which by \cref{proposition canonical representation of predictable processes,proposition decompose filtrations of predictable processes} is now $(\c G\lor\b F^W)\otimes\b F^B$-predictable, and satisfies by construction (viewed as equality in $\c A_T$)
\[
\hat\alpha_s(\hat\omega) = \Phi_s(\cdot,\hat\omega),\qquad\text{for all }(s,\hat\omega)\in[0,T]\times\hat\Omega.
\]
Now it only remains to define the process $\alpha:[0,T]\times\hat\Omega\to A$ by
\[
\alpha_s(\hat\omega) \coloneqq \Phi_s(\pi(\hat\omega),\hat\omega),\qquad (s,\hat\omega)\in [0,T]\times\hat\Omega,
\]
and we obtain the desired $\b F^{W,B}\lor\c G$-predictable process $\alpha\in\c A$ which $\b F^B$-identifies to $\hat\alpha$.
\end{proof}

Finally, we conclude this section with the following \cref{lemma existence N bar N}, which establishes the existence of an $\alpha$ identifying to a given $\hat\alpha$ for a more general class of $\hat\alpha$. This result is particularly important for  \cref{section randomised problem}, where we introduce the randomisation framework, as the Poisson process involved will neither be predictable nor adapted to $\b \c G\lor\b F^{W,B}$ and thus will not fit into the setting of \cref{corollary alpha given alpha bar well defined}.

\begin{corollary}\label{lemma existence N bar N}
    Let $(\hat\Omega,\hat{\c F},\hat\P)$ be an extension of $(\Omega,\c F,\P)$ and $\b F$ be a filtration independent of $\c G,W$. Further let $\hat\alpha:[0,T]\times\hat\Omega\to\c A_T$ be an $\b F$-progressive process such that $\hat\alpha_s(\hat\omega) \in \c A_s$ for all $(s,\hat\omega)\in [0,T]\times\hat\Omega$.
    Then there exists an $\b F^W\lor\c G\lor\b F$-progressive and $\b F^W\lor\c G\lor\c F_T$-predictable process $\alpha$ which $\b F$-identifies to $\hat\alpha$.
\end{corollary}
\begin{proof}
    Similarly to the construction in the proof of \cref{corollary alpha given alpha bar well defined}, we start by using that $\c A_s$ is isomorphic to $L^0(\Omega\times C^m,\c G\otimes\c C^m_s,\P\otimes\mu^m_W;A)$ for all $s\in [0,T]$, and we therefore can find a $\b F$-progressive process $\tilde\alpha:[0,T]\times\hat\Omega\to L^0(\Omega\times C^m,\c G\otimes\c C^m_T,\P\otimes\mu^m_W;A)$ such that (viewed as equality in $\c A_T$)
    \[
    \tilde\alpha_s(\hat\omega)(\cdot,W(\cdot)) = \hat\alpha_s(\hat\omega)(\cdot),\qquad\text{for all }(s,\hat\omega)\in [0,T]\times\hat\Omega,
    \]
    and $\tilde\alpha_s(\hat\omega) \in L^0(\Omega\times C^m,\c G\otimes\c C^m_s,\P\otimes\mu^m_W;A)$ for all $(s,\hat\omega)\in [0,T]\times\hat\Omega$.
    Now using \cref{proposition selecting a measurable process}, we can select a $\c G\otimes\mathbf C^m\otimes\b F$-progressive process $\Psi:[0,T]\times\Omega\times C^m\times\hat\Omega\to A$ such that, viewed as equality in $\c A_T$,
    \[
    \Psi_s(\cdot,\cdot,\hat\omega) = \tilde\alpha_s(\hat\omega),\qquad\text{for all }(s,\hat\omega)\in [0,T]\times\hat\Omega.
    \]
    Since $\Psi$ is defined on the canonical space, \cref{lemma canonical space predictable adapted measurable equivalence} shows that $\Psi$ is also $\c G\otimes\mathbf C^m\otimes\c F_T$-predictable. Finally we define the process $\Phi:[0,T]\times\Omega\times\hat\Omega\to A$ by
    \[
    \Phi_s(\omega,\hat\omega) \coloneqq \Psi_s(\omega,W(\omega),\hat\omega)
    ,\qquad\text{for all }(s,\omega,\hat\omega)\in [0,T]\times\Omega\times\hat\Omega,
    \]
    which is then $(\c G\lor\b F^W)\otimes\b F$-progressive and by \cref{proposition canonical representation of predictable processes,proposition decompose filtrations of predictable processes} also $(\c G\lor\b F^W)\otimes\c F_T$-predictable. Now we only need to define the process $\alpha:[0,T]\times\hat\Omega\to A$ by
    \[
    \alpha_s(\hat\omega) \coloneqq \Phi_s(\pi(\hat\omega),\hat\omega),\qquad (s,\hat\omega)\in [0,T]\times\hat\Omega,
    \]
    which is then the desired $\b F^W\lor\c G\lor\b F$-progressive and $\b F^W\lor\c G\lor\c F_T$-predictable process which $\b F$-identifies to $\hat\alpha$.
\end{proof}

\section{Randomised problem}\label{section randomised problem}

In this section, we derive a randomised formulation for mean-field control problems with common noise.
The main idea of the randomisation method introduced in \cite{bouchard_stochastic_2009} is to replace the control process by an independent Poisson point process taking values in the action space, and then controlling its intensity instead. While in the non-mean-field case, the set $A$ would be the natural choice for the action space, in which the aforementioned Poisson process should take its values, in the mean-field case this choice runs into problem concerning the additional mean-field dependence of the state dynamics when changing the intensity of the Poisson point process. In fact, in \cite{bayraktar_randomized_2018} (without common noise), the authors instead introduced an additional space $\c A_{\text{step}}$ of all $\b F^W$-progressive step processes taking values in $A$, which they then took as action space for the randomising Poisson point process. The motivation behind this approach is that the conditional law process $(\P^{\c F^{B,\P}_s}_{X_s})_s$ is $\b F^{B,\P}$-progressive and as such well-behaved when tilting the Poisson point process $\mu$ w.r.t.\@ $\b F^{B,\mu}$-predictable intensities.

In this paper, we take a similar approach. On the one hand, we simplify the idea of \cite{bayraktar_randomized_2018} by considering the action spaces $\c A_s = L^0(\Omega,\c G\lor\c F^W_s,\P;A)$ defined in \cref{section problem l2 formulation}, which only represent the (conditional) distribution of $\alpha$ at time $s$, instead of considering a whole piece-wise constant process in $\c A_{\text{step}}$ as action at a given time point.%
\footnote{The idea behind $\c A_{\text{step}}$ could be interesting when further extending this setting to include path-dependence. However for now, while it seems to ease the construction of the randomising Poisson point process, it later complicates the proofs and the resulting BSDE formulation. Further it seems to be hard to get rid of the dependence on the specific sequences used to construct $\c A_{\text{step}}$.}
On the other hand, we extend the setting to additionally consider problems with common noise, which is aided by viewing $\hat{\c A}$ as the set of admissible controls instead of $\c A$.

Our first step will be defining an $\c A_s$-valued Poisson point process. Now defining a Poisson point process with marks taking values on different spaces $\c A_s$ depending on the current time $s\in [0,T]$ is difficult, but we are going to utilise that $\c A_0\subseteq\c A_s\subseteq\c A_T$ for all $s\in [0,T]$.

\begin{assumption}\label{assumptions lambda family}
We assume that $(\lambda_s)_{s\in [0,T]}$ is a family of finite measures on $\c A_T$ such that
\begin{assumptionenum}[label=(C\arabic*)]
    \item the topological support of each $\lambda_s$ is given by $\c A_s$ for all $s\in [0,T]$,
    \item $\lambda_s\ll \lambda_r$ for all $0\leq s\leq r\leq T$,
    \item $\sup_{s\in [0,T]} \lambda_s(\c A_s) < \infty$.
\end{assumptionenum}
\end{assumption}

Since it may not be obvious why such a family exists, we will provide a possible construction in \cref{remark construction lambda family}. Note that since $\lambda_s$ topological support is given by $\c A_s$, we will view $\lambda_s$, depending on the context, as measure on $\c A_T$ or $\c A_s$ respectively.

\begin{lemma}\label{remark construction lambda family}
Such a family $(\lambda_s)_{s\in [0,T]}$ satisfying \cref{assumptions lambda family} exists.
\end{lemma}
\begin{proof}
    Since each $\c A_s$ is Polish,\footnote{Recall that every uncountable Polish space is Borel-isomorphic to $[0,1]$, which we can equip with the Lebesgue measure.} we can construct a family of probability measures $(\kappa_s)_{s\in [0,T]\cap\Q}$ on $\c A_T$ such that the topological support of each $\kappa_s$ is given by $\c A_s$ for all $s\in [0,T]\cap\Q$, and by normalisation, we can further assume that $\sup_{s\in [0,T]\cap\Q} \kappa_s(\c A_s) < \infty$.
    Note that this is a countable family and let $(s_n)_n = [0,T]\cap\Q$ be an enumeration of it. Now we define $\lambda_r \coloneqq \sum_{s_n\leq r} 2^{-n} \kappa_{s_n}$ for all $r\in [0,T]$. We see that by construction $\lambda_s\ll\lambda_r$ for all $t\leq s\leq r\leq T$. Further, we recall that the filtration $\c G\lor\b F^{W,\P}$ is left-continuous, see \cite[Chapter 2, Problem 7.6]{karatzas_brownian_1998}, and thus we can approximate every indicator function $\1_A \in\c A_r$, $A\in \c G\lor\c F^{W,\P}_r$ by a series of indicator functions $\1_{A_n}\in\c A_{q_n}$, $A_n\in\c G\lor\c F^{W,\P}_{q_n}$ with $q_n\leq r$,$q_n\in\Q$. Since such indicator functions generate $\c A_s = L^0(\Omega,\c G\lor\c F^W,\P;A)$, this shows that $\bigcup_{q\leq r, q\in\Q} \c A_q\subseteq\c A_r$ is dense in $\c A_r$, and thus the topological support of each $\lambda_r$ is really given by $\c A_r$.
\end{proof}

To construct a randomised setting, let us suppose we are given such a family $(\lambda_s)_{s\in [0,T]}$, and $(\hat\Omega,\hat{\c F},\hat\P)$ is (if needed) an extension of the original space $(\Omega,\c F,\P)$ such that there exists a Poisson random measure $\mu$ on $(0,T]\times\c A_T$, independent of $\c G,W,B$, with intensity $\frac{d\lambda_s}{d\lambda_T}(\alpha) \lambda_T(d\alpha)ds = \lambda_s(d\alpha) ds$. Note that with slight abuse of notation, we again denote the canonical extensions of the $\sigma$-algebra $\c G$, the (non-augmented) filtrations $\b F^W,\b F^{B}$, the random variable $\xi$ and the processes $W,B$ with the same symbols as their original versions. Further, we denote the canonical projection from $(\hat\Omega,\hat{\c F},\hat\P)$ to $(\Omega,\c F,\P)$ by $\pi:\hat\Omega\to\Omega$, and recall that by definition it is probability preserving, that is $\hat\P(\pi^{-1}(A)) = \P(A)$, for all $A\in\c F$.

Given an initial time $t\in [0,T]$, we define the processes $B^t,\mu^t$ via
\[
B^t\coloneqq (B_{\cdot\lor t} - B_t),\quad \mu^t\coloneqq \mu(\cdot\cap (t,T]\times\c A_T)).
\]
Note that by construction $B^t,\mu^t$ are $\hat\P$-independent of $\c F^{B,\mu,\hat\P}_t$.
Next, we construct the piece-wise constant $\b F^{\mu^t}$-progressive process $\hat I^{t,\alpha_t}:[t,T]\times\hat\Omega\to \c A_T$ as follows
\begin{align}
    \hat I^{t,\alpha_t}_s \coloneqq 
    \alpha_t + \int_{(t,s]}\int_{\c A_r} (\alpha - \hat{I}^{t,\alpha_t}_{r-}) \mu(dr,d\alpha),\qquad s\in [t,T],
     \label{eq randomised poisson point process}
\end{align}
where $\alpha_t\in\c A_t\subseteq\c A_T$ is a fixed chosen initial action.
Note that we will later show that both the particular choices of $\alpha_t\in\c A_t$ and the extension $(\hat\Omega,\hat{\c F},\hat\P)$ of $(\Omega,\c F,\P)$ have no influence on the resulting value function. Furthermore, we observe in the following \cref{corollary I given bar I is well defined} that there exists a $\b F^{W,\mu^t}\lor\c G$-progressive process $I^{t,\alpha_t}$ which $\b F^{\hat I^{t,\alpha_t}}$-identifies to $\hat I^{t,\alpha_t}$, and that this choice is even unique in joint law, independent of the chosen extension $(\hat\Omega,\hat{\c F},\hat\P)$ of $(\Omega,\c F,\P)$. In the following we will use the notation $\pi|_{\c G}:(\hat\Omega,\c G)\to (\Omega,\c G)$ to denote the $\c G$-measurable projection from $\hat\Omega$ to $\Omega$. 

\begin{corollary}\label{corollary I given bar I is well defined}
    Given $\hat I^{t,\alpha_t}$, we can choose a càdlàg, $\b F^W\lor\c G\lor\c F^{\hat I^{t,\alpha_t}}_T$-predictable and $\b F^{W,\hat I^{t,\alpha_t}}\lor\c G$-progressive $I^{t,\alpha_t}$ which $\b F^{\hat I^{t,\alpha_t}}$-identifies to $\hat I^{t,\alpha_t}$, unique up to indistinguishability.
    In particular, $I^{t,\alpha_t}$ is also $\b F^W\lor\c G\lor\c F^{\mu^t}_T$-predictable and $\b F^{W,\mu^t}\lor\c G$-progressive.

    Further, this choice is even unique in the joint law $(\pi|_{\c G},\xi,W,B,\mu,\hat I^{t,\alpha_t},I^{t,\alpha_t})$, by which we mean that if $(\hat\Omega',\hat{\c F}',\hat\P')$ is a another extension of $(\Omega,\c F,\P)$ on which suitable $\mu,\hat I^{t,\alpha_t},I^{t,\alpha_t}$ are defined, then also
    \[
    \c L^{\hat\P}(\pi|_{\c G},\xi,W,B,\mu,\hat I^{t,\alpha_t},I^{t,\alpha_t}) = \c L^{\hat\P'}(\pi|_{\c G},\xi,W,B,\mu,\hat I^{t,\alpha_t},I^{t,\alpha_t}).
    \]
\end{corollary}
\begin{proof}
    From \cref{lemma existence N bar N}%
\footnote{To match the setting described in \cref{section problem l2 formulation}, we can define $\c G'\coloneqq\c G\lor\c F^W_t$ and consider $\hat I'\coloneqq \hat I^{t,\alpha_t}_{\cdot - t}:[0,T-t]\times\hat\Omega\to\c A_T$ and $W' \coloneqq W^t_{\cdot - t}$. Note that, for example, $\hat I'$ is then $\c G'\lor\b F^{W'}$-progressive since $\c G'\lor\b F^{W'} = (\c G\lor\c F^W_{t+s})_{s\in [0,T-t]}$.}
    we obtain an $\b F^W\lor\c G\lor\c F^{\mu^t}_T$-predictable and $\b F^{W,\mu^t}\lor\c G$-progressive process $I:[t,T]\times\Omega\to A$ that $\b F^{\hat I^{t,\alpha_t}}$-identifies to $\hat I^{t,\alpha_t}$. Next we define
    \begin{align}
    I^{t,\alpha_t}_s\coloneqq I_t + \sum_{r\in(t,s],\hat I^{t,\alpha_t}_{r-}\not= \hat I^{t,\alpha_t}_r} (I_r - I_{r-}),\qquad s\in [t,T],
    \label{eq corollary 4.2 construction I t alpha_t}
    \end{align}
    which by construction is piece-wise constant and in particular càdlàg. Furthermore, just like $I$, the process $I^{t,\alpha_t}$ is also $\b F^W\lor\c G\lor\c F^{\mu^t}_T$-predictable and $\b F^{W,\mu^t}\lor\c G$-progressive. Finally, since $I$ is $\b F^{\hat I^{t,\alpha_t}}$-identifying with $\hat I^{t,\alpha_t}$, there exists a $(\c G\lor\b F^W)\otimes\b F^{\hat I^{t,\alpha_t}}$-progressive process $\Psi:[t,T]\times\Omega\times\hat\Omega\to A$ such that
    \begin{align}
        I_s(\hat\omega) = \Psi_s(\pi(\hat\omega),\hat\omega),\qquad\text{for all }(s,\hat\omega)\in [0,T]\times\hat\Omega,
        \label{eq corollary 4.2 I Psi}
    \end{align}
    and, viewed as equality in $\c A_T$,
    \begin{align}
        \hat I^{t,\alpha_t}_s(\hat\omega) = \Psi_s(\cdot,\hat\omega),\qquad\text{for all }(s,\hat\omega)\in [0,T]\times\hat\Omega.
        \label{eq corollary 4.2 hat I Psi}
    \end{align}
    Now defining $\Phi:[t,T]\times\Omega\times\hat\Omega\to A$ via
    \[
    \Phi_s\coloneqq \Psi_t + \sum_{r\in(t,s],\hat I^{t,\alpha_t}_{r-}\not= \hat I^{t,\alpha_t}_r} (\Psi_r - \Psi_{r-}) ,\qquad s\in [t,T],
    \]
    we obtain an $(\c G\lor\c F^W)\otimes\b F^{\hat I^{t,\alpha_t}}$-progressive process, which due to \eqref{eq randomised poisson point process,eq corollary 4.2 construction I t alpha_t,eq corollary 4.2 hat I Psi,eq corollary 4.2 I Psi} now also satisfies
    \[
    I^{t,\alpha_t}_s(\hat\omega) = \Phi_s(\pi(\hat\omega),\hat\omega),\qquad\text{for all }(s,\hat\omega)\in [0,T]\times\hat\Omega,
    \]
    and, viewed as equality in $\c A_T$,
    \[
    \hat I^{t,\alpha_t}_s(\hat\omega) = \Phi_s(\cdot,\hat\omega),\qquad\text{for all }(s,\hat\omega)\in [0,T]\times\hat\Omega.
    \]
    Hence, $I^{t,\alpha_t}$ also $\b F^{\hat I^{t,\alpha_t}}$-identifies to $\hat I^{t,\alpha_t}$.

    Finally the uniqueness result follows from \cref{lemma existence N bar N,lemma joint law N bar N map}, noting that $(\hat\Omega,\hat{\c F},\hat\P)$ and $(\hat\Omega',\hat{\c F}',\hat\P')$ are both extensions of $(\Omega,\c G\lor\c F^{W,B}_T,\P)$.
\end{proof}

Having a well-defined corresponding control process $I^{t,\alpha_t}$, we can now define the $\b F^{W,B^t,\mu^t,\hat\P}\lor\c G$-progressive state process $X^{t,\xi,\alpha_t}$ as the unique solution to the following uncontrolled dynamics
\begin{align}
    dX^{t,\xi,\alpha_t}_s &= b(s,\hat\P^{\c F^{B,\mu,\hat\P}_s}_{X^{t,\xi,\alpha_t}_s},X^{t,\xi,\alpha_t}_s,I^{t,\alpha_t}_s) ds + \sigma(s,\hat\P^{\c F^{B,\mu,\hat\P}_s}_{X^{t,\xi,\alpha_t}_s},X^{t,\xi,\alpha_t}_s,I^{t,\alpha_t}_s) dW_s\\
    &\qquad+ \sigma^0(s,\hat\P^{\c F^{B,\mu,\hat\P}_s}_{X^{t,\xi,\alpha_t}_s},X^{t,\xi,\alpha_t}_s,I^{t,\alpha_t}_s) dB_s,\qquad X^{t,\xi,\alpha_t}_t = \xi.
    \label{eq sde randomised problem}
\end{align}
Again $(\hat\P^{\c F^{B,\mu,\hat\P}_s}_{X^{t,\xi,\alpha_t}_s})_{s\in [t,T]}$ shall denote the $\b F^{B^t,\mu^t,\hat\P}$-optional and $\hat\P$-a.s.\@ continuous version of the conditional law $(\c L(X^{t,\xi,\alpha_t}_s|\c F^{B,\mu,\hat\P}_s))_{s\in [t,T]}$, that is $\hat\P^{\c F^{B,\mu,\hat\P}_s}_{X^{t,\xi,\alpha_t}_s} = \c L(X^{t,\xi,\alpha_t}_s|\c F^{B,\mu,\hat\P}_s)$, $\hat\P$-a.s., for all $s\in [t,T]$. We recall that the existence of such a version is ensured by \cite[Lemma A.1]{djete_mckeanvlasov_2022-1}. Further, \cref{assumptions sde coefficients} ensures by standard arguments the existence and uniqueness of $X^{t,\xi,\alpha_t}$.

As admissible controls to our randomised problem, we consider the set $\c V$ of all $Pred(\b F^{B,\mu})\otimes \c B(\c A_T)$-measurable processes which are strictly bounded away from $0$ and $\infty$, that is $0 < \inf_{[0,T]\times\hat\Omega\times \c A_T} \nu \leq \sup_{[0,T]\times\hat\Omega\times\c A_T} \nu < \infty$. Further, we will also introduce the subset $\c V_t\subseteq\c V$ of $Pred(\b F^{B^t,\mu^t})\otimes \c B(\c A_T)$-measurable processes.
Note that by Girsanov's Theorem, see \cite[Proposition 14.4.I]{daley_introduction_2008},
we can define for each admissible control $\nu\in \c V$ a probability measure using the Doléans-Dade exponential $L^\nu$ as follows
\begin{align}\label{eq randomised control girsanov formula}
    \frac{d\hat\P^\nu}{d\hat\P}\Big|_{\c F^{B,\mu,\hat\P}_s} \coloneqq L^\nu_s \coloneqq \exp\Big(\int_{(0,s]} \int_{\c A_r} \log \nu_r(\alpha) \mu(dr,d\alpha) - \int_0^s \int_{\c A_r} (\nu_r(\alpha) - 1) \lambda_r(d\alpha)dr \Big),\qquad s\in [0,T],
\end{align}
such that under $\hat\P^\nu$ the Poisson random measure $\mu$ has the intensity $\nu_s(\alpha) ds\lambda_s(d\alpha)$, (recall that $\mu((0,s],du)$ is supported on $\c A_s\subseteq\c A_T$). By the following \cref{lemma d P nu d P is square integrable}, which can be proved in the same way as \cite[Lemma 2.4]{kharroubi_feynmankac_2015}, we see that $\hat\P^\nu$ is indeed well-defined.

\begin{lemma}\label{lemma d P nu d P is square integrable}
    For every $\nu\in\c V$ the process $L^\nu$ defined in \eqref{eq randomised control girsanov formula} is a uniform integrable martingale and $L_T \in L^2(\hat\Omega,\c F^{B,\mu}_T,\hat\P;\R)$.
\end{lemma}

We note that by standard arguments, since $\hat\P_\xi = \hat\P^\nu_\xi$ for all $\nu\in\c V$, from \cref{assumptions sde coefficients} we obtain the following uniform estimate
\begin{align}\label{eq randomised state dynamics basic estimate}
    \sup_{\alpha_t\in\c A_t} \sup_{\nu\in\c V} \E^{\hat\P^\nu}\Big[\sup_{s\in [t,T]} |X^{t,\xi,\alpha_t}_s|^2 \Big] < \infty.
\end{align}
Our goal is now to maximise the following randomised reward functional
\begin{align}
    J^{\c R}(t,\xi,\alpha_t,\nu) \coloneqq \E^{\hat\P^\nu}\Big[ g(\hat\P_{X^{t,\xi,\alpha_t}_T}^{\c F^{B,\mu,\hat\P}_T},X^{t,\xi,\alpha_t}_T) + \int_t^T f(s,\hat\P_{X^{t,\xi,\alpha_t}_s}^{\c F^{B,\mu,\hat\P}_s},X^{t,\xi,\alpha_t}_s,I^{t,\alpha_t}_s) ds \Big],
\end{align}
where $\b E^{\hat\P^\nu}$ denotes the expectation under the probability measure $\hat\P^\nu$ corresponding to $\nu\in\c V$.
Finally, we also define the randomised value function via
\[
V^{\c R}(t,\xi,\alpha_t) \coloneqq \sup_{\nu\in\c V} J^{\c R}(t,\xi,\alpha_t,\nu) < \infty,
\]
which we know by \cref{assumptions reward functional} and \eqref{eq randomised state dynamics basic estimate} is well-defined. We note in the following \cref{proposition equivalence v v_t randomised value function}, that this is equivalent to taking only optimising over the controls in $\c V_t$.

\begin{proposition}\label{proposition equivalence v v_t randomised value function}
Let $t\in [0,T]$ and $\xi\in L^2(\Omega,\c G\lor\c F^W_t,\P;\R^d)$. Defining the randomised problem over the restricted control set $\c V_t\subseteq\c V$ leads to the same value function, that is
\[
V^{\c R}(t,\xi,\alpha_t) = \sup_{\nu\in\c V_t} J^{\c R}(t,\xi,\alpha_t),
\]
for all $\alpha_t\in\c A_t\subseteq\c A_T$.
\end{proposition}
\begin{proof}
    We first note that $\sup_{\nu\in\c V_t} J^{\c R}(t,\xi,\alpha_t,\nu) \leq \sup_{\nu\in\c V} J^{\c R}(t,\xi,\alpha_t,\nu) \eqqcolon V^{\c R}(t,\xi,\alpha_t)$ follows directly from $\c V_t\subseteq\c V$.
    For the other direction, let us fix $\nu\in\c V$.
    Since $\nu$ is $Pred(\b F^{B,\mu})\otimes\c B(\c A_T)$-measurable, by \cref{proposition decompose filtrations of predictable processes},%
\footnote{Note that $\b F^{B,\mu} = \b F^{B,\mu}_{t\land \cdot}\lor \b F^{B^t,\mu^t} \subseteq \c F^{B,\mu}_t\lor\b F^{B^t,\mu^t}$.}
    we can decompose $\nu$ into an $\c F^{B,\mu}_t\otimes Pred(\b F^{B^t,\mu^t})\otimes\c B(\c A_T)$-measurable process $\Upsilon:[0,T]\times\hat\Omega\times\hat\Omega\to\c A_T$ with $0<\inf_{[0,T]\times\hat\Omega\times\hat\Omega\times\c A_T} \Upsilon\leq \sup_{[0,T]\times\hat\Omega\times\hat\Omega\times\c A_T} \Upsilon < \infty$ such that $\nu(\hat\omega) = \Upsilon(\hat\omega,\hat\omega)$, for all $\hat\omega\in\hat\Omega$. Note that in particular $\Upsilon(\hat\omega,\cdot)\in\c V_t$ for all $\hat\omega\in\hat\Omega$. Now noting that $X^{t,\xi,\alpha_t}$ and $I^{t,\alpha_t}$ are $\hat\P$-independent of $\c F^{B,\mu,\hat\P}_t$, we see that
    using freezing lemma, see e.g.\@ \cite[Lemma 4.1]{baldi_stochastic_2017}, 
    that
    \begin{align}
        J^{\c R}(t,\xi,\alpha_t,\nu)
        &= \E^{\hat\P}\Big[L^{\nu}_T g(\hat\P_{X^{t,\xi,\alpha_t}_T}^{\c F^{B,\mu,\hat\P}_T},X^{t,\xi,\alpha_t}_T) + \int_t^T f(s,\hat\P_{X^{t,\xi,\alpha_t}_s}^{\c F^{B,\mu,\hat\P}_s},X^{t,\xi,\alpha_t}_s,I^{t,\alpha_t}_s) ds \Big]\\
        &= \E^{\hat\P}\Big[L^\nu_t \E^{\hat\P}\Big[\frac{L^\nu_T}{L^\nu_t} \Big(g(\hat\P^{\c F^{B,\mu,\hat\P}_T}_{X^{t,\xi,\alpha_t}_T},X^{t,\xi,\alpha_t}_T) + \int_t^T f(s,\hat\P^{\c F^{B,\mu,\hat\P}_s}_{X^{t,\xi,\alpha_t}_s},X^{t,\xi,\alpha_t}_s,I^{t,\alpha_t}_s) ds\Big)\Big| \c F^{B,\mu,\hat\P}_t\Big]\Big]\\
        &= \int_{\hat\Omega} L^{\nu(\hat\omega)}_t \E^{\hat\P}\Big[\frac{L^{\Upsilon(\hat\omega,\cdot)}_T}{L^{\Upsilon(\hat\omega,\cdot)}_t} \Big(g(\hat\P^{\c F^{B,\mu,\hat\P}_T}_{X^{t,\xi,\alpha_t}_T},X^{t,\xi,\alpha_t}_T) + \int_t^T f(s,\hat\P^{\c F^{B,\mu,\hat\P}_s}_{X^{t,\xi,\alpha_t}_s},X^{t,\xi,\alpha_t}_s,I^{t,\alpha_t}_s) ds\Big)\Big] \hat\P(d\hat\omega)\\
        &= \int_{\hat\Omega} L^{\nu(\hat\omega)}_t \E^{\hat\P}\Big[L^{\Upsilon(\hat\omega,\cdot)}_T \Big(g(\hat\P^{\c F^{B,\mu,\hat\P}_T}_{X^{t,\xi,\alpha_t}_T},X^{t,\xi,\alpha_t}_T) + \int_t^T f(s,\hat\P^{\c F^{B,\mu,\hat\P}_s}_{X^{t,\xi,\alpha_t}_s},X^{t,\xi,\alpha_t}_s,I^{t,\alpha_t}_s) ds\Big)\Big] \hat\P(d\hat\omega)\\
        &= \int_{\hat\Omega} L^{\nu(\hat\omega)}_t J^{\c R}(t,\xi,\alpha_t,\Upsilon(\hat\omega,\cdot)) \hat\P(d\hat\omega)\\
        &\leq \int_{\hat\Omega} L^{\nu(\hat\omega)}_t \sup_{\upsilon\in\c V_t} J^{\c R}(t,\xi,\alpha_t,\upsilon) \ \hat\P(d\hat\omega) = \sup_{\upsilon\in\c V_t} J^{\c R}(t,\xi,\alpha_t,\upsilon),
    \end{align}
    using that for every $\hat\omega\in\hat\Omega$ we have $\Upsilon(\hat\omega,\cdot)\in\c V_t\subseteq\c V$ and thus $\frac{L^{\Upsilon(\hat\omega,\cdot)}_T}{L^{\Upsilon(\hat\omega,\cdot)}_t}$ is $\hat\P$-independent of $\c F^{B,\mu,\hat\P}_t$, together with $\E^{\hat\P}[L^{\upsilon}_t] = 1$ and $L^\upsilon_t$ being $\c F^{B,\mu,\hat\P}_t$-measurable, for all $\upsilon\in\c V$.
\end{proof}

\begin{remark}\label{remark V > 0 vs V geq 0}
    We could also relax the requirement $\inf_{[0,T]\times\hat\Omega\times \c A_T} \nu > 0$ to $\nu\geq 0$. This leads to the same value function $V^{\c R}$, as also noted in \cite[Remark 3.1]{bandini_backward_2018}. However, to simplify our arguments, especially later in the proof of \cref{theorem equivalence original and randomised value function}, we will consider only $\inf_{[0,T]\times\hat\Omega\times \c A_T}\nu > 0$ as this condition ensures that $\hat\P\sim\hat\P^\nu$.
\end{remark}

\begin{remark}\label{remark conditional distribution unchanged under measure change}
    Given $\nu\in\c V$, since $\nu$ is $\b F^{B,\mu}$-predictable and $\hat\P^\nu\sim\hat\P$, the process $\hat\P^{\c F^{B,\mu,\hat\P}_s}_{X^{t,\xi,\alpha_t}_s}$ is also $\b F^{B^t,\mu^t,\hat\P^\nu}$-optional and for all $s\in [t,T]$ satisfies $\hat\P^{\c F^{B,\mu,\hat\P}_s}_{X^{t,\xi,\alpha_t}_s} = \c L^{\hat\P^\nu}(X^{t,\xi,\alpha_t}_s | \c F^{B,\mu,\hat\P^\nu}_s)$, $\hat\P^\nu$-a.s., which hence justifies writing $\hat\P^{\c F^{B,\mu,\hat\P}_s}_{X^{t,\xi,\alpha_t}_s} = \hat\P^{\nu,\c F^{B,\mu,\hat\P^\nu}_s}_{X^{t,\xi,\alpha_t}_s}$.
\end{remark}

\begin{lemma}\label{lemma J R continuous}
The reward functional $J^{\c R}(t,\cdot,\cdot,\cdot):L^2(\Omega,\c G\lor\c F^W_t,\P;\R^d)\times\c A_t\times\c V\to \R$ is continuous for each $t\in [0,T]$, where we equip $\c V$ with the $L^\infty$-norm.
\end{lemma}
\begin{proof}
Let $(\xi^n,\alpha_t^n,\nu^n)_n\subseteq L^2(\Omega,\c G\lor\c F^W_t,\P;\R^d)\times\c A_t\times\c V$ with $(\xi^n,\alpha_t^n,\nu^n) \to (\xi^\infty,\alpha_t^\infty,\nu^\infty) \in L^2(\Omega,\c G\lor\c F^W_t,\P;\R^d)\times\c A_t\times\c V$ in $L^2(\P)\times L^0(\hat\P)\times L^\infty(ds\times\hat\P)$. Further let us define $\bar\nu \coloneqq \sup_{k\in\N\cup\{\infty\}} \nu^k$. Note that since $\nu^n\to\nu^\infty$ in $L^\infty(ds\times\hat\P)$, by construction $\bar\nu\in\c V$ and thus $\hat\P^{\bar\nu} \sim \hat\P$. By \cref{lemma alpha bar alpha uniqueness}, and using that $\rho < 1$ we see that also $d_{\c A}^{\hat\P^{\bar\nu}}(I^{t,\alpha_t^n},I^{t,\alpha_t^\infty}) = d_{\hat{\c A}}^{\hat\P^{\bar\nu}}(\hat I^{t,\alpha_t^n},\hat I^{t,\alpha_t^\infty}) \to 0$. Furthermore, since $\hat\P|_{\c G\lor\c F^W_t} = \hat\P^{\bar\nu}|_{\c G\lor\c F^W_t}$, we also have $\xi^n\to\xi^\infty$ in $L^2(\hat\P^{\bar\nu})$.
Since $X^{t,\xi^n,\alpha_t^n}$ and $X^{t,\xi^\infty,\alpha_t^\infty}$ are $\c G\lor\b F^{W,B,\mu,\hat\P^{\bar\nu}}$-progressive, and since $\c G\lor\b F^W$ is independent of $\b F^{B,\mu}$ under $\hat\P^{\bar\nu}$, we see that, for all $s\in [t,T]$, $\hat\P^{\bar\nu}$-a.s.,
\[
\c W_2(\hat\P^{\bar\nu,\c F^{B,\mu,\hat\P^{\bar\nu}}_s}_{X^{t,\xi^n,\alpha_t^n}_s},\hat\P^{\bar\nu,\c F^{B,\mu,\hat\P^{\bar\nu}}_s}_{X^{t,\xi^\infty,\alpha_t^\infty}_s})^2
\leq \E^{\hat\P^{\bar\nu}}[|X^{t,\xi^n,\alpha_t^n}_s - X^{t,\xi^\infty,\alpha_t^\infty}_s|^2 | \c F^{B,\mu,\hat\P^{\bar\nu}}_s]
= \E^{\hat\P^{\bar\nu}}[|X^{t,\xi^n,\alpha_t^n}_s - X^{t,\xi^\infty,\alpha_t^\infty}_s|^2 | \c F^{B,\mu,\hat\P^{\bar\nu}}_T],
\]
which implies due to the $\hat\P^{\bar\nu}$-a.s.\@ continuity of $(\hat\P^{\bar\nu,\c F^{B,\mu,\hat\P^{\bar\nu}}_s}_{X^{t,\xi^n,\alpha_t^n}_s})_s$ and $(\hat\P^{\bar\nu,\c F^{B,\mu,\hat\P^{\bar\nu}}_s}_{X^{t,\xi^\infty,\alpha_t^\infty}_s})_s$ that
\[
\sup_{s\in [t,T]} \c W_2(\hat\P^{\bar\nu,\c F^{B,\mu,\hat\P^{\bar\nu}}_s}_{X^{t,\xi^n,\alpha_t^n}_s},\hat\P^{\bar\nu,\c F^{B,\mu,\hat\P^{\bar\nu}}_s}_{X^{t,\xi^\infty,\alpha_t^\infty}_s})^2
\leq \E^{\hat\P^{\bar\nu}}\Big[ \sup_{s\in [t,T]} |X^{t,\xi^n,\alpha_t^n}_s - X^{t,\xi^\infty,\alpha_t^\infty}_s|^2 \,\Big|\,\c F^{B,\mu,\hat\P^{\bar\nu}}_T\Big].
\]
Thus from \cref{assumptions sde coefficients}, standard arguments using the BDG-inequality and Gronwall's Lemma show that
\begin{align}
    \E^{\hat\P^{\bar\nu}}\Big[ \sup_{s\in [t,T]} \Big(|X^{t,\xi^n,\alpha_t^n}_s - X^{t,\xi^\infty,\alpha_t^\infty}_s|^2 + \c W_2(\hat\P^{\bar\nu,\c F^{B,\mu,\hat\P^{\bar\nu}}_s}_{X^{t,\xi^n,\alpha_t^n}_s},\hat\P^{\bar\nu,\c F^{B,\mu,\hat\P^{\bar\nu}}_s}_{X^{t,\xi^\infty,\alpha_t^\infty}_s})^2\Big)\Big] \to 0.
    \label{eq lemma J R continuous convergence of X}
\end{align}
Finally, we see that for all $n\in\N$,
\begin{align}
    &|J^{\c R}(t,\xi^n,\alpha_t^n,\nu^n) - J^{\c R}(t,\xi^\infty,\alpha_t^\infty,\nu^\infty)|\\
    &\leq \E^{\hat\P^{\bar\nu}}\Big[\Big| \frac{L^{\nu^n}_T}{L^{\bar\nu}_T} \Big(g(\hat\P^{\bar\nu,\c F^{B,\mu,\hat\P^{\bar\nu}}_T}_{X^{t,\xi^n,\alpha_t^n}_T},X^{t,\xi^n,\alpha_t^n}_T) + \int_t^T f(s,\hat\P^{\bar\nu,\c F^{B,\mu,\hat\P^{\bar\nu}}_s}_{X^{t,\xi^n,\alpha_t^n}_s},X^{t,\xi^n,\alpha_t^n}_s,I^{t,\alpha_t^n}_s) ds \Big)\\
    &\qquad - \frac{L^{\nu^\infty}_T}{L^{\bar\nu}_T} \Big(g(\hat\P^{\bar\nu,\c F^{B,\mu,\hat\P^{\bar\nu}}_T}_{X^{t,\xi^\infty,\alpha_t^\infty}_T},X^{t,\xi^\infty,\alpha_t^\infty}_T) + \int_t^T f(s,\hat\P^{\bar\nu,\c F^{B,\mu,\hat\P^{\bar\nu}}_s}_{X^{t,\xi^\infty,\alpha_t^\infty}_s},X^{t,\xi^\infty,\alpha_t^\infty}_s,I^{t,\alpha_t^\infty}_s) ds \Big) \Big| \Big].
\end{align}
Thus, together with, for all $k\in\N\cup\{\infty\}$,
\begin{align}
    0\leq \frac{L^{\nu^k}_T}{L^{\bar\nu}_T}
    &= \exp\Big(\int_{[0,T]}\int_{\c A_s} \log \frac{\nu^k_s(\alpha)}{\bar\nu_s(\alpha)} \mu(ds,d\alpha) - \int_0^T \int_{\c A_s} (\nu^k_s(\alpha) - \bar\nu_s(\alpha)) \lambda_s(d\alpha)ds \Big)\\
    &\leq \exp\Big(\int_0^T \int_{\c A_s} \lambda_s(d\alpha)ds\Big)^{\norm{\bar\nu}_\infty} < \infty,
\end{align}
and \eqref{eq lemma J R continuous convergence of X}, \eqref{eq randomised state dynamics basic estimate} and \cref{assumptions reward functional}, we obtain the desired result by dominated convergence.
\end{proof}

Our first main result is that this randomised problem is in some sense equivalent to the original problem. Note that this then also allows us to transfer the properties of $V$ from \cref{proposition value function law invariant and independent of probabilistic setting} to $V^{\c R}$.
We give the proof in \cref{section proof of equivalence randomised non randomised formulation}.

\begin{theorem}\label{theorem equivalence original and randomised value function}
    The randomised problem is equivalent to the original MFC problem in the sense that for all $t\in [0,T]$ and all $\xi\in L^2(\Omega,\c G\lor\c F^W_t,\P;\R^d)$ and $\alpha_t\in\c A_t\subseteq\c A_T$,
    \[
    V(t,\xi) = V^{\c R}(t,\xi,\alpha_t).
    \]
    In particular, $V^{\c R}$ depends neither on the initial action $\alpha_t\in\c A_t$, the family $(\lambda_s)_{s\in [0,T]}$ nor the extension $(\hat\Omega,\hat{\c F},\hat\P)$ of $(\Omega,\c F,\P)$.
\end{theorem}

\begin{remark}
    By \cref{proposition value function law invariant and independent of probabilistic setting}, $V$ depends on $\xi$ only through its law $\P_\xi$ and is independent of the specific choice of $(\Omega,\c F,\P)$ and $\c G$. The equivalence result above thus implies that $V^{\c R}$ also depends only on $\P_\xi$ and remains the same for all spaces $(\hat\Omega,\hat{\c F},\hat\P)$ for which suitable $\c G,W,B,\mu$ exist, since every space is a trivial extension of itself.
\end{remark}

\section{Randomised dynamic programming}\label{section bsde characterisation and randomised dpp}

In this section, we establish a randomised dynamic programming principle to characterise the value function of our mean-field control problem. Following the randomisation approach, we will first derive a BSDE representation for the (randomised) value function through penalisation techniques. This representation, combined with the equivalence between the original and randomised formulations established in \cref{theorem equivalence original and randomised value function}, subsequently enables us to derive the (randomised) dynamic programming principle. 

\subsection{BSDE representation}\label{section bsde characterisation}

We start by deriving a BSDE representation for the value function $V$ using its randomised setting from \cref{theorem equivalence original and randomised value function}. 
For this, we first define the following spaces
\begin{itemize}
    \item $\c S^2_{[t,T]}(\b F^{B,\mu,\hat\P})$, the set of all càdlàg $\b F^{B,\mu,\hat\P}$-adapted processes $Y:[t,T]\times\hat\Omega\to\R$ processes $Y$ such that $\E^{\hat\P}[\int_t^T |Y_u|^2 du] < \infty$,
    \item $L^2_{W,[t,T]}(\b F^{B,\mu,\hat\P})$, the set of all $\b F^{B,\mu,\hat\P}$-predictable processes $Z:[t,T]\times\hat\Omega\to\R^n$ that satisfy $\E^{\hat\P}[\int_t^T |Z_u|^2 du] < \infty$,
    \item $L^2_{\lambda,[t,T]}(\b F^{B,\mu,\hat\P})$, the set of all $Pred(\b F^{B,\mu,\hat\P})\otimes\c B(A)$-measurable processes $U:[t,T]\times\hat\Omega\times\c A_T\to\R$ such that $\E^{\hat\P}[\int_t^T \int_{\c A_u} |U_u(\alpha)|^2 \lambda_u(d\alpha)du] < \infty$,
    \item $\c K^2_{[t,T]}(\b F^{B,\mu,\hat\P})\subseteq\c S^2_{[t,T]}(\b F^{B,\mu,\hat\P})$ the subset of all processes $K:[t,T]\times\hat\Omega\to\R$ that are additionally nondecreasing, $\b F^{B,\mu,\hat\P}$-predictable and satisfy $K_t = 0$.
\end{itemize}
As it is standard for the randomisation approach, we will characterise the value function as the minimal solution to a certain constrained BSDE. For this, our main tool will be \cite[Theorem 2.1]{kharroubi_feynmankac_2015} since it replaces the problem of characterising such a minimal solution by studying the limit of solutions to a class of penalised BSDEs. Hence, we will first study in the following \cref{lemma penalised bsde minimal solution existence} the penalised BSDE \eqref{eq penalised bsde value function} and show that it corresponds to the optimisation over the following restricted randomised control set $\c V^n\subseteq\c V$ of all $Pred(\b F^{B,\mu})\otimes\c B(\c A_T)$-measurable processes $\nu:[0,T]\times\hat\Omega\times\c A_T\to (0,n]$ with $\inf_{[0,T]\times\hat\Omega\times\c A_T} \nu > 0$, based on the methodology by \cite{kharroubi_feynmankac_2015,bayraktar_randomized_2018,bandini_backward_2018}. The proof is given in \cref{appendix proof lemma penalised bsde minimal solution existence}. 

\begin{lemma}\label{lemma penalised bsde minimal solution existence}
    For $n\in\N$, $t\in [0,T]$, $\alpha_t\in\c A_t\subseteq\c A_T$ and $\xi\in L^2(\Omega,\c G\lor\c F^W_t,\P;\R^d)$, there exists a unique solution
    \[
    (Y^{n,t,\xi,\alpha_t},Z^{n,t,\xi,\alpha_t},U^{n,t,\xi,\alpha_t})\in \c S^2_{[t,T]}(\b F^{B,\mu,\hat\P})\times L^2_{W,[t,T]}(\b F^{B,\mu,\hat\P})\times L^2_{\lambda,[t,T]}(\b F^{B,\mu,\hat\P})
    \]
    to following penalised BSDE,
    \begin{align}\label{eq penalised bsde value function}
    Y^{n,t,\xi,\alpha_t}_s &= \E^{\hat\P}[g(X^{t,\xi,\alpha_t}_T,\hat\P^{\c F^{B,\mu,\hat\P}_T}_{X^{t,\xi,\alpha_t}_T}) | \c F^{B,\mu,\hat\P}_T] + \int_s^T \E^{\hat\P}[f(r,\hat\P^{\c F^{B,\mu,\hat\P}_r}_{X^{t,\xi,\alpha_t}_r},X^{t,\xi,\alpha_t}_r,I^{t,\alpha_t}_r) | \c F^{B,\mu,\hat\P}_r] dr\\
    &\qquad - \int_s^T Z^{n,t,\xi,\alpha_t}_r dB_r + n \int_s^T \int_{\c A_r} (U^{n,t,\xi,\alpha_t}_r(\alpha))_+ \lambda_r(d\alpha) dr - \int_s^T \int_{\c A_r} U^{n,t,\xi,\alpha_t}_r(\alpha) \mu(dr,d\alpha).
    \end{align}
    Further, for any $s\in [t,T]$ and $r\in [s,T]$, we also have the following representation,
    \begin{align}\label{eq Y n t xi snell envelope formula}
    Y^{n,t,\xi,\alpha_t}_s
    &= \esssup_{\nu\in\c V^n} \E^{\hat\P^\nu}\Big[Y^{n,t,\xi,\alpha_t}_r + \int_s^r f(u,\hat\P^{\nu,\c F^{B,\mu,\hat\P^\nu}_u}_{X^{t,\xi,\alpha_t}_u},X^{t,\xi,\alpha_t}_u,I^{t,\alpha_t}_u) du\Big| \c F^{B,\mu,\hat\P^\nu}_s\Big],\qquad\hat\P\text{-a.s.},
    \end{align}
    and moreover for any $r\in [t,T]$, we have the following estimate
    \begin{align}\label{eq Y n t xi snell envelope upper bound at s=t}
        \E^{\hat\P}[Y^{n,t,\xi,\alpha_t}_t] \leq \sup_{\nu\in\c V^n} \E^{\hat\P^\nu}\Big[Y^{n,t,\xi,\alpha_t}_r + \int_t^r f(u,\hat\P^{\nu,\c F^{B,\mu,\hat\P^\nu}_u}_{X^{t,\xi,\alpha_t}_u},X^{t,\xi,\alpha_t}_u,I^{t,\alpha_t}_u) du\Big].
    \end{align}
\end{lemma}

Using the previous result on penalised BSDEs, we can now show the following BSDE formulation for our randomised control problem. Our approach is similar to \cite{bayraktar_randomized_2018,bandini_backward_2018} and based on \cite[Theorem 2.1]{kharroubi_feynmankac_2015} which connects penalised BSDEs like \eqref{eq penalised bsde value function} with constrained BSDEs of the form \eqref{eq constrained bsde value function}. The proof is given in \cref{appendix proof lemma bsde minimal solution existence}. 

\begin{theorem}\label{lemma bsde minimal solution existence}
    Let $(t,m)\in [0,T]\times\c P_2(\R^d)$. For any $\alpha_t\in\c A_t\subseteq\c A_T$ and $\xi\in L^2(\Omega,\c G\lor\c F^W_t,\P;\R^d)$ such that $\P_\xi = m$, there exists a unique minimal solution
    \[
    (Y^{t,\xi,\alpha_t},Z^{t,\xi,\alpha_t},U^{t,\xi,\alpha_t},K^{t,\xi,\alpha_t})\in \c S^2_{[t,T]}(\b F^{B,\mu,\hat\P})\times L^2_{W,[t,T]}(\b F^{B,\mu,\hat\P})\times L^2_{\lambda,[t,T]}(\b F^{B,\mu,\hat\P})\times \c K^2_{[t,T]}(\b F^{B,\mu,\hat\P}),
    \]
    in the sense that for any other solution $(\tilde Y^{t,\xi,\alpha_t},\tilde Z^{t,\xi,\alpha_t},\tilde U^{t,\xi,\alpha_t},\tilde K^{t,\xi,\alpha_t})$ we have $Y^{t,\xi,\alpha_t}\leq \tilde Y^{t,\xi,\alpha_t}$, to the following constrained BSDE
    \begin{align}\label{eq constrained bsde value function}
        Y^{t,\xi,\alpha_t}_s &= \E^{\hat\P}[g(X^{t,\xi,\alpha_t}_T,\hat\P^{\c F^{B,\mu,\hat\P}_T}_{X^{t,\xi,\alpha_t}_T}) | \c F^{B,\mu,\hat\P}_T] + \int_s^T \E^{\hat\P}[f(r,\hat\P^{\c F^{B,\mu,\hat\P}_r}_{X^{t,\xi,\alpha_t}_r},X^{t,\xi,\alpha_t}_r,I^{t,\alpha_t}_r) | \c F^{B,\mu,\hat\P}_r] dr\\
        &\qquad - \int_s^T Z^{t,\xi,\alpha_t}_r dB_r + K^{t,\xi,\alpha_t}_T - K^{t,\xi,\alpha_t}_s - \int_s^T \int_{\c A_r} U^{t,\xi,\alpha_t}_r(\alpha) \mu(dr,d\alpha),\\
        U^{t,\xi,\alpha_t}_s(\alpha)&\leq 0,\qquad \text{for }\lambda_s(d\alpha)\text{-a.e. }\alpha \in \c A_s,\text{ for }ds\text{-a.e. }s\in [t,T].
    \end{align}
    Further for any $s\in [t,T]$ and $r\in [s,T]$ we also have the following representation,
    \begin{align}\label{eq recursive snell envelope formula}
        Y^{t,\xi,\alpha_t}_s
        &= \esssup_{\nu\in\c V} \E^{\hat\P^\nu}\Big[Y^{t,\xi,\alpha_t}_r + \int_s^r f(u,\hat\P^{\nu,\c F^{B,\mu,\hat\P^\nu}_u}_{X^{t,\xi,\alpha_t}_u},X^{t,\xi,\alpha_t}_u,I^{t,\alpha_t}_u) du\Big| \c F^{B,\mu,\hat\P^\nu}_s\Big],\qquad \hat\P\text{-a.s}.
    \end{align}
   Finally, $Y^{t,\xi,\alpha_t}_t$ is $\hat\P$-a.s. constant, and satisfies for any $r\in [t,T]$ also
    \begin{align}\label{eq recursive snell envelope formula at s=t}
    Y^{t,\xi,\alpha_t}_t
    &= \sup_{\nu\in\c V} \E^{\hat\P^\nu}\Big[Y^{t,\xi,\alpha_t}_r + \int_t^r f(u,\hat\P^{\nu,\c F^{B,\mu,\hat\P^\nu}_u}_{X^{t,\xi,\alpha_t}_u},X^{t,\xi,\alpha_t}_u,I^{t,\alpha_t}_u) du\Big],\qquad\hat\P\text{-a.s.},
    \end{align}
    and thus $Y^{t,\xi,\alpha_t}_t = V(t,m)$, $\hat\P$-a.s.
\end{theorem}

\subsection{Randomised dynamic programming principle}\label{section randomised dpp}

Recalling that in \cref{section randomised problem} we have established an equivalence between the value function of the original and the randomised problem, we note that $V^{\c R} = V$ only depends on time and initial law of $X$. Thus the question about a dynamic programming principle is well-posed and in this section we want to use the randomised formulation to establish such a DPP for $V^{\c R}$ and thus also for $V$.

We start by studying the flow property of the randomised state dynamics \eqref{eq sde randomised problem}. For this, we will need to define solutions $X^{t,\zeta,\gamma_t}$ to \eqref{eq sde randomised problem} for more general initial conditions $\zeta\in L^2(\hat\Omega,\c G\lor\c F^{W,B,\mu}_t,\hat\P;\R^d)$ and $\gamma_t \in L^0(\hat\Omega,\c G\lor\c F^{W,B,\mu}_t,\hat\P;A)$. We note that as in \cref{section randomised problem} the existence and uniqueness of $X^{t,\zeta,\gamma_t}$ is standard since both the existence and uniqueness of $I^{t,\gamma_t}$ given $\hat I^{t,\gamma_t}$ by \cref{corollary I given bar I is well defined} also extend to such more general initial conditions. If moreover $\sup_{\nu\in\c V} \E^{\hat\P}[|\zeta|^2] < \infty$ holds, then similarly to the estimate in \eqref{eq randomised state dynamics basic estimate}, we obtain
\begin{align}
\sup_{\gamma_t}\sup_{\nu\in\c V} \E^{\hat\P^\nu}\Big[\sup_{s\in [t,T]} |X^{t,\zeta,\gamma_t}_s|^2 \Big] < \infty.
\label{eq extended randomised state dynamics basic estimate}
\end{align}
This now allows us to write down the following flow property for \eqref{eq sde randomised problem}.

\begin{lemma}\label{lemma flow property}
    Let $t\in [0,T]$ and $\xi\in L^2(\Omega,\c G\lor\c F^W_t,\P;\R^d)$. Then for all $s\in [t,T]$, the following flow property holds
    \[
    \big(X^{s,X^{t,\xi,\alpha_t}_s,I^{t,\alpha_t}_s}_u,I^{s,I^{t,\alpha_t}_s}_u,\hat\P^{\c F^{B,\mu,\hat\P}_u}_{X^{s,X^{t,\xi,\alpha_t}_s,I^{t,\alpha_t}_s}_u}\big)_{u\in [s,T]} = \big(X^{t,\xi,\alpha_t}_u,I^{t,\alpha_t}_u,\hat\P^{\c F^{B,\mu,\hat\P}_u}_{X^{t,\xi,\alpha_t}_u}\big)_{u\in [s,T]},\qquad\hat\P\text{-a.s}.
    \]
\end{lemma}

\begin{proof}
We note that by \cref{corollary I given bar I is well defined}, $I^{t,\alpha_t}$ and $I^{s,I^{t,\alpha_t}}$ are indistinguishable on $[s,T]$. Thus $X^{t,\xi,\alpha_t}$ and $X^{s,X^{t,\xi,\alpha_t}_s,I^{t,\alpha_t}_s}$ both solve \eqref{eq sde randomised problem} on $[s,T]$, and hence the flow property follows from the strong uniqueness of solutions to \eqref{eq sde randomised problem}.
\end{proof}

Next let us take a closer look at such randomised settings with a randomised strategy $\nu\in\c V$ and general initial conditions $\zeta\in L^2(\hat\Omega,\c G\lor\c F^{W,B,\mu}_t,\hat\P;\R^d)$, $\gamma_t\in L^0(\hat\Omega,\c G\lor\c F^{W,B,\mu}_t,\hat\P;A)$. We want to decompose both $\zeta,\gamma_t$ and $\nu$ along the $\sigma$-algebra $\c F^{B,\mu,\hat\P}_t$, to be able to relate this to our original setting with $\xi\in L^2(\hat\Omega,\c G\lor\c F^W_t,\hat\P;\R^d)$ and $\alpha_t\in\c A_t$, which we studied in the previous sections. For this, we recall that we introduced $B^t\coloneqq (B_{\cdot\lor t} - B_t)$ and $\mu^t\coloneqq \mu(\cdot\cap ((t,T]\times\c A_T))$ in \cref{section randomised problem}.
We then note that by \cref{proposition decompose filtrations of predictable processes}, we can decompose both $\zeta,\gamma_t$ and $\nu$ into unique $\Xi\in L^2(\hat\Omega\times\hat\Omega,\c F^{B,\mu}_t\otimes(\c G\lor\c F^W_t),\hat\P\otimes\hat\P;\R^d)$, $C_t\in L^0(\hat\Omega\times\hat\Omega,\c F^{B,\mu}_t\otimes(\c G\lor\c F^W_t),\hat\P\otimes\hat\P;A)$ and an $\c F^{B,\mu}_t\otimes Pred(\b F^{B^t,\mu^t})\otimes\c B(\c A_T)$-measurable process $\Upsilon$ with $0<\inf_{[0,T]\times\hat\Omega\times\hat\Omega\times\c A_T} \Upsilon\leq \sup_{[0,T]\times\hat\Omega\times\hat\Omega\times\c A_T} \Upsilon < \infty$ such that $(\zeta,\gamma_t,\nu)(\hat\omega) = (\Xi,C_t,\Upsilon)(\hat\omega,\hat\omega)$, for $\hat\P$-a.a.\@ $\hat\omega\in\hat\Omega$. This decomposition is the motivation for the following lemma, which studies how this decomposition reflects in the reward functional.

\begin{lemma}\label{lemma generalised randomised payoff decomposition}    
    Let $t\in [0,T]$, $\zeta\in L^2(\hat\Omega,\c G\lor\c F^{W,B,\mu}_t,\hat\P;\R^d)$, $\gamma_t\in L^0(\hat\Omega,\c G\lor\c F^{W,B,\mu}_t,\hat\P;A)$ and $\nu\in\c V$,
    and let $\Xi\in L^2(\hat\Omega\times\hat\Omega,\c F^{B,\mu}_t\otimes(\c G\lor\c F^W_t),\hat\P\otimes\hat\P;\R^d)$, $C_t\in L^0(\hat\Omega\times\hat\Omega,\c F^{B,\mu}_t\otimes(\c G\lor\c F^W_t),\hat\P\otimes\hat\P;A)$ and $\Upsilon$ be an $\c F^{B,\mu}_t\otimes Pred(\b F^{B^t,\mu^t})\otimes\c B(\c A_T)$-measurable process such that $0<\inf_{[0,T]\times\hat\Omega\times\hat\Omega\times\c A_T} \Upsilon\leq \sup_{[0,T]\times\hat\Omega\times\hat\Omega\times\c A_T} \Upsilon < \infty$ (in particular, $\Upsilon(\hat\omega,\cdot)\in\c V_t$ for every $\hat\omega\in\hat\Omega$), such that $(\zeta,\gamma_t,\nu)(\hat\omega) = (\Xi,C_t,\Upsilon)(\hat\omega,\hat\omega)$, for $\hat\P$-a.a.\@ $\hat\omega\in\hat\Omega$. Assuming that $\sup_{\upsilon\in\c V} \E^{\hat\P^\upsilon}[|\zeta|^2]<\infty$, then it holds that
    \begin{align}
        &\E^{\hat\P^\nu}\Big[g(\hat\P^{\nu,\c F^{B,\mu,\hat\P^\nu}_T}_{X^{t,\zeta,\gamma_t}_T},X^{t,\zeta,\gamma_t}_T) + \int_t^T f(s,\hat\P^{\nu,\c F^{B,\mu,\hat\P^\nu}_s}_{X^{t,\zeta,\gamma_t}_s},X^{t,\zeta,\gamma_t}_s,I^{t,\gamma_t}_s) ds\Big| \c F^{B,\mu,\hat\P^\nu}_t\Big](\hat\omega)\\
        &= \E^{\hat\P^\upsilon}\Big[g(\hat\P^{\upsilon,\c F^{B,\mu,\hat\P^\upsilon}_T}_{X^{t,\xi,\alpha_t}_T},X^{t,\xi,\alpha_t}_T) + \int_t^T f(s,\hat\P^{\upsilon,\c F^{B,\mu,\hat\P^\upsilon}_s}_{X^{t,\xi,\alpha_t}_s},X^{t,\xi,\alpha_t}_s,I^{t,\alpha_t}_s) ds \Big]\bigg|_{\xi = \Xi(\hat\omega,\cdot),\alpha_t = C_t(\hat\omega,\cdot),\upsilon = \Upsilon(\hat\omega,\cdot)}\\
        &= J^{\c R}(t,\xi,\alpha_t,\upsilon)\Big|_{\xi = \Xi(\hat\omega,\cdot),\alpha_t = C_t(\hat\omega,\cdot), \upsilon = \Upsilon(\hat\omega,\cdot)},\qquad\text{ for }\hat\P\text{-a.a.\@ }\hat\omega\in\hat\Omega.
        \label{eq generalised randomised payoff decomposition}
    \end{align}
\end{lemma}
\begin{proof}
To prove the formula \eqref{eq generalised randomised payoff decomposition}, we will 
first consider the case where $\hat\omega\to (\Xi(\hat\omega,\cdot),C_t(\hat\omega,\cdot),\Upsilon(\hat\omega,\cdot))$ takes finitely many values, and then use continuity of $J$ in $\xi,\alpha_t,\nu$ and of $(X,\P^{\c F^{B,\mu,\P}}_X)$ in $\Xi,C_t,\Upsilon$ to obtain the result for general $\Xi,C_t,\Upsilon$. 
\begin{enumerate}[wide,label=(\roman*)]
\item\label{lemma 7.2 step i} Let us assume that $\hat\omega\mapsto (\Xi(\hat\omega,\cdot),C_t(\hat\omega,\cdot),\Upsilon(\hat\omega,\cdot))$ only takes finitely many values (in the spaces $L^2(\hat\Omega,\c G\lor\c F^W_t,\hat\P;\R^d)$, $\c A_t$ and $\c V_t$), and thus are of the form
\begin{align}
&\Xi(\hat\omega^1,\hat\omega^2) = \sum_{k=1}^K \xi^k(\pi(\hat\omega^2))\1_{Z^k}(\hat\omega^1),\quad C_t(\hat\omega^1,\hat\omega^2) = \sum_{l=1}^L \alpha_t^l(\hat\omega^2)\1_{H^l}(\hat\omega^1),\\
&\Upsilon(\hat\omega^1,\hat\omega^2) = \sum_{m=1}^M \upsilon^m(\hat\omega^2)\1_{E^m}(\hat\omega^1) ,\quad\text{for }(\hat\omega^1,\hat\omega^2)\in\hat\Omega\times\hat\Omega,
\end{align}
where $(Z^k)_k,(H^l)_l,(E^m)_m$ are each disjoint partitions of $\hat\Omega$, and $\xi^k\in L^2(\Omega,\c G\lor\c F^W_t,\P;\R^d)$, $\alpha^l\in\c A_t$ and $\upsilon^m\in\c V_t$. Let us now consider the processes $\hat I^{t,\alpha^l_t}$ given by \eqref{eq randomised poisson point process}, and the corresponding $I^{t,\alpha^l_t}$ which $\b F^{\hat I^{t,\alpha^l_t}}$-identifies to $\hat I^{t,\alpha^l_t}$ from \cref{corollary I given bar I is well defined}. We see that then $\sum_{l=1}^L I^{t,\alpha^l_t}\1_{H^l}$ by construction $\b F^{\hat I^{t,\gamma_t}}$-identifies to 
$\hat I^{t,\gamma_t}$ and thus again by \cref{corollary I given bar I is well defined} is indistinguishable from $I^{t,\gamma_t}$ under $\hat\P\sim\hat\P^\nu$.
Similarly, denoting the corresponding unique solutions to \eqref{eq sde randomised problem} with starting value $\xi^k$ by $X^{s,\xi^k,\alpha^l_t}$, we note that by strong uniqueness of solutions to \eqref{eq sde randomised problem} the processes $X^{s,\zeta,\gamma_t}$ and $\sum_{k=1}^K \sum_{l=1}^L X^{t,\xi^k,\alpha^l_t} \1_{Z^k\cap H^l}$ are indistinguishable under $\hat\P\sim\hat\P^\nu$. Finally, we also define the tilted probability measures $\hat\P^{\upsilon^m}\sim\hat\P$ belonging to $\upsilon^m$ by $\frac{d\hat\P^{\upsilon^m}}{d\hat\P}\big|_{\c F^{B,\mu,\hat\P}_u} \coloneqq L^{\upsilon^m}_u$ for $u\in [0,T]$, see \eqref{eq randomised control girsanov formula}, and note that by construction $L^\nu = \sum_{m=1}^M L^{\upsilon^m} \1_{E^m}$. All together, this allows us to decompose the LHS of \eqref{eq generalised randomised payoff decomposition} as follows
\begin{align}
    &\E^{\hat\P^\nu}\Big[g(\hat\P^{\nu,\c F^{B,\mu,\hat\P^\nu}_T}_{X^{t,\zeta,\gamma_t}_T},X^{t,\zeta,\gamma_t}_T) + \int_t^T f(s,\hat\P^{\nu,\c F^{B,\mu,\hat\P^\nu}_s}_{X^{t,\zeta,\gamma_t}_s},X^{t,\zeta,\gamma_t}_s,I^{t,\gamma_t}_s) ds\Big| \c F^{B,\mu,\hat\P^\nu}_t\Big]\\
    &= \sum_{k=1}^K \sum_{l=1}^L \sum_{m=1}^M \1_{Z^k\cap H^l\cap E^m} \E^{\hat\P^{\upsilon^m}}\Big[g(\hat\P^{\upsilon^m,\c F^{B,\mu,\hat\P^{\upsilon^m}}_T}_{X^{t,\xi^k,\alpha_t^l}_T},X^{t,\xi^k,\alpha_t^l}_T)\\
    &\hspace{145pt} + \int_t^T f(s,\hat\P^{\upsilon^m,\c F^{B,\mu,\hat\P^{\upsilon^m}}_s}_{X^{t,\xi^k,\alpha_t^l}_s},X^{t,\xi^k,\alpha_t^l}_s,I^{t,\alpha_t^l}_s) ds\Big| \c F^{B,\mu,\hat\P^{\upsilon^m}}_t\Big].
    \label{eq lemma 7.2 decompose randomised payoff for simple functions}
\end{align}

For each summand, we note that since $\upsilon^m\in\c V_t$, by construction $\xi^k,\alpha^l_t,W,B^t,\mu^t$ are $\hat\P^{\upsilon^m}$ independent of $\c F_t^{B,\mu,\hat\P^{\upsilon^m}}$.
Thus from the strong uniqueness of solutions to \eqref{eq sde randomised problem} we can deduce that $X^{t,\xi^k,\alpha_t^l}$ and $(\hat\P^{\upsilon^m,\c F^{B,\mu,\hat\P^{\upsilon^m}}_s}_{X^{t,\xi^k,\alpha_t^l}_s})_{s\in [t,T]}$ are $\hat\P^{\upsilon^m}$-independent of $\c F^{B,\mu,\hat\P^{\upsilon^m}}_t$. For this, one can e.g.\@ note that the standard construction of a solution to \eqref{eq sde randomised problem} using a Picard iteration leads to a process $\hat\P^{\upsilon^m}$-independent of $\c F^{B,\mu,\hat\P^{\upsilon^m}}_t$ in every step and thus also in the limit. Therefore
\begin{align}
    &\E^{\hat\P^{\upsilon^m}}\Big[g(\hat\P^{\upsilon^m,\c F^{B,\mu,\hat\P^{\upsilon^m}}_T}_{X^{t,\xi^k,\alpha_t^l}_T},X^{t,\xi^k,\alpha_t^l}_T) + \int_t^T f(s,\hat\P^{\upsilon^m,\c F^{B,\mu,\hat\P^{\upsilon^m}}_s}_{X^{t,\xi^k,\alpha_t^l}_s},X^{t,\xi^k,\alpha_t^l}_s,I^{t,\alpha_t^l}_s) ds\Big| \c F^{B,\mu,\hat\P^{\upsilon^m}}_t\Big]\\
    &= \E^{\hat\P^{\upsilon^m}}\Big[g(\hat\P^{\upsilon^m,\c F^{B,\mu,\hat\P^{\upsilon^m}}_T}_{X^{t,\xi^k,\alpha_t^l}_T},X^{t,\xi^k,\alpha_t^l}_T) + \int_t^T f(s,\hat\P^{\upsilon^m,\c F^{B,\mu,\hat\P^{\upsilon^m}}_s}_{X^{t,\xi^k,\alpha_t^l}_s},X^{t,\xi^k,\alpha_t^l}_s,I^{t,\alpha_t^l}_s) ds\Big]\\
    &= J^{\c R}(s,\xi^k,\alpha_t^l,\upsilon^m),\quad \hat\P\text{-a.s.},
\end{align}
which together with \eqref{eq lemma 7.2 decompose randomised payoff for simple functions}, shows that the claim \eqref{eq generalised randomised payoff decomposition} holds for $\hat\omega\mapsto (\Xi(\hat\omega,\cdot),C_t(\hat\omega,\cdot),\Upsilon(\hat\omega,\cdot))$ only taking finitely many values.
\item Let us now consider general $(\Xi,C_t,\Upsilon)$. We first note that due to $\rho < 1$, 
and since $\Upsilon$ is essentially bounded, 
we can find a sequence $(\Xi^n,C_t^n,\Upsilon^n)_n$ such that $(\Xi^n,C_t^n,\Upsilon^n) \to (\Xi,C_t,\Upsilon)$ in $L^2\times L^\infty\times L^\infty$ and a.s., such that $\sup_{\upsilon\in\c V}\E^{\hat\P^\upsilon\otimes\hat\P^\upsilon}[|\Xi^n|^2] < \infty$ and each $\hat\omega\mapsto (\Xi^n(\hat\omega,\cdot),C_t^n(\hat\omega,\cdot),\Upsilon^n(\hat\omega,\cdot))$ only takes finitely many values. 
Correspondingly we define $\zeta^n,\gamma_t^n,\nu^n$ by $(\zeta^n,\gamma_t^n,\nu^n)(\hat\omega)\coloneqq (\Xi^n,C_t^n,\Upsilon^n)(\hat\omega,\hat\omega)$, which by construction then satisfy $(\zeta^n,\gamma_t^n,\nu^n)\to (\zeta,\gamma_t,\nu)$ in $L^2\times L^\infty\times L^\infty$ and a.s., and $\sup_{\upsilon\in\c V}\E^{\hat\P^\upsilon}[|\zeta^n|^2] < \infty$. 
Thus by \cref{lemma 7.2 step i}, we already know that \eqref{eq generalised randomised payoff decomposition} holds for each $(\Xi^n,C_t^n,\Upsilon^n)$ and $(\zeta^n,\gamma_t^n,\nu^n)$, and our task now will be showing that \eqref{eq generalised randomised payoff decomposition} still holds in the limit, for which it is sufficient to show the convergence of the left- and right-hand sides as $n\to\infty$.

To this end, for the right-hand side of \eqref{eq generalised randomised payoff decomposition}, we note that $J^{\c R}:[0,T]\times L^2(\Omega,\c G\lor\c F^W_t,\P;\R^d)\times\c A_t\times\c V\to \R$ is continuous in $\xi,\alpha_t,\nu$ by \cref{lemma J R continuous}, and thus
\[
\lim_{n\to\infty} J^{\c R}(t,\Xi^n(\hat\omega,\cdot),C_t^n(\hat\omega,\cdot),\Upsilon^n(\hat\omega,\cdot))
= J^{\c R}(t,\Xi(\hat\omega,\cdot),C_t(\hat\omega,\cdot),\Upsilon(\hat\omega,\cdot)),
\qquad\text{for $\hat\P$-a.a.\@ }\hat\omega\in\hat\Omega.
\]
For the left-hand side of \eqref{eq generalised randomised payoff decomposition}, we introduce $\bar\nu\coloneqq \nu\lor\sup_n \nu^n \in\c V$. Then using the conditional Jensen's inequality, we obtain that
\begin{align}
&\E^{\hat\P^{\bar\nu}}\bigg[\bigg|\E^{\hat\P^{\nu^n}}\Big[g(\hat\P^{\nu^n,\c F^{B,\mu,\hat\P^{\nu^n}}_T}_{X^{t,\zeta^n,\gamma_t^n}_T},X^{t,\zeta^n,\gamma_t^n}_T) + \int_t^T f(s,\hat\P^{\nu^n,\c F^{B,\mu,\hat\P^{\nu^n}}_s}_{X^{t,\zeta^n,\gamma_t^n}_s},X^{t,\zeta^n,\gamma_t^n}_s,I^{t,\gamma_t^n}_s) ds\Big| \c F^{B,\mu,\hat\P^{\nu^n}}_t\Big]\\
&\qquad - \E^{\hat\P^{\nu}}\Big[g(\hat\P^{\nu,\c F^{B,\mu,\hat\P^{\nu}}_T}_{X^{t,\zeta,\gamma_t}_T},X^{t,\zeta,\gamma_t}_T) + \int_t^T f(s,\hat\P^{\nu,\c F^{B,\mu,\hat\P^{\nu}}_s}_{X^{t,\zeta,\gamma_t}_s},X^{t,\zeta,\gamma_t}_s,I^{t,\gamma_t}_s) ds\Big| \c F^{B,\mu,\hat\P^{\nu}}_t\Big]\bigg| \bigg]\\
&\leq \E^{\hat\P^{\bar\nu}}\bigg[\bigg|\frac{L^{\nu^n}_T}{L^{\bar\nu}_T} \Big(g(\hat\P^{\bar\nu,\c F^{B,\mu,\hat\P^{\bar\nu}}_T}_{X^{t,\zeta^n,\gamma_t^n}_T},X^{t,\zeta^n,\gamma_t^n}_T) + \int_t^T f(s,\hat\P^{\bar\nu,\c F^{B,\mu,\hat\P^{\bar\nu}}_s}_{X^{t,\zeta^n,\gamma_t^n}_s},X^{t,\zeta^n,\gamma_t^n}_s,I^{t,\gamma_t^n}_s) ds\Big)\\
&\qquad - \frac{L^\nu_T}{L^{\bar\nu}_T} \Big(g(\hat\P^{\bar\nu,\c F^{B,\mu,\hat\P^{\bar\nu}}_T}_{X^{t,\zeta,\gamma_t}_T},X^{t,\zeta,\gamma_t}_T) + \int_t^T f(s,\hat\P^{\bar\nu,\c F^{B,\mu,\hat\P^{\bar\nu}}_s}_{X^{t,\zeta,\gamma_t}_s},X^{t,\zeta,\gamma_t}_s,I^{t,\gamma_t}_s) ds\Big)\bigg|\bigg],
\end{align}
and thus repeating the arguments in the proof of \cref{lemma J R continuous} using the estimate \eqref{eq extended randomised state dynamics basic estimate}, we conclude that
\begin{align}
&\E^{\hat\P^{\nu^n}}\Big[g(\hat\P^{\nu^n,\c F^{B,\mu,\hat\P^{\nu^n}}_T}_{X^{t,\zeta^n,\gamma_t^n}_T},X^{t,\zeta^n,\gamma_t^n}_T) + \int_t^T f(s,\hat\P^{\nu^n,\c F^{B,\mu,\hat\P^{\nu^n}}_s}_{X^{t,\zeta^n,\gamma_t^n}_s},X^{t,\zeta^n,\gamma_t^n}_s,I^{t,\gamma_t^n}_s) ds\Big| \c F^{B,\mu,\hat\P^{\nu^n}}_t\Big]\\
&\to \E^{\hat\P^{\nu}}\Big[g(\hat\P^{\nu,\c F^{B,\mu,\hat\P^{\nu}}_T}_{X^{t,\zeta,\gamma_t}_T},X^{t,\zeta,\gamma_t}_T) + \int_t^T f(s,\hat\P^{\nu,\c F^{B,\mu,\hat\P^{\nu}}_s}_{X^{t,\zeta,\gamma_t}_s},X^{t,\zeta,\gamma_t}_s,I^{t,\gamma_t}_s) ds\Big| \c F^{B,\mu,\hat\P^{\nu}}_t\Big]\text{ in }L^1(\hat\P^{\bar\nu}).
\end{align}
Thus from the uniqueness of limits, we conclude that \eqref{eq generalised randomised payoff decomposition} also holds for general $(\Xi,C_t,\Upsilon)$, which completes the proof.
\end{enumerate}
\end{proof}

This allows us now to obtain the following more general connection between the mean-field control problem and the BSDE in \cref{theorem value function bsde characterisation}.
    
\begin{theorem}\label{theorem value function bsde characterisation}
    Let $t\in [0,T]$ and $\xi\in L^2(\Omega,\c G\lor\c F^W_t,\P;\R^d)$. Further let $\alpha_t\in\c A_t\subseteq \c A_T$ and let $(Y^{t,\xi,\alpha_t},Z^{t,\xi,\alpha_t},U^{t,\xi,\alpha_t},K^{t,\xi,\alpha_t})$ be the unique minimal solution to \eqref{eq constrained bsde value function} from \cref{lemma bsde minimal solution existence}. Then for any $s\in [t,T]$, it holds that
    \[
    Y^{t,\xi,\alpha_t}_s = V(s,\hat\P^{\nu,\c F^{B,\mu,\hat\P^\nu}_s}_{X^{t,\xi,\alpha_t}_s}),\qquad \hat\P\text{-a.s.}
    \]
\end{theorem}

\begin{proof}
We start by decomposing $(X_s^{t,\xi,\alpha_t},I_s^{t,\alpha_t})$. Since $X_s^{t,\xi,\alpha_t}$ and $I_s^{t,\alpha_t}$ are both $\c G\lor\c F^{W,B,\mu,\hat\P}_s$-measurable, by e.g.\@ \cref{proposition decompose filtrations of predictable processes} there exist $\c F^{B,\mu}_s\otimes (\c G\lor\c F^W_s)$-measurable $\zeta:\hat\Omega\times\hat\Omega\to\R^d$ and $C_s:\hat\Omega\times\hat\Omega\to A$ such that
\[
(X^{t,\xi,\alpha_t}_s,I^{t,\alpha_t}_s)(\hat\omega) = (\zeta,C_s)(\hat\omega,\hat\omega), \quad \text{for $\hat\P$-a.a.\@ }\hat\omega\in\hat\Omega.
\]
Note that by construction $\hat\P^{\c F^{B,\mu,\hat\P}_s}_{X^{t,\xi,\alpha_t}_s}(\hat\omega) = \hat\P_{\zeta(\hat\omega,\cdot)}$, for $\hat\P$-a.a.\@ $\hat\omega\in\hat\Omega$, since $\c F^{B,\mu,\hat\P}_t$ and $\c G\lor\c F^{W,\hat\P}_t$ are independent under $\hat\P$.
Hence, from \cref{proposition equivalence v v_t randomised value function,theorem equivalence original and randomised value function}, we now see that
\[
V(s,\hat\P^{\c F^{B,\mu,\hat\P}_s}_{X^{t,\xi,\alpha_t}_s}(\hat\omega))
= V^{\c R}(s,\zeta(\hat\omega,\cdot),C_s(\hat\omega,\cdot))
= \sup_{\nu\in\c V_s} J^{\c R} (s,\zeta(\hat\omega,\cdot),C_s(\hat\omega,\cdot),\nu),
\ \text{for $\hat\P$-a.a.\@ }\hat\omega\in\hat\Omega,
\]

At the same time, we note that \cref{lemma bsde minimal solution existence} (for $r=T$) together with the flow property in \cref{lemma flow property} implies that $\hat\P$-a.s.
\begin{align}
    Y^{t,\xi,\alpha_t}_s &= \esssup_{\nu\in\c V} \E^{\hat\P^\nu}\Big[g(\hat\P^{\nu,\c F^{B,\mu,\hat\P^\nu}_T}_{X^{t,\xi,\alpha_t}_T},X^{t,\xi,\alpha_t}_T) + \int_s^T f(u,\hat\P^{\nu,\c F^{B,\mu,\hat\P^\nu}_u}_{X^{t,\xi,\alpha_t}_u},X^{t,\xi,\alpha_t}_u,I^{t,\alpha_t}_u) du\Big| \c F^{B,\mu,\hat\P^\nu}_s\Big]\\
    &= \esssup_{\nu\in\c V} \E^{\hat\P^\nu}\Big[g(\hat\P^{\nu,\c F^{B,\mu,\hat\P^\nu}_T}_{X^{s,X^{t,\xi,\alpha_t}_s,I^{t,\alpha_t}_s}_T},X^{s,X^{t,\xi,\alpha_t}_s,I^{t,\alpha_t}_s}_T)\\
    &\hspace{65pt}+ \int_s^T f(u,\hat\P^{\nu,\c F^{B,\mu,\hat\P^\nu}_u}_{X^{s,X^{t,\xi,\alpha_t}_s,I^{t,\alpha_t}_s}_u},X^{s,X^{t,\xi,\alpha_t}_s,I^{t,\alpha_t}_s}_u,I^{t,\alpha_t}_u) du\Big| \c F^{B,\mu,\hat\P^\nu}_s\Big].
\end{align}
In particular, the claim now follows if we can show that $\hat\P$-a.s.
\begin{align}
&\sup_{\nu\in\c V_s} J^{\c R} (s,\zeta(\hat\omega,\cdot),C_s(\hat\omega,\cdot),\nu)\\
&= \esssup_{\nu\in\c V} \E^{\hat\P^\nu}\Big[g(\hat\P^{\nu,\c F^{B,\mu,\hat\P^\nu}_T}_{X^{s,X^{t,\xi,\alpha_t}_s,I^{t,\alpha_t}_s}_T},X^{s,X^{t,\xi,\alpha_t}_s,I^{t,\alpha_t}_s}_T)\\
&\hspace{65pt}+ \int_s^T f(u,\hat\P^{\nu,\c F^{B,\mu,\hat\P^\nu}_u}_{X^{s,X^{t,\xi,\alpha_t}_s,I^{t,\alpha_t}_s}_u},X^{s,X^{t,\xi,\alpha_t}_s,I^{t,\alpha_t}_s}_u,I^{t,\alpha_t}_u) du\Big| \c F^{B,\mu,\hat\P^\nu}_s\Big].
\end{align}

To this end, we note that by the estimate \eqref{eq randomised state dynamics basic estimate}, we have $\sup_{\nu\in\c V}\E^{\hat\P^\nu}[|\zeta|^2] = \sup_{\nu\in\c V}\E^{\hat\P^\nu}[|X^{t,\xi,\alpha_t}_s|^2] < \infty$.
Hence "$\leq$" follows from \cref{lemma generalised randomised payoff decomposition} since $\c V_s\subseteq\c V$. On the other hand, for "$\geq$" it suffices to notice that we can decompose every $\nu\in\c V$ by \cref{proposition decompose filtrations of predictable processes} into an $\c F^{B,\mu}_s\otimes Pred(\b F^{B^s,\mu^s})\otimes\c B(\c A_T)$-measurable process $\Upsilon$ such that $\nu_r(\hat\omega,\alpha) = \Upsilon_r(\hat\omega,\hat\omega,\alpha)$ for all $\hat\omega\in\hat\Omega,\alpha\in\c A_T,r\in [0,T]$, as this direction now follows from \cref{lemma generalised randomised payoff decomposition} since $\Upsilon(\hat\omega,\cdot)\in\c V_s$ for $\hat\P$-a.a.\@ $\hat\omega\in\hat\Omega$.
\end{proof}

Finally, we can now conclude the following randomised dynamic programming principle (DPP).

\begin{theorem}\label{theorem randomised dpp}
    Let $(t,m) \in [0,T]\times \c P_2(\R^d)$. Then for any choice of $\alpha_t\in\c A_t\subseteq\c A_T$ and $\xi\in L^2(\Omega,\c G\lor\c F^W_t,\P;\R^d)$ with $\hat\P_\xi = m$, it holds for $s\in [t,T]$ that
    \begin{align}
    V(t,m) = \sup_{\nu\in \c V} \E^{\hat\P^\nu}\Big[ V(s,\hat\P^{\nu,\c F^{B,\mu,\hat\P^\nu}_s}_{X^{t,\xi,\alpha_t}_s}) + \int_t^s f(r,\hat\P^{\nu,\c F^{B,\mu,\hat\P^\nu}_r}_{X^{t,\xi,\alpha_t}_r},X^{t,\xi,\alpha_t}_r,I^{t,\alpha_t}_r) dr \Big].
    \label{eq randomised dpp}
    \end{align}
\end{theorem}
\begin{proof}
This follows directly from \cref{lemma bsde minimal solution existence,theorem value function bsde characterisation}.
\end{proof}

\begin{remark}
Assuming that the value function $V$ is sufficiently regular, this approach leads to the standard (non-randomised) DPP for mean-field control problems, which has been well-studied in the literature, with \cite{djete_mckeanvlasov_2022,djete_mckeanvlasov_2022-1} providing one of the most comprehensive studies. Specifically, we can interpret the right-hand side in \eqref{eq randomised dpp} as a new control problem with running reward $f$ and terminal reward $V$. Thus, as long as $V$ satisfies the assumptions on $g$ in \cref{assumptions reward functional}, we can apply the same theory developed in this paper again to this new problem. Consequently, by \cref{theorem equivalence original and randomised value function}, we then obtain
\begin{align}
V(t,m)
&= \sup_{\nu\in \c V} \E^{\hat\P^\nu}\Big[ V(s,\hat\P^{\nu,\c F^{B,\mu,\hat\P^\nu}_s}_{X^{t,\xi,\alpha_t}_s}) + \int_t^s f(r,\hat\P^{\nu,\c F^{B,\mu,\hat\P^\nu}_r}_{X^{t,\xi,\alpha_t}_r},X^{t,\xi,\alpha_t}_r,I^{t,\alpha_t}_r) dr \Big]\\
&= \sup_{\alpha\in\c A} \E^\P\Big[ V(s,\P^{\c F^{B,\P}_s}_{X^{t,\xi,\alpha}_s}) + \int_t^s f(r,\P^{\c F^{B,\P}_r}_{X^{t,\xi,\alpha}_r},X^{t,\xi,\alpha}_r,\alpha_r) dr \Big].
\end{align}
\end{remark}


\appendix

\section{Proof of \cref{theorem equivalence original and randomised value function}}
\label{section proof of equivalence randomised non randomised formulation}

Our goal in this section is proving that $V = V^{\c R}$. For this let us fix $t\in [0,T]$ and $\alpha_t\in\c A_t$. Furthermore, let us fix also the family $(\lambda_s)_{s\in [0,T]}$ satisfying \cref{assumptions lambda family}. As already outlined in \cref{theorem equivalence original and randomised value function}, we will obtain as a side result that $V^{\c R}$ is independent of specific choice of $\alpha_t\in\c A_t$, which we will come into play later again in the upcoming dynamic programming result in \cref{section randomised dpp}. In what follows, we will prove the two direction $V\leq V^{\c R}$ and $V^{\c R}\leq V$ separately. The general structure of the proofs is similar to \cite{bandini_backward_2018,bayraktar_randomized_2018}, albeit the detailed arguments are naturally different in our mean-field setting due to the presence of the mean-field interaction and the common noise.

A crucial tool will be the following \cref{proposition randomised setting is independent of the extension}, which shows that the randomised value function $V^{\c R}$ is independent of the specific choice of the extension of the probability space $(\Omega,\c F,\P)$. This will allow us to simplify the upcoming proofs by considering suitable extensions of $(\Omega,\c F,\P)$ as needed for each direction. 

\begin{lemma}\label{proposition randomised setting is independent of the extension}
    The randomised value function $V^{\c R}$ does not depend on the specific extension $(\hat\Omega,\hat{\c F},\hat\P)$ of $(\Omega,\c F,\P)$.
\end{lemma}
\begin{proof}
Let $t\in [0,T]$, $\xi\in L^2(\Omega,\c G\lor\c F^W_t,\P;\R^d)$ and $\alpha_t\in\c A_t$. Further let $(\lambda_s)_{s\in [0,T]}$ satisfy \cref{assumptions lambda family}, and let $(\hat\Omega^i,\hat{\c F}^i,\hat\P^i)$, $i\in\{1,2\}$ be both extensions of $(\Omega,\c F,\P)$, each carrying a Poisson point measure $\mu^i$ with intensity $\lambda_s(d\alpha)ds$ independent of $\c G^i,W^i,B^i$.
We denote the respective canonical projections by $\pi^i:(\hat\Omega^i,\hat{\c F}^i)\to(\Omega,\c F)$. By the extension property, we have then that $\c L^{\hat\P^1}(\pi^1) = \c L^{\hat\P^2}(\pi^2)$. Therefore, by introducing $\pi^i|_{\c G^i}:(\hat\Omega^i,\hat{\c G}^i)\to (\Omega,\c G)$, $i\in \{1,2\}$, we observe that also $\c L^{\hat\P^1}(\pi^1|_{\c G^1},W^1,B^1) = \c L^{\hat\P^2}(\pi^2|_{\c G^2},W^2,B^2)$.
Finally, using the independence of $\c G^i,W^i,B^i$ from $\mu^i$ and thus also from the corresponding $\hat I^{t,\alpha_t,i}$ defined by \eqref{eq randomised poisson point process}, together with $\c L^{\hat\P^1}(\mu^1) = \c L^{\hat\P^2}(\mu^2)$ this shows that
\[
\c L^{\hat\P^1}(\pi^1|_{\c G^1},W^1,B^1,\hat I^{t,\alpha_t,1},\mu^1) = \c L^{\hat\P^2}(\pi^1|_{\c G^2},W^2,B^2,\hat I^{t,\alpha_t,2},\mu^2).
\]
Then \cref{corollary I given bar I is well defined} shows that the corresponding càdlàg, $\b F^{W^i}\lor\c G^i\lor\c F^{\mu^{t,i}}_T$-predictable and $\b F^{W^i,\mu^{t,i}}\lor\c G^i$-progressive $I^{t,\alpha_t,i}$ which $\b F^{\hat I^{t,\alpha_t,i}}$-identify to $\hat I^{t,\alpha_t,i}$, $i\in\{1,2\}$ are unique (up to indistinguishability) and moreover satisfy
\[
\c L^{\hat\P^1}(\pi^1|_{\c G^1},W^1,B^1,\hat I^{t,\alpha_t,1},I^{t,\alpha_t,1},\mu^1) = \c L^{\hat\P^2}(\pi^2|_{\c G^2},W^2,B^2,\hat I^{t,\alpha_t,2},I^{t,\alpha_t,2},\mu^2).
\]
Using the weak uniqueness of solutions to the SDE \eqref{eq sde randomised problem}, we hence obtain for the corresponding unique solutions $X^{t,\xi^i,\alpha_t,i}$ to \eqref{eq sde randomised problem}, $i\in\{1,2\}$, that
\begin{align}
\c L^{\hat\P^1}(X^{t,\xi^1,\alpha_t,1},\mu^1,I^{t,\alpha_t,1},(\hat\P^{1,\c F^{B^1,\mu^1,\hat\P^1}_s}_{X^{t,\xi^1,\alpha_t,1}_s})_s) = \c L^{\hat\P^2}(X^{t,\xi^2,\alpha_t,2},\mu^2,I^{t,\alpha_t,2},(\hat\P^{2,\c F^{B^2,\mu^2,\hat\P^2}_s}_{X^{t,\xi^2,\alpha_t,2}_s})_s).
\label{eq proposition 4.9 law equivalence of state processes}
\end{align}

Now let $\nu^1\in\c V^1$ be a randomised control and $\hat\P^{1,\nu^1}\sim\hat\P^1$ be the corresponding tilted probability measure. Then using \cite[Lemma 4.2]{bandini_backward_2018} and \cref{proposition canonical representation of predictable processes} we can represent it in the following form, for all $(s,\hat\omega^1,a)\in [0,T]\times\hat\Omega^1\times\c A_T$,
\begin{align}
\label{eq proposition 4.9 representation nu 1}
    \nu^1_s(\hat\omega^1,a) &= \nu^{(0)}_s(B^1(\hat\omega^1),a) \ \1_{[0,\tau^1_1(\hat\omega^1)]}(s)\\
    &\quad+ \sum_{k\geq 1} \nu^{(k)}_s(B^1(\hat\omega^1),(\tau^1_1(\hat\omega^1),\alpha^1_1(\hat\omega^1)),\dots,(\tau^1_k(\hat\omega^1),\alpha^1_k(\hat\omega^1)),a) \ \1_{(\tau^1_k(\hat\omega^1), \tau^1_{k+1}(\hat\omega^1)]}(s),
\end{align}
for $Pred(\mathbf C^n)\otimes \c B(([0,\infty)\times \c A_T)^k)\otimes\c B(\c A_T)$-$\c B((0,\infty))$-measurable processes $\nu^{(k)}:C^n\times ([0,\infty)\times \c A_T)^k\times \c A_T\to (0,\infty)$, for all $k\geq 0$, where $(\tau^1_k,\alpha^1_k)_{k\geq 1}$ are the events of $\mu^1$. This representation now allows us to transfer the control $\nu^1$ to $(\hat\Omega^2,\hat{\c F}^2,\hat\P^2)$ by defining, for all $(s,\hat\omega^1,a)\in [0,T]\times\hat\Omega^1\times\c A_T$,
\begin{align}
\label{eq proposition 4.9 representation nu 2}
    \nu^2_s(\hat\omega^2,a) &= \nu^{(0)}_s(B^2(\hat\omega^2),a) \ \1_{[0,\tau^2_1(\hat\omega^2)]}(s)\\
    &\quad + \sum_{k\geq 1} \nu^{(k)}_s(B^2(\hat\omega^2),(\tau^2_1(\hat\omega^2),\alpha^2_1(\hat\omega^2)),\dots,(\tau^2_k(\hat\omega^2),\alpha^2_k(\hat\omega^2)),a) \ \1_{(\tau^2_k(\hat\omega^1), \tau^2_{k+1}(\hat\omega^2)]}(s),
\end{align}
where $(\tau^2_k,\alpha^2_k)_{k\geq 1}$ are the events of $\mu^2$. We note that by \cref{proposition canonical representation of predictable processes}, $\nu^2$ is $Pred(\b F^{B^2,\mu^2})\otimes\c B(\c A_T)$-measurable and thus $\nu^2\in\c V^2$. Furthermore, the representations \eqref{eq proposition 4.9 representation nu 1,eq proposition 4.9 representation nu 2} together with \eqref{eq proposition 4.9 law equivalence of state processes} show that
\begin{align}
    \c L^{\hat\P^1}(X^{t,\xi^1,\alpha_t,1},\mu^1,I^{t,\alpha_t,1},(\hat\P^{1,\c F^{B^1,\mu^1,\hat\P^1}_s}_{X^{t,\xi^1,\alpha_t,1}_s})_s,\nu^1) = \c L^{\hat\P^2}(X^{t,\xi^2,\alpha_t,2},\mu^2,I^{t,\alpha_t,2},(\hat\P^{2,\c F^{B^2,\mu^2,\hat\P^2}_s}_{X^{t,\xi^2,\alpha_t,2}_s})_s,\nu^2).
\end{align}
Finally, this allows us to conclude that
\begin{align}
    J^{\c R,1}(t,\xi,\alpha_t,\nu^1) &= \E^{\hat\P^1}\Big[L^{\nu^1}_T g(\hat\P_{1,X^{t,\xi^1,\alpha_t,1}_T}^{\c F^{B^1,\mu^1,\hat\P^1}_T},X^{t,\xi^1,\alpha_t,1}_T) + \int_t^T f(s,\hat\P_{1,X^{t,\xi^1,\alpha_t,1}_s}^{\c F^{B^1,\mu^1,\hat\P^1}_s},X^{t,\xi,\alpha_t,1}_s,I^{t,\alpha_t,1}_s) ds \Big]\\
    &= \E^{\hat\P^2}\Big[L^{\nu^2}_T g(\hat\P_{2,X^{t,\xi^1,\alpha_t,2}_T}^{\c F^{B^2,\mu^2,\hat\P^2}_T},X^{t,\xi^2,\alpha_t,2}_T) + \int_t^T f(s,\hat\P_{2,X^{t,\xi^2,\alpha_t,2}_s}^{\c F^{B^2,\mu^2,\hat\P^2}_s},X^{t,\xi,\alpha_t,2}_s,I^{t,\alpha_t,2}_s) ds \Big]\\
    &= J^{\c R,2}(t,\xi,\alpha_t,\nu^2),
\end{align}
where $L^{\nu^i}_T$ is given by \eqref{eq randomised control girsanov formula}.
Since $\nu^1\in\c V^1$ was arbitrary, this shows that
\[
V^{\c R,1}(t,\xi,\alpha_t) = \sup_{\nu^1\in{\c V}^1} J^{\c R,1}(t,\xi,\alpha_t,\nu^1) \leq \sup_{\nu^2\in{\c V}^2} J^{\c R,2}(t,\xi,\alpha_t,\nu^2) = V^{\c R,2}(t,\xi,\alpha_t).
\]
By the same arguments we can also show that $V^{\c R,2}(t,\xi,\alpha_t)\leq V^{\c R,1}(t,\xi,\alpha_t)$, and thus we conclude that $V^{\c R,1} \equiv V^{\c R,2}$.
\end{proof}

\subsection{Proof of $V\leq V^{\c R}$}\label{subsection V leq V^R}

This direction relies on the result in \cite[Proposition A.1]{bandini_randomisation_2016_extended_arxiv_version}\footnote{\cite{bandini_randomisation_2016_extended_arxiv_version} is the extended version of the paper \cite{bandini_backward_2018}, and \cite[Proposition A.1]{bandini_randomisation_2016_extended_arxiv_version} is a more general version of \cite[Proposition 4.1]{bandini_backward_2018}.} which allows us to approximate general progressive controls with suitable point processes having both full support and a bounded intensity. Note that this requires taking the view point of the admissible control set as the set of $\b F^B$-predictable processes taking values in $\c A_T$ as described in \cref{section problem l2 formulation}, since this allows us to view the mean-field dynamics as a standard (albeit infinite-dimensional) control problem, for which we are now able to use the above approximation result.

Let us fix $t\in [0,T]$, $\xi\in L^2(\Omega,\c G\lor\c F^W_t,\P;\R^d)$ and some $\alpha_t\in\c A_t$. Our goal is now to show that for any control $\alpha\in\c A$, it holds that $J(t,\xi,\alpha) \leq V^{\c R}(t,\xi,\alpha_0)$. To this end, let us fix some arbitrary $\alpha\in\c A$, then by \cref{proposition bar alpha given alpha well defined} there exists a $\hat\alpha\in\hat{\c A}$ which is $\b F^B$-identifying to $\alpha$, and furthermore by \cref{lemma alpha bar alpha uniqueness} this $\hat\alpha\in\hat{\c A}$ is unique up to modification.

Our first step will be constructing approximating $\c A_T$-valued Poisson point processes for the $\c A_T$-valued control process $\hat\alpha\in\hat{\c A}$.
We recall that these processes are supposed to take the role of $\hat I^{t,\alpha_t}$ in the randomised setting which we are going to construct, and thus we want them to start in $\alpha_t$ at time $t$. For this reason, we will first establish an approximation on $[t,T]$, and afterwards extend the underlying Poisson random measure to $(0,T]$.
We note that $(\hat\alpha_s)_{s\in [t,T]}$ is a $\b F^B$-progressive (in fact by the construction in \cref{proposition bar alpha given alpha well defined} it is even $\b F^B$-predictable), and that $\hat\alpha_s\in\c A_s\subseteq\c A_T$ for all $s\in [t,T]$. Thus, by \cite[Proposition A.1]{bandini_randomisation_2016_extended_arxiv_version}, there exists a suitable extension $(\Omega\times\Omega',(\c F\otimes\c F')^{\P\otimes\P'},\P\otimes\P')$ of our probability space $(\Omega,\c F,\P)$, 
such that for any $\delta > 0$ there exists a marked point process $(\tau^\delta_k,\alpha^\delta_k)_{k\geq 1}$ with the corresponding Poisson random measure $\mu^\delta = \sum_{k\geq 1} \delta_{(\tau^\delta_k,\alpha^\delta_k)}$ on $(t,T]\times\c A_T$ such that
\begin{enumerate}[label=(\roman*)]
    \item $\c F^{\mu^\delta,\P\otimes\P'}_T\subseteq \c F^{B,\P\otimes\P'}_T\lor\c F'$,%
\footnote{With slight abuse of notation, we identify $\c F'$ with its canonical extension to $\Omega\times\Omega'$.}
    so in particular $\mu^\delta$ is independent of $\c G$ and $W$ under $\P\otimes\P'$,%
\footnote{The independence follows from the explicit construction in \cite[Proposition A.1]{bandini_randomisation_2016_extended_arxiv_version} which uses only $\b F^{B,\P\otimes\P'}$ together with a fixed family of random variables defined on $(\Omega',\c F',\P')$, which is independent of $\delta$.}
    \item the process $\hat I^{t,\alpha_t,\delta}$ defined by
    \[
    \hat I^{t,\alpha_t,\delta}_s \coloneqq \alpha_t \1_{[t,\tau^\delta_1)} + \sum_{k\geq 1} \alpha^\delta_k \1_{[\tau^\delta_k,\tau^\delta_{k+1})} = \alpha_t + \int_{(t,s]}\int_{\c A_T} (\alpha - I^{t,\alpha_t,\delta}_{r-}) \mu^\delta(dr,d\alpha),\qquad s\in [t,T],
    \]
    satisfies $d_{\hat{\c A},[t,T]}^{\P\otimes\P'}(\hat\alpha,\hat I^{t,\alpha_t,\delta}) < \delta$,
    \item the $\b F^{B,\mu^\delta,\P\otimes\P'}$-compensator of $\mu^\delta$ is absolutely continuous with respect to $\lambda_s(d\alpha)ds$ and can be written as
    \[
    \nu^\delta_s(\alpha)\lambda_s(d\alpha)ds,
    \]
    where $\nu^\delta$ is $Pred(\b F^{B,\mu^\delta})\otimes\c B(\c A_T)$-measurable and satisfies
    \[
    0 < \inf_{[t,T]\times\Omega\times\Omega'\times\c A_T} \nu^\delta \leq \sup_{[t,T]\times\Omega\times\Omega'\times\c A_T} \nu^\delta < \infty,
    \]
    \item $\alpha^\delta_k\in\c A_{\tau^\delta_k}$, for all $k\geq 1$.
\end{enumerate}
Note that this statement is slightly different than \cite[Proposition A.1]{bandini_randomisation_2016_extended_arxiv_version}. At first glance, property (i) may seem new, but it directly follows from the explicit construction in \cite[Proposition A.1]{bandini_randomisation_2016_extended_arxiv_version} constructing $\mu^\delta$ from $\b F^{B,\P\otimes\P'}$ and a family of random variables defined on $(\Omega',\c F',\P')$, which is independent of $\delta$.
The more crucial difference, however, is that while the action set $A$ in \cite{bandini_randomisation_2016_extended_arxiv_version} is constant over time, our action sets $\c A_s$ vary with $s\in [t,T]$, which now leads to the adapted form of the intensity process in (iii) and the additional property (iv), necessary for consistency with the control randomisation introduced in \cref{section randomised problem}.
Fortunately, we can extend its proof to cover our use case with minimal changes. Our admissible controls $\hat\alpha\in\hat{\c A}$ satisfy additionally $\hat\alpha_s\in\c A_s$ for all $s\in [0,T]$, which allows us to choose the step function approximation $\bar\alpha = \sum_{n=0}^\infty \alpha_n \1_{[t_n,t_{n+1})}$ on \cite[pp.\@ 41-42]{bandini_randomisation_2016_extended_arxiv_version} such that each $\hat\alpha_n\in\c A_{t_n}\subseteq\c A_T$. We then replace the kernel $q^m$ on \cite[p.\@ 42]{bandini_randomisation_2016_extended_arxiv_version} with a family of kernels $(q^m_{s})_{s\in [0,T]}$, where each $q^m_{s}$ is constructed from $\lambda_{s}$ instead of $\lambda$ and thus supported on $\c A_s$ instead of $\c A_T$. The definition of $\beta^m_n$ on \cite[p.\@ 42]{bandini_randomisation_2016_extended_arxiv_version} is also modified to include the time, so $\beta^m_n\coloneqq q^m_{R^m_n}(\alpha_n,U^m_n)$. The estimates on \cite[pp.\@ 42-43]{bandini_randomisation_2016_extended_arxiv_version} remain valid, and we obtain a marked point process $\kappa^m$ with the compensator $\tilde\kappa^m(dt,d\alpha) = \sum_{n\geq 1} \1_{(S^m_n,R^m_n]}(t) q^m_t(\alpha_n,d\alpha)\lambda_{nm} dt$, replacing $q^m$ by $q^m_t$ in \cite[Lemma A.1]{bandini_randomisation_2016_extended_arxiv_version} and its proof. Finally, by replacing $\lambda(d\alpha)dt$ by $\lambda_t(d\alpha)dt$ in the intensity of the additional "small" intensity Poisson process $\pi^k$ on \cite[p.\@ 45]{bandini_randomisation_2016_extended_arxiv_version}, we can obtain the desired point measure $\mu^\delta = \kappa^m + \pi^k$ without further modifications.

To extend the Poisson random measure $\mu^\delta$ to $[0,T]\times\c A_T$, we will extend the probability space again by a probability space $(\Omega'',\c F'',\P'')$ on which a Poisson random measure $\mu$ on $(0,t]\times\c A_T$ with $\b F^{\mu^\delta}$-intensity $\lambda_s(d\alpha)ds$ lives. Hence we obtain $(\hat\Omega,\hat{\c F},\hat\Q) \coloneqq (\Omega\times\Omega'\times\Omega'',(\c F\otimes\c F'\otimes\c F'')^{\P\otimes\P'\otimes\P''},\P\otimes\P'\otimes\P'')$, such that $\mu^\delta \coloneqq \mu(\cdot \cap ((0,t]\times\c A_T)) + \mu^\delta(\cdot\cap((t,T]\times\c A_T))$ has the $\b F^{B,\mu^\delta}$-intensity $\nu^\delta_s(\alpha)\lambda_s(d\alpha)ds$, where we extend $\nu^\delta$ to $[0,t]\times\c A_T$ with $\nu^\delta|_{[0,t]\times\c A_T}\equiv 1$.
As before, with slight abuse of notation, we identify each sigma algebra $\c F$ with its canonical extension to the product space. 
Note that property (i) also extends to $\mu^\delta$ on $[0,T]$, that is, 
$\c F^{\mu^\delta,\hat\Q}_T \subseteq \c F^{B,\hat\Q}_T\lor\c F'\otimes\c F''$, and in particular $\mu^\delta$ on $[0,T]$ is again independent of $\c G$ and $W$ under $\hat\Q$.

Having defined approximations $(\hat I^{t,\alpha_t,\delta})_\delta$ to $\hat\alpha$, we will now also show that the corresponding state processes converge. We first note that by \cref{corollary I given bar I is well defined} for every $\delta > 0$, there exist an $\b F^{W,\mu^\delta}\lor\c G$-progressive, càdlàg $I^{t,\alpha_t,\delta}$ that $\b F^{I^{t,\alpha_t,\delta}}$-identifies to $\hat I^{t,\alpha_t,\delta}$. Moreover, by \cref{lemma alpha bar alpha uniqueness}, we have
\begin{align}
    d_{\c A,[t,T]}^{\hat\Q}(\alpha,I^{t,\alpha_t,\delta}) = d_{\hat{\c A},[t,T]}^{\hat\Q}(\hat\alpha,\hat I^{t,\alpha_t,\delta}) = d_{\hat{\c A},[t,T]}^{\P\otimes\P'}(\hat\alpha,\hat I^{t,\alpha_t,\delta}) < \delta,\quad\text{for all }\delta>0.
    \label{eq proof V leq V^R alpha I delta close}
\end{align}
Correspondingly, we now define $X^{t,\xi,\alpha_t,\delta}$ as the unique solution to the SDE
\begin{align}
dX^{t,\xi,\alpha_t,\delta}_s
&= b(s,\hat\Q^{\c F^{B,\mu^\delta,\hat\Q}_s}_{X^{t,\xi,\alpha_t,\delta}_s},X^{t,\xi,\alpha_t,\delta}_s,I^{t,\alpha_t,\delta}_s) ds + \sigma(s,\hat\Q^{\c F^{B,\mu^\delta,\hat\Q}_s}_{X^{t,\xi,\alpha_t,\delta}_s},X^{t,\xi,\alpha_t,\delta}_s,I^{t,\alpha_t,\delta}_s) dW_s\\
&\qquad+ \sigma^0(s,\hat\Q^{\c F^{B,\mu^\delta,\hat\Q}_s}_{X^{t,\xi,\alpha_t,\delta}_s},X^{t,\xi,\alpha_t,\delta}_s,I^{t,\alpha_t,\delta}_s) dB_s,\qquad X^{t,\xi,\alpha_t,\delta}_t = \xi.
\label{eq randomised state dynamics X t xi alpha_t delta}
\end{align}
We next show that $X^{t,\xi,\alpha_t,\delta}$ converges to the state process $X^{t,\xi,\alpha}$ for our control $\alpha\in\c A$ in the original control problem, as $\delta\to 0$. To this end, we note that we can rewrite the conditional distributions occurring in \eqref{eq mfc state dynamics,eq randomised state dynamics X t xi alpha_t delta} as follows: Since $X^{t,\xi,\alpha}$ is $\c G\lor\b F^{W,B,\hat\Q}$-progressive and by construction $\c G,W,B$ and $\c F'\otimes\c F''$ are independent under $\hat\Q$, we see that $\hat\Q^{\c F^{B,\hat\Q}_s}_{X^{t,\xi,\alpha}_s} = \hat\Q^{\c F^{B,\hat\Q}_T\lor\c F'\otimes\c F''}_{X^{t,\xi,\alpha}_s}$, $\hat\Q$-a.s., for all $s\in [t,T]$. For the $\c G\lor\b F^{W,B,\mu^\delta,\hat\Q}$-progressive process $X^{t,\xi,\alpha_t,\delta}$, we note that by property (i) of $\mu^\delta$, we have $\c F^{\mu^\delta,\hat\Q}_T\subseteq\c F^{B,\hat\Q}_T\lor\c F'\otimes\c F''$, which together with the independence of $\c G,W$ from $B,\c F'\otimes\c F''$ under $\hat\Q$ shows that $\hat\Q^{\c F^{B,\mu^\delta,\hat\Q}_s}_{X^{t,\xi,\alpha_t,\delta}_s} = \hat\Q^{\c F^{B,\hat\Q}_T\lor\c F'\otimes\c F''}_{X^{t,\xi,\alpha_t,\delta}_s}$, $\hat\Q$-a.s., for all $s\in [t,T]$. 
Therefore we can estimate, for all $s\in [t,T]$, $\hat\Q$-a.s.,
\[
\c W_2(\hat\Q^{\c F^{B,\mu^\delta,\hat\Q}_s}_{X^{t,\xi,\alpha_t,\delta}_s}, \hat\Q^{\c F^{B,\hat\Q}_s}_{X^{t,\xi,\alpha}_s})^2 = \c W_2(\hat\Q^{\c F^{B,\hat\Q}_T\lor\c F'\otimes\c F''}_{X^{t,\xi,\alpha_t,\delta}_s},\hat\Q^{\c F^{B,\hat\Q}_T\lor\c F'\otimes\c F''}_{X^{t,\xi,\alpha}_s})^2 \leq \E^{\hat\Q}\big[|X^{t,\xi,\alpha_t,\delta}_s - X^{t,\xi,\alpha}_s|^2 \,\big|\,\c F^{B,\hat\Q}_T\lor\c F'\otimes\c F''\big],
\]
which implies that
\[
\sup_{s\in[t,T]} \c W_2(\hat\Q^{\c F^{B,\mu^\delta,\hat\Q}_s}_{X^{t,\xi,\alpha_t,\delta}_s},\hat\Q^{\c F^{B,\hat\Q}_s}_{X^{t,\xi,\alpha}_s})^2
\leq \E^{\hat\Q}\Big[\sup_{s\in[t,T]} |X^{t,\xi,\alpha_t,\delta}_s - X^{t,\xi,\alpha}_s|^2 \,\Big|\,\c F^{B,\hat\Q}_T\lor\c F'\otimes\c F'' \Big].
\]
Now, from \eqref{eq proof V leq V^R alpha I delta close}, and using \cref{assumptions sde coefficients}, standard arguments involving the BDG-inequality and Gronwall's Lemma lead to the following estimate
\[
\E^{\hat\Q}\Big[\sup_{s\in[t,T]} \Big(|X^{t,\xi,\alpha_t,\delta}_s - X^{t,\xi,\alpha}_s|^2 + \c W_2(\hat\Q^{\c F^{B,\mu^\delta,\hat\Q}_s}_{X^{t,\xi,\alpha_t,\delta}_s},\hat\Q^{\c F^{B,\hat\Q}_s}_{X^{t,\xi,\alpha}_s})^2 \Big) \Big] \to 0,
\]
and thus by \cref{assumptions reward functional} that
\begin{align}
&\E^{\hat\Q}\Big[g(\hat\Q^{\c F^{B,\mu^\delta,\hat\Q}_T}_{X^{t,\xi,\alpha_t,\delta}_T},X^{t,\xi,\alpha_t,\delta}_T) + \int_t^T f(s,\hat\Q^{\c F^{B,\mu^\delta,\hat\Q}_s}_{X^{t,\xi,\alpha_t,\delta}_s},X^{t,\xi,\alpha_t,\delta}_s,I^{t,\alpha_t,\delta}_s) ds \Big]\\
&\to \E^{\hat\Q}\Big[g(\hat\Q^{\c F^{B,\hat\Q}_T}_{X^{t,\xi,\alpha}_T},X^{t,\xi,\alpha}_T) + \int_t^T f(s,\hat\Q^{\c F^{B,\hat\Q}_s}_{X^{t,\xi,\alpha}_s},X^{t,\xi,\alpha}_s,\alpha_s) ds \Big] = J(t,\xi,\alpha).
\end{align}

Therefore, the claim follows if we prove that
\[
\E^{\hat\Q}\Big[g(\hat\Q^{\c F^{B,\mu^\delta,\hat\Q}_T}_{X^{t,\xi,\alpha_t,\delta}_T},X^{t,\xi,\alpha_t,\delta}_T) + \int_t^T f(s,\hat\Q^{\c F^{B,\mu^\delta,\hat\Q}_s}_{X^{t,\xi,\alpha_t,\delta}_s},X^{t,\xi,\alpha_t,\delta}_s,I^{t,\alpha_t,\delta}_s) ds \Big] \leq V^{\c R}(t,\xi,\alpha_t),\qquad\text{for all }\delta > 0.
\]
For this, we will express the left-hand side as the pay-off of a randomised control, using that by \cref{proposition randomised setting is independent of the extension} we can freely choose our randomised probability setting. We will define a new probability measure $\hat\P^{(\delta)}\sim\hat\Q$ on $(\hat\Omega,\hat{\c F})$ via
\[
\frac{d\hat\P^{(\delta)}}{d\hat\Q}\Big|_{\c F^{B,\mu^\delta,\hat\Q}_s} \coloneqq L^{\frac 1 {\nu^\delta}}_s = \exp\Big(\int_{(0,s]} \int_{\c A_r} \log \frac 1 {\nu^\delta_r(\alpha)} \ \mu^\delta(dr,d\alpha) - \int_0^s \int_{\c A_r} \Big(\frac 1 {\nu^\delta_r(\alpha)} - 1\Big) \lambda_r(d\alpha)dr \Big),\quad s\in [0,T].
\]
By \cref{lemma d P nu d P is square integrable} and Girsanov's Theorem, see \cite[Proposition 14.4.I]{daley_introduction_2008}, we note that $\mu^\delta$ has under $\hat\P^{(\delta)}$ the intensity $\lambda_s(d\alpha)ds$, since by property (iii) the process $\nu^\delta$ and thus also $\frac 1 {\nu^\delta}$ are strictly bounded away from $0$ and $\infty$. Further, following the same arguments as in \cite[page 2155]{fuhrman_randomized_2015} shows that $\mu^\delta$ is independent of $W,B$ and also $B$ resp. $W$ are still $\b F^{B,\mu^\delta,\hat\P^{(\delta)}}$- resp. $\b F^{W,\mu^\delta,\hat\P^{(\delta)}}$-Brownian motions under $\hat\P^{(\delta)}$. Finally, since $\nu^\delta$ is $Pred(\b F^{B,\mu^\delta})\otimes\c B(\c A_T)$-measurable and thus independent of $\c G$ under $\hat\Q$, this shows that $\mu^\delta$ stays independent of $\c G$ under $\hat\P^{(\delta)}$.

Next, we observe that the Radon-Nikodym derivative given by $\frac{d\hat\P^{(\delta)}}{d\hat\Q}|_{\c F^{B,\mu^\delta,\hat\Q}_\cdot} = L^{\frac 1 {\nu^\delta}}$ is by construction $\b F^{B,\mu^\delta}$-progressive and that $\E^{\P'\otimes\P''}[\frac{d\hat\P^{(\delta)}}{d\hat\Q}|_{\c F^{B,\mu^\delta,\hat\Q}_T}(\omega,\cdot)] = \E^{\P'\otimes\P''}[L^{\frac 1 {\nu^\delta}}_T(\omega,\cdot)] \equiv 1$, for $\P$-a.e.\@ $\omega\in\Omega$, which shows that $(\hat\Omega,\hat{\c F},\hat\P^{(\delta)})$ is an extension of $(\Omega,\c F,\P)$.
Finally, as outlined in \cref{remark conditional distribution unchanged under measure change}, we see that $X^{t,\xi,\alpha_t,\delta}$ satisfies the (uncontrolled) randomised state equation \eqref{eq sde randomised problem} under $\hat\P^{(\delta)}$, and thus $(\hat\Omega,\hat{\c F},\hat\P^{(\delta)})$ forms a randomised setting as defined in \cref{section randomised problem}. 

From the properties (iii) and (iv) of $\mu^\delta$, we note that $\nu^\delta\in\c V$ is an admissible randomised control to this randomised setting.
Further, by construction $\hat\Q = (\P^{(\delta)})^{\nu^\delta}$ is the probability measure belonging to the randomised control $\nu^\delta\in\c V$, and thus we conclude by \cref{proposition randomised setting is independent of the extension} that
\[
\E^{\hat\Q}\Big[g(\hat\Q^{\c F^{B,\mu^\delta,\hat\Q}_T}_{X^{t,\xi,\alpha_t,\delta}_T},X^{t,\xi,\alpha_t,\delta}_T) + \int_t^T f(s,\hat\Q^{\c F^{B,\mu^\delta,\hat\Q}_s}_{X^{t,\xi,\alpha_t,\delta}_s},X^{t,\xi,\alpha_t,\delta}_s,I^{t,\alpha_t,\delta}_s) ds \Big] = J^{\c R}_{(\hat\Omega,\hat{\c F},\hat\P^{(\delta)})}(t,\xi,\alpha_t,\nu^\delta) \leq V^{\c R}(t,\xi,\alpha_t).
\]
Finally, since $\alpha\in\c A$ was arbitrary, we thus conclude that $V(t,\xi) = \sup_{\alpha\in\c A} J(t,\xi,\alpha) \leq V^{\c R}(t,\xi,\alpha_t)$ for all $\alpha_t\in\c A_t$.

\subsection{Proof of $V^{\c R}\leq V$}

Based on the approach in \cite{bandini_backward_2018,bayraktar_randomized_2018}, we will consider a canonical extension (in some sense) $(\hat\Omega,\hat{\c F},\hat\P)$ of $(\Omega,\c F,\P)$, using that by \cref{proposition randomised setting is independent of the extension} we are free to choose the extension we want to work on. Together with $(\hat\Omega,\hat{\c F},\hat\P)$, we will in \cref{subsection extended auxiliary problem} define an auxiliary problem, which essentially extends the original control problem to this extended space. Denoting its value function by $V^{\c E}$, our goal for this section will be proving that $V^{\c R}\leq V^{\c E}\leq V$.

\subsubsection{The canonical extended space and an auxiliary problem}\label{subsection extended auxiliary problem}

We will start by construction the space $(\hat\Omega,\hat{\c F},\hat\P)$, which we will need for the later proofs in this section. We recall from \cref{section problem l2 formulation} that $(\lambda_s)_{s\in [0,T]}$ is a family of measures on $\c A_T$ such that each $\lambda_s$ is fully supported on $\c A_s$, and $\lambda_s\ll\lambda_r$ for $0\leq s\leq r\leq T$. Now following essentially the construction in \cite{bandini_backward_2018}, we can find a surjective, measurable map $\pi_{\c A_T}:\R\to\c A_T$ and a family of measures $(\lambda'_s)_{s\in [0,T]}$ on $\R$ such that $\lambda_s = \lambda'_s\circ\pi_{\c A_T}^{-1}$ and $\lambda'_s\ll\lambda'_r$ for $0\leq s\leq r\leq T$. 
This allows us to consider the canonical space $(\Omega',\c F',\P')$ of a nonexplosive Poisson point process on $(0,\infty)\times\R$ with intensity $\lambda'_s(dr)ds$, which is constructed as follows: Let $\Omega'$ be the set of all sequences $\omega' = (\tau_k,r_k)_{k\geq 1}\subseteq (0,\infty)\times\R$ with $\tau_k < \tau_{k+1} < \infty$ and $\tau_k\to\infty$. Next we denote the canonical process by $(T_k,R_k)_{k\geq 1}$ and let $\mu' \coloneqq \sum_{k\geq 1} \delta_{(T_k,R_k)}$ be the corresponding point measure. Letting $\P'$ be the (unique) probability measure, under which $\mu'$ is a Poisson random measure with intensity $\lambda'_s(dr)ds$, (recall that $\lambda_s$ and thus $\lambda'_s$ are bounded), we define $\c F'\coloneqq \sigma(T_k,R_k|k\geq 1)\lor\c N^{\P'}$, where $\c N^{\P'}$ denotes all $\P'$-null sets. We note that by construction $\c F' = \c F^{\mu,\P'}_\infty$.
Finally, we define $A_k\coloneqq \pi_{\c A_T}(R_k)\in\c A_T$ for each $k\geq 1$ and note that by construction $A_k\in\c A_{T_k}$, $\P'$-a.s. Therefore we define the random point measure $\mu \coloneqq \sum_{k\geq 1} \delta_{(T_k,A_k)}$, which by construction has intensity $\lambda_s(d\alpha)ds$ under $\P'$.
In the following we will denote by $(\hat\Omega,\hat{\c F},\hat\P) \coloneqq (\Omega\times\Omega',(\c F\otimes\c F')^{\P\otimes\P'},\P\otimes\P')$ the extension of $\Omega$ with $\Omega'$. As usual, we denote the canonical extension of $\c G,\xi,W,B,\c F'$ to $(\hat\Omega,\hat{\c F},\hat\P)$ with the same symbol.

Next we will define an auxiliary problem on this extended space. For this, we define $\c A^{\c E}$ as the set of all $\b F^{W,B}\lor\c G\lor\c F'$-predictable processes $\alpha$. Then for each $t\in [0,T]$, $\xi\in L^2(\Omega,\c G\lor\c F^W_t,\P;\R^n)$ and $\alpha\in\c A^{\c E}$, we denote by $X^{t,\xi,\alpha}$ the unique $\b F^{W,B,\hat\P}\lor\c G\lor\c F'$-progressive process solving the state dynamics
\begin{align}
    dX_s^{t,\xi,\alpha}
    &= b(s, \hat\P_{X^{t,\xi,\alpha}_s}^{\c F^{B,\hat\P}_s\lor\c F'},X^{t,\xi,\alpha}_s,\alpha_s)ds + \sigma(s,\hat\P_{X^{t,\xi,\alpha}_s}^{\c F^{B,\hat\P}_s\lor\c F'},X^{t,\xi,\alpha}_s,\alpha_s) dW_s\\
    &\quad + \sigma^0(s,\hat\P_{X^{t,\xi,\alpha}_s}^{\c F^{B,\hat\P}_s\lor\c F'},X^{t,\xi,\alpha}_s,\alpha_s) dB_s,\qquad X^{t,\xi,\alpha}_t = \xi.
    \label{eq extended problem state dynamics}
\end{align}
As before, here $(\hat\P^{\c F^{B,\hat\P}_s\lor\c F'}_{X^{t,\xi,\alpha}_s})_{s\in [t,T]}$ shall denote the $\b F^{B,\hat\P}\lor\c F'$-predictable and $\hat\P$-a.s.\@ continuous version of $(\c L(X^{t,\xi,\alpha}_s|\c F^{B,\hat\P}_s\lor\c F'))_{s\in [t,T]}$, that is $\hat\P^{\c F^{B,\hat\P}_s\lor\c F'}_{X^{t,\xi,\alpha}_s} = \c L(X^{t,\xi,\alpha}_s|\c F^{B,\hat\P}_s\lor\c F')$, $\hat\P$-a.s., for all $s\in [t,T]$. We recall that the existence of such a version is ensured by \cite[Lemma A.1]{djete_mckeanvlasov_2022-1}, since $(t,\omega,\omega')\mapsto (B_t(\omega),\omega')$ is a continuous Hunt process and thus by \cite[Chapter 3, Section 2.3, pp. 30–32]{chung_introduction_1990}, $O(\b F^{B,\hat\P}\lor\c F') = Pred(\b F^{B,\hat\P}\lor\c F')$.
Again we note that by standard results, the existence and uniqueness of $X^{t,\xi,\alpha}$ follows from \cref{assumptions sde coefficients}, as well as the basic estimate
\begin{align}
    \sup_{\alpha\in\c A^{\c E}} \E^{\hat\P}\Big[\sup_{s\in [t,T]} |X^{t,\xi,\alpha}_s|^2 \Big] < \infty.
\end{align}
Finally, by \cref{assumptions reward functional} the following reward functional
\[
J^{\c E}(t,\xi,\alpha) \coloneqq \E^{\hat\P}\Big[ g(\hat\P_{X^{t,\xi,\alpha}_T}^{{\c F}^{B,\hat\P}_T\lor\c F'},X^{t,\xi,\alpha}_T) + \int_t^T f(s,\hat\P_{X^{t,\xi,\alpha}_s}^{{\c F}^{B,\hat\P}_s\lor\c F'},X^{t,\xi,\alpha}_s,\alpha_s) ds \Big]
\]
is again well-defined, and we can introduce the corresponding value function as
\[
V^{\c E}(t,\xi) \coloneqq \sup_{\alpha\in\c A^{\c E}} J^{\c E}(t,\xi,\alpha) < \infty.
\]

\subsubsection{Proof of $V^{\c E}\leq V$}\label{subsection V^E leq V}

Let $(\hat\Omega,\hat{\c F},\hat\P)$ be the canonical extended space constructed in \cref{subsection extended auxiliary problem}, and 
let us fix $t\in [0,T]$ and $\xi\in L^2(\Omega,\c G\lor\c F^W_t,\P;\R^d)$. 
Further, let $\alpha\in\c A^{\c E}$ be an admissible control to the extended problem and $X^{t,\xi,\alpha}$ the corresponding state process satisfying \eqref{eq extended problem state dynamics}. Since by definition $\alpha:[0,T]\times\Omega\times\Omega'\to A$ is $(\b F^{W,B}\lor\c G)\otimes\c F'$-predictable, we see that for all $\omega'\in\Omega'$, the process $\alpha^{\omega'}\coloneqq \alpha(\cdot,\omega')$ is $\b F^{W,B}\lor\c G$-predictable and in particular $\alpha^{\omega'}\in\c A$. This motivates us to consider for each $\omega'\in\Omega'$ the corresponding state process $X^{t,\xi,\alpha^{\omega'}}$ to $\alpha^{\omega'}$ satisfying the state dynamics \eqref{eq mfc state dynamics}, and our next goal is showing that
\[
X^{t,\xi,\alpha}(\cdot,\omega') = X^{t,\xi,\alpha^{\omega'}},\ \P\text{-a.s.},\quad \text{for }\P'\text{-a.a. }\omega'\in\Omega'.
\]
To this end, we note that since $(\hat\P^{\c F^{B,\hat\P}_s\lor\c F'}_{X^{t,\xi,\alpha}_s})_{s\in [t,T]}$ is $\b F^{B,\hat\P}\lor\c F'$-predictable and for our canonical extended space $\b F^{B,\hat\P}\lor\c F' = (\b F^B\otimes\c F')^{\hat\P}$, it follows that for $\P'$-a.a.\@ $\omega'\in\Omega'$, the process $(\hat\P^{\c F^{B,\hat\P}_s\lor\c F'}_{X^{t,\xi,\alpha}_s}(\cdot,\omega'))_{s\in [t,T]}$ is $\b F^{B,\P}$-predictable. At the same time, we see that since $\hat\P = \P\otimes\P'$, for all $s\in [t,T]$ and for $\P'$-a.a.\@ $\omega'\in\Omega'$,
\[
\hat\P^{\c F^{B,\hat\P}_s\lor\c F'}_{X^{t,\xi,\alpha}_s}(\cdot,\omega') = \c L^{\hat\P}(X^{t,\xi,\alpha}_s | (\c F^B_s\otimes\c F')^{\hat\P})(\cdot,\omega') = \c L^{\P}(X^{t,\xi,\alpha}_s(\cdot,\omega') | \c F^{B,\P}_s),\quad\P\text{-a.s},
\]
and thus $(\hat\P^{\c F^{B,\hat\P}_s\lor\c F'}_{X^{t,\xi,\alpha}_s}(\cdot,\omega'))_{s\in [t,T]}$ is a suitable $\P$-a.s.\@ continuous and $\b F^{B,\P}$-predictable version of the conditional law process $(\c L^{\P}(X^{t,\xi,\alpha}_s(\cdot,\omega') | \c F^{B,\P}_s))_{s\in [t,T]}$, for $\P'$-a.a.\@ $\omega'\in\Omega'$. Finally, this shows that for $\P'$-a.a.\@ $\omega'\in\Omega'$, the process $X^{t,\xi,\alpha}(\cdot,\omega')$ also satisfies the SDE \eqref{eq mfc state dynamics} with the control $\alpha^{\omega'} = \alpha(\cdot,\omega')$, and thus by the uniqueness of strong solutions to \eqref{eq mfc state dynamics}, we deduce that $X^{t,\xi,\alpha}(\cdot,\omega') = X^{t,\xi,\alpha^{\omega'}}$, $\P$-a.s.

This now allows us to rewrite the extended reward functional $J^{\c E}$ in terms of $J$ as follows,
\begin{align}
J^{\c E}(t,\xi,\alpha)
&= \int_{\Omega'} \E\Big[g(\hat\P^{\c F^{B,\hat\P}_T\lor\c F'}_{X^{t,\xi,\alpha}_T}(\cdot,\omega'),X^{t,\xi,\alpha}_T(\cdot,\omega'))\\
&\qquad + \int_t^T f(s, \hat\P^{\c F^{B,\hat\P}_s\lor\c F'}_{X^{t,\xi,\alpha}_s}(\cdot,\omega'),X^{t,\xi,\alpha}_s(\cdot,\omega'),\alpha_s(\cdot,\omega')) ds\Big] \P'(d\omega')\\
&= \int_{\Omega'} J(t,\xi,\alpha^{\omega'}) \P'(d\omega')
\leq V(t,\xi),
\end{align}
and since $\alpha\in\c A^{\c E}$ was chosen arbitrary, this shows that $V^{\c E}(t,\xi) = \sup_{\alpha\in\c A^{\c E}} J^{\c E}(t,\xi,\alpha)\leq V(t,\xi)$.

\subsubsection{Proof of $V^{\c R}\leq V^{\c E}$}\label{subsection V^R leq V^E}

Recalling that by \cref{proposition randomised setting is independent of the extension} we are free to choose our randomised probability space, we will consider the canonical extended space $(\hat\Omega,\hat{\c F},\hat\P) = (\Omega \times \Omega',(\c F\otimes\c F')^{\P\otimes\P'},\P\otimes\P')$ from \cref{subsection extended auxiliary problem}. Let us further fix $t\in [0,T]$, $\xi\in L^2(\Omega,\c G\lor\c F^W_t,\P;\R^d)$ and $\alpha_t\in\c A_t$. We now need to show that for any $\nu\in\c V$, we have $J^{\c R}(t,\xi,\alpha_t,\nu) \leq V^{\c E}(t,\xi)$.

To this end, let us fix $\nu\in\c V$ and let $\hat\P^\nu\sim\hat\P$ be the corresponding tilted probability measure under which $\mu$ has the intensity $\nu_s(\alpha)\lambda_s(d\alpha)ds$, defined in \eqref{eq randomised control girsanov formula}.
Following the approach by \cite{bandini_backward_2018}, we will construct a piece-wise constant process $\hat\alpha^\nu\in\c A^{\c E}$ that has under $\hat\P$ the same distribution as the Poisson point process $\hat I^{t,\alpha_t}$ under $\hat\P^\nu$ using \cite[Lemma 4.3]{bandini_backward_2018}.
To set this up, we define for each $\omega\in\Omega$ the process $\nu^\omega:[0,T]\times\Omega'\times\c A_T\to (0,\infty)$ by $\nu^\omega_s(\omega',\alpha)\coloneqq \nu_s(\omega,\omega',\alpha)$. This $\nu^\omega$ is then $Pred(\b F^\mu)\otimes\c B(\c A_T)$-measurable and satisfies $0\leq \inf_{[0,T]\times\Omega'\times\c A_T} \nu^\omega\leq \sup_{[0,T]\times\Omega'\times\c A_T} \nu^\omega < \infty$, see also \cite[Lemma 4.2]{bandini_backward_2018}. Correspondingly we define probability measures ${\P'}^{\nu^\omega}\sim\P'$ on $(\Omega',{\c F}')$ via $\frac{d{\P'}^{\nu^\omega}}{d\P'}|_{\c F^\mu_T} = L^{\nu^\omega}_T$, where $L^{\nu^\omega}$ is given by \eqref{eq randomised control girsanov formula}. Using an analogous result to \cref{lemma d P nu d P is square integrable}, we note that ${\P'}^{\nu^\omega}$ is indeed well-defined, and moreover we observe that by construction $\hat\P^\nu(d\omega,d\omega') = {\P'}^{\nu^\omega}(d\omega')\P(d\omega)$.

Next we apply \cite[Lemma 4.3]{bandini_backward_2018} and obtain a process
\[
\hat\alpha'^\nu =  \alpha_t \1_{[t,\tau'^\nu_1)} + \sum_{k\geq 1} \alpha'^\nu_k \1_{[\tau'^\nu_k,\tau'^\nu_{k+1})},
\]
where $\tau'^\nu_k\geq t$ are $\b F^{B,\hat\P}\lor\c F'$-stopping times and $\alpha'^\nu_k$ are $\c F^{B,\hat\P}_{\tau^\nu_k}\lor\c F'$-measurable, $k\geq 1$, such that
\[
\c L^{\P'}(\hat\alpha'^\nu(\omega,\cdot)) = \c L^{{\P'}^{\nu^\omega}}(\hat I^{t,\alpha_t}),\qquad\text{for }\P\text{-a.a. }\omega\in\Omega.
\]
Now we note that $(t,\omega,\omega')\mapsto (B_t(\omega),\omega')$ is a continuous Hunt process and thus by \cite[Chapter 3, Section 2.3, pp. 30–32]{chung_introduction_1990} every $\b F^{B,\hat\P}\lor\c F'$-optional time is also $\b F^{B,\hat\P}\lor\c F'$-predictable. 
This allows us to use the arguments in \cite[Lemma 6]{claisse_pseudo-markov_2016} to conclude that there exist $\b F^B\lor\c F'$-predictable stopping times $\tau^\nu_k$ and $\c F^B_{\tau^\nu_k}\lor\c F'$-measurable $\alpha^\nu_k$ such that $\tau^\nu_k = \tau'^\nu_k$ and $\alpha^\nu_k = \alpha'^\nu_k$, $\hat\P$-a.s. Therefore, the process
\[
\hat\alpha^\nu \coloneqq \alpha_t \1_{[t,\tau^\nu_1)} + \sum_{k\geq 1} \alpha^\nu_k \1_{[\tau^\nu_k,\tau^\nu_{k+1})}
\]
is $\b F^B\lor\c F'$-predictable and indistinguishable from $\hat\alpha'^\nu$, and thus we have
\[
\c L^{\P'}(\hat\alpha^\nu(\omega,\cdot)) = \c L^{\P'}(\hat\alpha'^\nu(\omega,\cdot)) = \c L^{{\P'}^{\nu^\omega}}(\hat I^{t,\alpha_t}),\qquad\text{for }\P\text{-a.a. }\omega\in\Omega.
\]

We note that this implies that
\begin{align}
\c L^{\hat\P}(\pi|_{\c G},\xi,W,B,\hat\alpha^\nu) = \c L^{\hat\P^\nu}(\pi|_{\c G},\xi,W,B,\hat I^{t,\alpha_t}),
\label{eq proof V R <= V E law equality}
\end{align}
where $\pi|_{\c G}:(\hat\Omega,\c G)\to (\Omega,\c G)$ is the canonical $\c G$-measurable projection, since for all bounded, measurable functions $\phi:(\Omega\times\R^d\times\R^m\times\R^n,\c G\times\c B(\R^d\times\R^m\times\R^n))\to\R$ and $\psi:(\c A_T,\c B(\c A_T))\to\R$ we have
\begin{align}
    &\E^{\hat\P^\nu}[\phi(\pi|_{\c G},\xi,W,B) \psi(\hat I^{t,\alpha_t})] \\
    &= \int_{\Omega\times\Omega'} \phi(\omega,\xi(\omega),W(\omega),B(\omega))  \psi(\hat I^{t,\alpha_t}(\omega,\omega')) \hat\P^\nu(d\omega,d\omega')\\
    &= \int_\Omega \phi(\omega,\xi(\omega),W(\omega),B(\omega)) \int_{\Omega'} \psi(\hat I^{t,\alpha_t}(\omega,\omega')) {\P'}^{\nu^\omega}(d\omega') \P(d\omega)\\
    &= \int_\Omega \phi(\omega,\xi(\omega),W(\omega),B(\omega)) \int_{\Omega'} \psi(\hat\alpha^\nu(\omega,\omega')) \P'(d\omega') \P(d\omega)\\
    &= \E^{\hat\P}[\phi(\pi|_{\c G},\xi,W,B) \psi(\hat\alpha^\nu)].
\end{align}

As a next step, we recall that by \cref{corollary I given bar I is well defined} there is an unique (up to indistinguishability) càdlàg, $\b F^{W}\lor\c G\lor\c F^{\hat I^{t,\alpha_t}}_T$-predictable and $\b F^{W,\hat I^{t,\alpha_t}}\lor\c G$-progressive process $I^{t,\alpha_t}$ $\b F^{\hat I^{t,\alpha_t}}$-identifying to $\hat I^{t,\alpha_t}$. Similarly, \cref{lemma existence N bar N,lemma alpha bar alpha uniqueness} shows that there is an unique (up to modification) $\b F^{W}\lor\c G\lor\c F^{\hat\alpha^\nu}_T$-predictable and $\b F^{W,\hat\alpha^\nu}\lor\c G$-progressive process $\alpha^\nu$ which $\b F^{\hat\alpha^\nu}$-identifies to $\hat\alpha^\nu$, and moreover together with \eqref{eq proof V R <= V E law equality}, we obtain by \cref{lemma joint law N bar N map} that
\begin{align}
\c L^{\hat\P}(\pi|_{\c G},\xi,W,B,\hat\alpha^\nu,\alpha^\nu) = \c L^{\hat\P^\nu}(\pi|_{\c G},\xi,W,B,\hat I^{t,\alpha_t},I^{t,\alpha_t}).
\label{eq proof V R <= V E second law equality}
\end{align}
We note that $\alpha^\nu$ is $\b F^{W,B}\lor\c G\lor\c F'$-predictable and hence $\alpha^\nu\in\c A^{\c E}$.

Since $I^{t,\alpha_t}$ is $\b F^{W,\hat I^{t,\alpha_t}}\lor\c G$-progressive, there exists a unique $\b F^{W,B^t,\hat I^{t,\alpha_t},\hat\P}\lor\c G$-progressive process $X^{t,\xi,\alpha_t}$ solving the randomised state dynamics \eqref{eq sde randomised problem}, and similar since $\alpha^\nu$ is $\b F^{W,\hat\alpha^\nu}\lor\c G$-progressive, there exists a unique $\b F^{W,B,\hat\P}\lor\c G\lor\c F'$-progressive $X^{t,\xi,\alpha^\nu}$ process solving \eqref{eq extended problem state dynamics}. 
This implies that the corresponding conditional law process $(\P^{\nu,\c F^{B,\mu,\hat\P^\nu}_s}_{X^{t,\xi,\alpha_t}_s})_{s\in [t,T]}$ is $\b F^{B,\hat I^{t,\alpha_t},\hat\P^\nu}$-progressive, since $\c G,W$ are $\hat\P^\nu$-independent of $B,\mu$ and similar that $(\P^{\c F^{B,\hat\P}_s\lor\c F'}_{X^{t,\xi,\alpha^\nu}_s})_{s\in [t,T]}$ is $\b F^{B,\hat\alpha^\nu,\hat\P}$-progressive, since $\c G,W$ are $\hat\P$-independent of $B,\c F'$. This implies that
\[
\hat\P^{\nu,\c F^{B,\mu,\hat\P^\nu}_s}_{X^{t,\xi,\alpha_t}_s} = \hat\P^{\nu,\c F^{B,\hat I^{t,\alpha_t},\hat\P^\nu}_s}_{X^{t,\xi,\alpha_t}_s},
\qquad \hat\P^{\c F^{B,\hat\P}_s\lor\c F'}_{X^{t,\xi,\alpha^\nu}_s} = \hat\P^{\c F^{B,\hat\alpha^\nu,\hat\P}_s}_{X^{t,\xi,\alpha^\nu}_s},\qquad\hat\P\text{-a.s.},
\]
and therefore both $X^{t,\xi,\alpha_t}$ and $X^{t,\xi,\alpha^\nu}$ satisfy the following equation
\begin{align}
    dX_s &= b(s,\hat\P^{\c F^{B,\hat\alpha,\hat\P}_s}_{X_s},X_s,\alpha_s) ds + \sigma(s,\hat\P^{\c F^{B,\hat\alpha,\hat\P}_s}_{X_s},X_s,\alpha_s) dW_s + \sigma^0(s,\hat\P^{\c F^{B,\hat\alpha,\hat\P}_s}_{X_s},X_s,\alpha_s) dB_s,\  X_t = \xi,
    \label{eq proof V R <= V E auxiliary SDE}
\end{align}
with the corresponding control $\hat\alpha = \hat I^{t,\alpha_t}$ respectively $\hat\alpha = \hat\alpha^\nu$. Consequently, together with \eqref{eq proof V R <= V E second law equality} the weak uniqueness of solutions to \eqref{eq proof V R <= V E auxiliary SDE} now shows that
\[
\c L^{\hat\P}((\hat\P^{\c F^{B,\hat\P}_s\lor\c F'}_{X^{t,\xi,\alpha^\nu}_s})_{s\in [t,T]}, X^{t,\xi,\alpha^\nu},\alpha^\nu)
= \c L^{\hat\P^\nu}((\hat\P^{\nu,\c F^{B,\mu,\hat\P^\nu}_s}_{X^{t,\xi,\alpha_t}_s})_{s\in [t,T]}, X^{t,\xi,\alpha_t},I^{t,\alpha_t}).
\]

Finally, this implies that
\begin{align}
    J^{\c R}(t,\xi,\alpha_t,\nu)
    &= \E^{\hat\P^\nu}\Big[ g(\hat\P_{X^{t,\xi,\alpha_t}_T}^{\c F^{B,\mu,\hat\P}_T},X^{t,\xi,\alpha_t}_T) + \int_t^T f(s,\hat\P_{X^{t,\xi,\alpha_t}_s}^{\c F^{B,\mu,\hat\P}_s},X^{t,\xi,\alpha_t}_s,I^{t,\alpha_t}_s) ds \Big]\\
    &= \E^{\hat\P}\Big[g(X^{t,\xi,\alpha^\nu}_T, \hat\P_{X^{t,\xi,\alpha^\nu}_T}^{{\c F}^{B,\hat\P}_T\lor\c F'}) + \int_t^T f(s,X^{t,\xi,\alpha^\nu}_s,\hat\P_{X^{t,\xi,\alpha^\nu}_s}^{{\c F}^{B,\hat\P}_s\lor\c F'},\alpha^\nu_s) ds \Big]
    = J^{\c E}(t,\xi,\alpha),
\end{align}
and since $\nu\in\c V$ was arbitrary, this shows that $V^{\c R}(t,\xi,\alpha_t) = \sup_{\nu\in\c V} J^{\c R}(t,\xi,\alpha_t,\nu) \leq V^{\c E}(t,\xi)$.


\section{Proof of \cref{lemma penalised bsde minimal solution existence}}
\label{appendix proof lemma penalised bsde minimal solution existence}

\begin{proof}
\begin{enumerate}[wide,label=(\roman*)]
\item 
We first note that we can view the penalised BSDE \eqref{eq penalised bsde value function} as BSDE on $[0,T]$ with the driver
\[
F(s,\hat\omega,u) \coloneqq \1_{[t,T]}(s) \Big(H_{s-}(\hat\omega) + n \int_{\c A_s} (u(\alpha))_+ \lambda_s(d\alpha)\Big),
\]
where $H_r \coloneqq \E^{\hat\P}[f(r,\hat\P^{\c F^{B,\mu,\hat\P}_r}_{X^{t,\xi,\alpha_t}_r},X^{t,\xi,\alpha_t}_r,I^{t,\alpha_t}_r) | \c F^{B,\mu,\hat\P}_r]$, and terminal value
\[
G \coloneqq \E^{\hat\P}[g(X^{t,\xi,\alpha_t}_T,\hat\P^{\c F^{B,\mu,\hat\P}_T}_{X^{t,\xi,\alpha_t}_T}) | \c F^{B,\mu,\hat\P}_T].
\]
Then by \cite[Lemma 2.4]{tang_necessary_1994}, there exists a unique solution $(Y^{n,t,\xi,\alpha_t},Z^{n,t,\xi,\alpha_t},U^{n,t,\xi,\alpha_t}) \in \c S^2_{[t,T]}(\b F^{B,\mu,\hat\P})\times L^2_{W,[t,T]}(\b F^{B,\mu,\hat\P})\times L^2_{\lambda,[t,T]}(\b F^{B,\mu,\hat\P})$ to \eqref{eq penalised bsde value function}, noting that while boundedness of the driver $F$ in \cite[(2.15)]{tang_necessary_1994} is not fulfilled by our driver $F$, this assumption is only needed to show that for each $\bar U\in L^2_{\lambda,[t,T]}(\b F^{B,\mu,\hat\P})$,
\begin{align}
\E^{\hat\P}\Big[\Big|H + \int_0^T F(s,\bar U_s) ds \Big|^2\Big] < \infty,
\end{align}
which in our setting holds since by \cref{assumptions reward functional} together with \eqref{eq randomised state dynamics basic estimate}, we have
\begin{align}
&\E^{\hat\P}\Big[\Big|H + \int_0^T F(s,\bar U_s) ds \Big|^2\Big]\\
&\leq 
2\E^{\hat\P}\Big[|\E^{\hat\P}[g(X^{t,\xi,\alpha_t}_T,\hat\P^{\c F^{B,\mu,\hat\P}_T}_{X^{t,\xi,\alpha_t}_T}) | \c F^{B,\mu,\hat\P}_T]|^2\\
&\qquad + T \int_t^T \Big|\E^{\hat\P}[f(s,\hat\P^{\c F^{B,\mu,\hat\P}_s}_{X^{t,\xi,\alpha_t}_s},X^{t,\xi,\alpha_t}_s,I^{t,\alpha_t}_s) | \c F^{B,\mu,\hat\P}_s] + n \int_{\c A_s} (\bar U_s(\alpha))_+ \lambda_s(d\alpha) \Big|^2 ds\Big]\\
&\leq 2\E^{\hat\P}\Big[|g(X^{t,\xi,\alpha_t}_T,\hat\P^{\c F^{B,\mu,\hat\P}_T}_{X^{t,\xi,\alpha_t}_T})|^2 + T \int_t^T |f(s,\hat\P^{\c F^{B,\mu,\hat\P}_s}_{X^{t,\xi,\alpha_t}_s},X^{t,\xi,\alpha_t}_s,I^{t,\alpha_t}_s)|^2 ds\\
&\qquad + n T \sup_{s\in [0,T]} \lambda_s(\c A_s) \int_t^T \int_{\c A_s} |\bar U_s(\alpha)|^2 \lambda_s(d\alpha) ds \Big] < \infty.
\end{align}

\item\label{proof lemma 6.1 step 2}
Next we will prove the representation \eqref{eq Y n t xi snell envelope formula}. From the BSDE \eqref{eq penalised bsde value function}, for any $\nu\in\c V$, by taking the conditional expectation under $\hat\P^\nu$ given $\c F^{B,\mu,\hat\P^\nu}_s$, we obtain for all $s\in [t,T]$ and $r\in [s,T]$,
\begin{align}
    Y^{n,t,\xi,\alpha_t}_s
    &= \E^{\hat\P^\nu}\Big[Y^{n,t,\xi,\alpha_t}_r + \int_s^r f(u,\hat\P^{\nu,\c F^{B,\mu,\hat\P^\nu}_u}_{X^{t,\xi,\alpha_t}_u},X^{t,\xi,\alpha_t}_u,I^{t,\alpha_t}_u) du\\
    &\qquad + \int_s^r\int_{\c A_u} \big(n (U^{n,t,\xi,\alpha_t}_u(\alpha))_+ - \nu_u(\alpha) U^{n,t,\xi,\alpha_t}_u(\alpha)\big) \lambda_u(d\alpha) du\Big| \c F^{B,\mu,\hat\P^\nu}_s\Big].
    \label{eq Y n t xi expectation}
\end{align}
Noting that for $\nu\in\c V^n$ we have $0 < \nu\leq n$ and thus $n(U^{n,t,\xi,\alpha_t})_+ - \nu U^{n,t,\xi,\alpha_t} \geq 0$, we see that \eqref{eq Y n t xi expectation} shows that for all $\nu\in\c V^n$, $s\in [t,T]$ and $r\in [s,T]$,
\[
Y^{n,t,\xi,\alpha_t}_s
\geq \E^{\hat\P^\nu}\Big[Y^{n,t,\xi,\alpha_t}_r + \int_s^r f(u,\hat\P^{\nu,\c F^{B,\mu,\hat\P^\nu}_u}_{X^{t,\xi,\alpha_t}_u},X^{t,\xi,\alpha_t}_u,I^{t,\alpha_t}_u) du \Big| \c F^{B,\mu,\hat\P^\nu}_s\Big],
\]
and hence
\[
Y^{n,t,\xi,\alpha_t}_s
\geq \esssup_{\nu\in\c V^n} \E^{\hat\P^\nu}\Big[Y^{n,t,\xi,\alpha_t}_r + \int_s^r f(u,\hat\P^{\nu,\c F^{B,\mu,\hat\P^\nu}_u}_{X^{t,\xi,\alpha_t}_u},X^{t,\xi,\alpha_t}_u,I^{t,\alpha_t}_u) du \Big| \c F^{B,\mu,\hat\P^\nu}_s\Big].
\]

For the other direction, we define for each $\varepsilon\in (0,1]$ the randomised control $\nu^{n,\varepsilon}\in\c V^n$ via
\begin{align}
\nu^{n,\varepsilon}_s(\alpha) \coloneqq \begin{cases*}
n\1_{\{U^{n,t,\xi,\alpha_t}_s(\alpha) \geq 0\}} + \varepsilon\1_{\{U^{n,t,\xi,\alpha_t}_s(\alpha) < 0\}}&, on $[t,T]\times\c A_T$\\
1&, on $[0,t)\times\c A_T$.
\end{cases*}
\label{eq Y n t xi definition of nu n varepsilon}
\end{align}
Correspondingly, we define the probability measures $\hat\P^{\nu^{n,\varepsilon}}\sim\hat\P$ via $\frac{d\hat\P^{\nu^{n,\varepsilon}}}{d\hat\P}\big|_{\c F^{B,\mu,\hat\P}_u} \coloneqq L^{n,\varepsilon}_u$ for all $u\in [0,T]$, see \eqref{eq randomised control girsanov formula}. Thus from \eqref{eq Y n t xi expectation} we obtain for all $s\in [t,T]$ and $r\in [s,T]$,
\begin{align}
Y^{n,t,\xi,\alpha_t}_s
&\leq \E^{\hat\P^{\nu^{n,\varepsilon}}}\Big[Y^{n,t,\xi,\alpha_t}_r + \int_s^r f(u,\hat\P^{\nu^{n,\varepsilon},\c F^{B,\mu,\hat\P^{\nu^{n,\varepsilon}}}_u}_{X^{t,\xi,\alpha_t}_u},X^{t,\xi,\alpha_t}_u,I^{t,\alpha_t}_u) du \Big| \c F^{B,\mu,\hat\P^{\nu^{n,\varepsilon}}}_s\Big]\\
&\qquad + \E^{\hat\P^{\nu^{n,\varepsilon}}}\Big[\varepsilon\int_s^r \int_{\c A_u} (U^{n,t,\xi,\alpha_t}_u(\alpha))_- \lambda_u(d\alpha) du \Big|\c F^{B,\mu,\hat\P^{\nu^{n,\varepsilon}}}_s \Big]\\
&\leq \esssup_{\nu\in\c V^n} \E^{\hat\P^\nu}\Big[Y^{n,t,\xi,\alpha_t}_r + \int_s^r f(u,\hat\P^{\nu,\c F^{B,\mu,\hat\P^\nu}_u}_{X^{t,\xi,\alpha_t}_u},X^{t,\xi,\alpha_t}_u,I^{t,\alpha_t}_u) du \Big| \c F^{B,\mu,\hat\P^\nu}_s\Big]\\
&\qquad + \varepsilon \E^{\hat\P^{\nu^{n,\varepsilon}}}\Big[\int_s^r \int_{\c A_u} (U^{n,t,\xi,\alpha_t}_u(\alpha))_- \lambda_u(d\alpha) du \Big|\c F^{B,\mu,\hat\P^{\nu^{n,\varepsilon}}}_s \Big].
\label{eq Y n t xi upper bound}
\end{align}
Therefore the desired result follows by letting $\varepsilon\downarrow 0$, if we can show that the last term vanishes in the limit. 
To estimate the last term, we observe that since by construction $0 < \nu^{n,\varepsilon} \leq \nu^{n,1}$ for all $\varepsilon\in (0,1]$, we have that for all $s\in [t,T]$,
\begin{align}
0\leq \frac{L^{\nu^{n,\varepsilon}}_T}{L^{\nu^{n,\varepsilon}}_s} \frac{L^{\nu^{n,1}}_s}{L^{\nu^{n,1}}_T}
&= \exp\Big(\int_{[s,T]}\int_{\c A_r} \log \frac{\nu^{n,\varepsilon}_r(\alpha)}{\nu^{n,1}_r(\alpha)} \mu(dr,d\alpha) - \int_0^T \int_{\c A_r} (\nu^{n,\varepsilon}_r(\alpha) - \nu^{n,1}_r(\alpha)) \lambda_r(d\alpha)dr \Big)\\
&\leq \exp\Big(\int_0^T \int_{\c A_r} \lambda_r(d\alpha)dr\Big)^{\norm{\nu^{n,1}}_\infty} \eqqcolon C < \infty.
\end{align}
This allows us to bound the last term in \eqref{eq Y n t xi upper bound} as follows
\begin{align}
&\E^{\hat\P^{\nu^{n,\varepsilon}}}\Big[\int_s^r \int_{\c A_u} (U^{n,t,\xi,\alpha_t}_u(\alpha))_- \lambda_u(d\alpha) du \Big|\c F^{B,\mu,\hat\P^{\nu^{n,\varepsilon}}}_s \Big]\\
&= \E^{\hat\P}\Big[\frac{L^{\nu^{n,\varepsilon}}_T}{L^{\nu^{n,\varepsilon}}_s} \int_s^r \int_{\c A_u} (U^{n,t,\xi,\alpha_t}_u(\alpha))_- \lambda_u(d\alpha) du \Big|\c F^{B,\mu,\hat\P}_s \Big]\\
&\leq C \E^{\hat\P}\Big[\frac{L^{\nu^{n,1}}_T}{L^{\nu^{n,1}}_s} \int_s^r \int_{\c A_u} (U^{n,t,\xi,\alpha_t}_u(\alpha))_- \lambda_u(d\alpha) du \Big|\c F^{B,\mu,\hat\P}_s \Big]\\
&= C \E^{\hat\P^{\nu^{n,1}}}\Big[ \int_s^r \int_{\c A_u} (U^{n,t,\xi,\alpha_t}_u(\alpha))_- \lambda_u(d\alpha) du \Big|\c F^{B,\mu,\hat\P^{\nu^{n,1}}}_s \Big].
\end{align}
This term is now $L^1(\hat\P^{\nu^{n,1}})$-integrable, and due to $\hat\P^{\nu^{n,1}}\sim\hat\P$ also $\hat\P$-a.s.\@ finite, since
\begin{align}
&\E^{\hat\P^{\nu^{n,1}}}\Big[ \E^{\hat\P^{\nu^{n,1}}}\Big[ \int_s^r \int_{\c A_u} (U^{n,t,\xi,\alpha_t}_u(\alpha))_- \lambda_u(d\alpha) du \Big|\c F^{B,\mu,\hat\P^{\nu^{n,1}}}_s \Big] \Big]\\
&\leq \bigg(\E^{\hat\P}[ |L^{\nu^{n,1}}_T|^2 ] \E^{\hat\P}\Big[\Big(\int_s^r \int_{\c A_u} (U^{n,t,\xi,\alpha_t}_u(\alpha))_- \lambda_u(d\alpha) du\Big)^2\Big] \bigg)^{\frac 1 2}\\
&\leq \bigg(\E^{\hat\P}[|L^{\nu^{n,1}}_T|^2 ] \E^{\hat\P}\Big[\int_s^r \int_{\c A_u} |U^{n,t,\xi,\alpha_t}_u(\alpha)|^2 \lambda_u(d\alpha) du \Big] \int_s^r \lambda_u(\c A_u) du \bigg)^{\frac 1 2} < \infty,
\end{align}
since $U^{n,t,\xi,\alpha_t}\in L^2_{\lambda,[t,T]}(\b F^{B,\mu,\hat\P})$ and $L^{\nu^{n,1}}_T \in L^2(\hat\Omega,\c F^{B,\mu}_T,\hat\P;\R)$ by \cref{lemma d P nu d P is square integrable}.

\item Finally, let us turn to \eqref{eq Y n t xi snell envelope upper bound at s=t}. By taking the expectation with respect to $\E^{\hat\P^{\nu^{n,\varepsilon}}}$ in \eqref{eq Y n t xi upper bound} for $s=t$, we obtain hat for all $r\in [t,T]$ and $\varepsilon\in (0,1]$,
\begin{align}
\E^{\hat\P^{\nu^{n,\varepsilon}}}[Y^{n,t,\xi,\alpha_t}_t]
&\leq \E^{\hat\P^{\nu^{n,\varepsilon}}}\Big[Y^{n,t,\xi,\alpha_t}_r + \int_t^r f(u,\hat\P^{\nu^{n,\varepsilon},\c F^{B,\mu,\hat\P^{\nu^{n,\varepsilon}}}_u}_{X^{t,\xi,\alpha_t}_u},X^{t,\xi,\alpha_t}_u,I^{t,\alpha_t}_u) du \Big]\\
&\qquad + \varepsilon \E^{\hat\P^{\nu^{n,\varepsilon}}}\Big[\int_t^r \int_{\c A_u} (U^{n,t,\xi,\alpha_t}_u(\alpha))_- \lambda_u(d\alpha) du \Big].
\end{align}
Since by construction $\frac{d\hat\P^{\nu^{n,\varepsilon}}}{d\hat\P}|_{\c F^{B,\mu,\hat\P}_t} = L^{\nu^{n,\varepsilon}}_t \equiv 1$, we see that $\E^{\hat\P^{\nu^{n,\varepsilon}}}[Y^{n,t,\xi,\alpha_t}_t] = \E^{\hat\P}[Y^{n,t,\xi,\alpha_t}_t]$, and thus
\begin{align}
\E^{\hat\P}[Y^{n,t,\xi,\alpha_t}_t] = \E^{\hat\P^{\nu^{n,\varepsilon}}}[Y^{n,t,\xi,\alpha_t}_t]
&\leq \sup_{\nu\in\c V^n} \E^{\hat\P^{\nu}}\Big[Y^{n,t,\xi,\alpha_t}_r + \int_t^r f(u,\hat\P^{\nu,\c F^{B,\mu,\hat\P^{\nu}}_u}_{X^{t,\xi,\alpha_t}_u},X^{t,\xi,\alpha_t}_u,I^{t,\alpha_t}_u) du \Big]\\
&\qquad + \varepsilon \E^{\hat\P^{\nu^{n,\varepsilon}}}\Big[\int_t^r \int_{\c A_u} (U^{n,t,\xi,\alpha_t}_u(\alpha))_- \lambda_u(d\alpha) du \Big].
\end{align}
Finally, by letting $\varepsilon\downarrow 0$, we obtain \eqref{eq Y n t xi snell envelope upper bound at s=t}, recalling that we have already shown in \cref{proof lemma 6.1 step 2} that the last term will vanish.
\end{enumerate}
\end{proof}


\section{Proof of \cref{lemma bsde minimal solution existence}}
\label{appendix proof lemma bsde minimal solution existence}

\begin{proof}
\begin{enumerate}[wide,label=(\roman*)]
\item\label{proof theorem 6.2 part i} We will start by proving the existence and uniqueness of such a minimal solution 
to \eqref{eq constrained bsde value function} using \cite[Theorem 2.1]{kharroubi_feynmankac_2015}, which tells us that this unique minimal solution to \eqref{eq constrained bsde value function} should be the increasing limit of the solutions to the penalised BSDEs \eqref{eq penalised bsde value function} obtained from \cref{lemma penalised bsde minimal solution existence}. For this we view \eqref{eq constrained bsde value function} and the penalised versions \eqref{eq penalised bsde value function} as a BSDEs on $[0,T]$ with the driver $F(s,\hat\omega) \coloneqq \1_{[t,T]}(s) H_{s-}(\hat\omega)$, where $H_r \coloneqq \E^{\hat\P}[f(r,\hat\P^{\c F^{B,\mu,\hat\P}_r}_{X^{t,\xi,\alpha_t}_r},X^{t,\xi,\alpha_t}_r,I^{t,\alpha_t}_r) | \c F^{B,\mu,\hat\P}_r]$, and the terminal value $G \coloneqq \E^{\hat\P}[g(X^{t,\xi,\alpha_t}_T,\hat\P^{\c F^{B,\mu,\hat\P}_T}_{X^{t,\xi,\alpha_t}_T}) | \c F^{B,\mu,\hat\P}_T]$. Then we note that \cite[Assumption (H0)]{kharroubi_feynmankac_2015} holds: (ii) and (iii) are naturally satisfied as the generator $F$ does not depend on $Y,Z,U$, and (i) holds since $\xi\in L^2(\hat\Omega,\c G\lor\c F^W_t,\hat\P;\R^d)$ together with\cref{assumptions reward functional} and \eqref{eq randomised state dynamics basic estimate} show that
\begin{align}
\E^{\hat\P}\Big[|G|^2 + \int_0^T |F(s)|^2 ds\Big] 
&\leq \E^{\hat\P}\Big[|g(X^{t,\xi,\alpha_t}_T,\hat\P^{\c F^{B,\mu,\hat\P}_T}_{X^{t,\xi,\alpha_t}_T})|^2 + \int_t^T |f(s,\hat\P^{\c F^{B,\mu,\hat\P}_s}_{X^{t,\xi,\alpha_t}_s},X^{t,\xi,\alpha_t}_s,I^{t,\alpha_t}_s)|^2 ds\Big] < \infty.
\end{align}
We remark that \cite[Assumption (H1)]{kharroubi_feynmankac_2015} is only needed to show the following estimate
\begin{align}
\sup_n \E^{\hat\P}\Big[\sup_{s\in [t,T]} |Y^{n,t,\xi,\alpha_t}_s|^2\Big] < \infty.
\label{eq theorem 6.2 H1 estimate}
\end{align}
Thus, instead of verifying \cite[Assumption (H1)]{kharroubi_feynmankac_2015}, we will directly show that \eqref{eq theorem 6.2 H1 estimate} holds. For this, we note that following the same arguments as \cite[Lemma 4.9]{fuhrman_randomized_2015}, from the representation \eqref{eq Y n t xi snell envelope formula} in \cref{lemma penalised bsde minimal solution existence} together with \cref{assumptions sde coefficients,assumptions reward functional}, we can show that
\[
\sup_{s\in [t,T]} |Y^{n,t,\xi,\alpha_t}_s|^2 \leq C\Big( 1 + \E^{\hat\P}\Big[ \sup_{s\in [t,T]} |X^{t,\xi,\alpha_t}_s|^2 \Big|\c F^{B,\mu,\hat\P}_T\Big]\Big),\quad\hat\P\text{-a.s.},
\]
for a constant $C$ independent of $n$, which together with \eqref{eq randomised state dynamics basic estimate} implies \eqref{eq theorem 6.2 H1 estimate}.
Thus by \cite[Theorem 2.1]{kharroubi_feynmankac_2015}, the BSDE \eqref{eq constrained bsde value function} has a unique minimal solution
\[
(Y^{t,\xi,\alpha_t},Z^{t,\xi,\alpha_t},U^{t,\xi,\alpha_t},K^{t,\xi,\alpha_t})\in \c S^2_{[0,T]}(\b F^{B,\mu,\hat\P})\times L^2_{W,[0,T]}(\b F^{B,\mu,\hat\P})\times L^2_{\lambda,[0,T]}(\b F^{B,\mu,\hat\P})\times \c K^2_{[0,T]}(\b F^{B,\mu,\hat\P}),
\]
when viewed as BSDE on the whole time interval $[0,T]$, and which satisfies $Y^{n,t,\xi,\alpha_t}_s\uparrow Y^{t,\xi,\alpha_t}_s$, for all $s\in [t,T]$, $\hat\P$-a.s.\@ and in $L^2(ds\otimes\hat\P)$. Therefore, as a last step we need to show that also $K^{t,\xi,\alpha_t}_t = 0$ and thus $K^{t,\xi,\alpha_t}\in \c K^2_{[t,T]}(\b F^{B,\mu,\hat\P})$, which we will postpone to \cref{proof theorem 6.2 part iii} of the proof as it will follow from showing that $Y^{t,\xi,\alpha_t}_t$ is $\hat\P$-a.s.\@ constant.


\item Regarding the representation \eqref{eq recursive snell envelope formula}, we first note that by \cref{lemma penalised bsde minimal solution existence}, we have the formula \eqref{eq Y n t xi snell envelope formula} for $Y^{n,t,\xi,\alpha_t}$ for all $n\in\N$. Thus, using that $\c V_n\subseteq\c V_{n+1}\subseteq\c V$ and $\bigcup_{n\in\N} \c V_n = \c V$, we conclude from $Y^{n,t,\xi,\alpha_t}_s\uparrow Y^{t,\xi,\alpha_t}_s$, for all $s\in [t,T]$, $\hat\P$-a.s., by monotone convergence that \eqref{eq recursive snell envelope formula} holds for $Y^{t,\xi,\alpha_t}$.


\item\label{proof theorem 6.2 part iii} Next we will prove that $Y^{t,\xi,\alpha_t}_t$ is $\hat\P$-a.s.\@ constant, and with it that $K^{t,\xi,\alpha_t}_t = 0$ and hence $K^{t,\xi,\alpha_t}\in\c K^2_{[t,T]}(\b F^{B,\mu,\hat\P})$, which also completes the construction in \cref{proof theorem 6.2 part i}. For this, we will show that $Y^{t,\xi,\alpha_t}_t = V(t,m)$, $\hat\P$-a.s.

First, we note that \eqref{eq recursive snell envelope formula} together with \eqref{eq constrained bsde value function} implies that
\begin{align}
    Y^{t,\xi,\alpha_t}_t
    &= \esssup_{\nu\in\c V} 
    \E^{\hat\P^\nu}\Big[g(\hat\P^{\nu,\c F^{B,\mu,\hat\P^\nu}_T}_{X^{t,\xi,\alpha_t}_T},X^{t,\xi,\alpha_t}_T) + \int_t^T f(s,\hat\P^{\nu,\c F^{B,\mu,\hat\P^\nu}_s}_{X^{t,\xi,\alpha_t}_s},X^{t,\xi,\alpha_t}_s,I^{t,\alpha_t}_s) ds\Big| \c F^{B,\mu,\hat\P^\nu}_t\Big],\ \hat\P\text{-a.s.}
    \label{eq Y t xi reward representation}
\end{align}
Next, we observe that by \cref{proposition decompose filtrations of predictable processes} we can decompose each $\nu\in\c V$ into an $\c F^{B,\mu}_t\otimes Pred(\b F^{B^t,\mu^t})\otimes\c B(\c A_T)$-measurable process $\Upsilon$ with $0<\inf_{[0,T]\times\hat\Omega\times\hat\Omega\times\c A_T} \Upsilon\leq \sup_{[0,T]\times\hat\Omega\times\hat\Omega\times\c A_T} \Upsilon < \infty$ such that $\nu(\hat\omega) = \Upsilon(\hat\omega,\hat\omega)$, for all $\hat\omega\in\hat\Omega$. 
In particular, we have $\Upsilon(\hat\omega,\cdot)\in\c V_t$ for every $\hat\omega\in\hat\Omega$.
Now recalling that $X^{t,\xi,\alpha_t}$ and $I^{t,\alpha_t}$ are 
$\hat\P$-independent of $\c F^{B,\mu,\hat\P}_t$, we can apply the freezing lemma, see e.g.\@ \cite[Lemma 4.1]{baldi_stochastic_2017}, 
to rewrite the conditional expectation in \eqref{eq Y t xi reward representation} as follows,
\begin{align}
    &\E^{\hat\P^\nu}\Big[g(\hat\P^{\nu,\c F^{B,\mu,\hat\P^\nu}_T}_{X^{t,\xi,\alpha_t}_T},X^{t,\xi,\alpha_t}_T) + \int_t^T f(s,\hat\P^{\nu,\c F^{B,\mu,\hat\P^\nu}_s}_{X^{t,\xi,\alpha_t}_s},X^{t,\xi,\alpha_t}_s,I^{t,\alpha_t}_s) ds\Big| \c F^{B,\mu,\hat\P^\nu}_t\Big]\\
    &= \E^{\hat\P}\Big[\frac{L^\nu_T}{L^\nu_t} \Big(g(\hat\P^{\c F^{B,\mu,\hat\P}_T}_{X^{t,\xi,\alpha_t}_T},X^{t,\xi,\alpha_t}_T) + \int_t^T f(s,\hat\P^{\c F^{B,\mu,\hat\P}_s}_{X^{t,\xi,\alpha_t}_s},X^{t,\xi,\alpha_t}_s,I^{t,\alpha_t}_s) ds\Big)\Big| \c F^{B,\mu,\hat\P}_t\Big]\\
    &= \E^{\hat\P}\Big[\frac{L^\upsilon_T}{L^\upsilon_t} \Big(g(\hat\P^{\c F^{B,\mu,\hat\P}_T}_{X^{t,\xi,\alpha_t}_T},X^{t,\xi,\alpha_t}_T) + \int_t^T f(s,\hat\P^{\c F^{B,\mu,\hat\P}_s}_{X^{t,\xi,\alpha_t}_s},X^{t,\xi,\alpha_t}_s,I^{t,\alpha_t}_s) ds\Big)\Big]\bigg|_{\upsilon = \Upsilon(\hat\omega,\cdot)}\\
    &= \E^{\hat\P}\Big[L^\upsilon_T \Big(g(\hat\P^{\c F^{B,\mu,\hat\P}_T}_{X^{t,\xi,\alpha_t}_T},X^{t,\xi,\alpha_t}_T) + \int_t^T f(s,\hat\P^{\c F^{B,\mu,\hat\P}_s}_{X^{t,\xi,\alpha_t}_s},X^{t,\xi,\alpha_t}_s,I^{t,\alpha_t}_s) ds\Big)\Big]\bigg|_{\upsilon = \Upsilon(\hat\omega,\cdot)}\\
    &= J^{\c R}(t,\xi,\alpha_t,\upsilon)\big|_{\upsilon = \Upsilon(\hat\omega,\cdot)},\qquad\hat\P\text{-a.s.},
\end{align}
since for every $\hat\omega\in\hat\Omega$ we have $\upsilon = \Upsilon(\hat\omega,\cdot)\in\c V_t$ and thus $\frac{L^\upsilon_T}{L^\upsilon_t}$ is $\hat\P$-independent of $\c F^{B,\mu,\hat\P}_t$ together with $\E^{\hat\P}[L^{\upsilon}_t] = 1$ and $L^\upsilon_t$ being $\c F^{B,\mu,\hat\P}_t$-measurable. Therefore by using \eqref{eq Y t xi reward representation} together with \cref{proposition equivalence v v_t randomised value function,theorem equivalence original and randomised value function} we obtain one the one hand
\begin{align}
    V(t,m) &= \sup_{\nu\in\c V_t} J^{\c R}(t,\xi,\alpha_t,\nu)\\
    &= \sup_{\nu\in\c V_t} \E^{\hat\P^\nu}\Big[g(\hat\P^{\nu,\c F^{B,\mu,\hat\P^\nu}_T}_{X^{t,\xi,\alpha_t}_T},X^{t,\xi,\alpha_t}_T) + \int_t^T f(s,\hat\P^{\nu,\c F^{B,\mu,\hat\P^\nu}_s}_{X^{t,\xi,\alpha_t}_s},X^{t,\xi,\alpha_t}_s,I^{t,\alpha_t}_s) ds\Big| \c F^{B,\mu,\hat\P^\nu}_t\Big]
    \leq 
    Y^{t,\xi,\alpha_t}_t,
\end{align}
and on the other hand also
\begin{align}
    Y^{t,\xi,\alpha_t}_t &= \sup_{\nu\in\c V} J^{\c R}(t,\xi,\alpha_t,\upsilon)\big|_{\upsilon = \Upsilon(\hat\omega,\cdot)} \leq \sup_{\upsilon\in\c V_t} J^{\c R}(t,\xi,\alpha_t,\upsilon) = V(t,m).
\end{align}
This shows that $Y^{t,\xi,\alpha_t}_t = V(t,m)$, $\hat\P$-a.s., and in particular that $Y^{t,\xi,\alpha_t}_t$ is $\hat\P$-a.s.\@ constant. Finally, we note that this also implies $K^{t,\xi,\alpha_t}_t = 0$, which shows that $K^{t,\xi,\alpha_t}\in\c K^2_{[t,T]}(\b F^{B,\mu,\hat\P})$.


\item
Finally, let us prove the representation \eqref{eq recursive snell envelope formula at s=t}.
From the BSDE \eqref{eq constrained bsde value function}, by taking the conditional expectation under $\hat\P^\nu$, we obtain for all $r\in [t,T]$ and $\nu\in\c V$,
\[
Y^{t,\xi,\alpha_t}_t \geq \E^{\hat\P^\nu}[Y^{t,\xi,\alpha_t}_t] \geq \E^{\hat\P^\nu}\Big[Y^{t,\xi,\alpha_t}_r + \int_t^r f(u,\hat\P^{\nu,\c F^{B,\mu,\hat\P^\nu}_u}_{X^{t,\xi,\alpha_t}_u},X^{t,\xi,\alpha_t}_u,I^{t,\alpha_t}_u) du\Big],
\]
since $U^{t,\xi,\alpha_t}$ is nonnegative and $K^{t,\xi,\alpha_t}$ is nondecreasing, and $Y^{t,\xi,\alpha_t}_t$ is $\hat\P$-a.s.\@ constant. Taking the supremum over all $\nu\in\c V$ then shows that
\[
Y^{t,\xi,\alpha_t}_t \geq \sup_{\nu\in\c V} \E^{\hat\P^\nu}\Big[Y^{t,\xi,\alpha_t}_r + \int_t^r f(u,\hat\P^{\nu,\c F^{B,\mu,\hat\P^\nu}_u}_{X^{t,\xi,\alpha_t}_u},X^{t,\xi,\alpha_t}_u,I^{t,\alpha_t}_u) du\Big].
\]

For the other direction, we recall that by \cref{lemma penalised bsde minimal solution existence}, we have the upper estimate \eqref{eq Y n t xi snell envelope upper bound at s=t} for $Y^{n,t,\xi,\alpha_t}$, for all $n\in\N$. Thus, using that $\c V_n\subseteq\c V_{n+1}\subseteq\c V$ and $\bigcup_{n\in\N} \c V_n = \c V$, we deduce from $Y^{n,t,\xi,\alpha_t}_s\uparrow Y^{t,\xi,\alpha_t}_s$, for all $s\in [t,T]$, $\hat\P$-a.s., by monotone convergence that, for all $r\in [t,T]$,
\[
Y^{t,\xi,\alpha_t}_t = \E^{\hat\P}[Y^{t,\xi,\alpha_t}_t]
\leq \sup_{\nu\in\c V} \E^{\hat\P^\nu}\Big[Y^{t,\xi,\alpha_t}_r + \int_t^r f(u,\hat\P^{\nu,\c F^{B,\mu,\hat\P^\nu}_u}_{X^{t,\xi,\alpha_t}_u},X^{t,\xi,\alpha_t}_u,I^{t,\alpha_t}_u) du\Big].
\]
\end{enumerate}
\end{proof}


\let\c\oldc 

\bibliographystyle{plain}
\bibliography{bibliography}

\end{document}